\documentclass[11pt]{amsart}
\usepackage[dvips]{graphicx,color}
\usepackage{amssymb,amscd,latexsym,epsfig,xypic,hyperref}
\usepackage[mathscr]{euscript}
\xyoption{curve}

\DeclareFontFamily{OT1}{rsfs}{}
\DeclareFontShape{OT1}{rsfs}{n}{it}{<-> rsfs10}{}
\DeclareMathAlphabet{\curly}{OT1}{rsfs}{n}{it}

\makeatletter
\newcommand{\eqnum}{\refstepcounter{equation}\textup{\tagform@{\theequation}}}
\makeatother

\renewcommand\;{\hspace{.6pt}}
\renewcommand\P[1]{{\mathbb P}^{\;#1}}
\newcommand\PP{\mathbb P}
\newcommand\LL{\mathbb L}
\newcommand\TT{\mathbb T}
\newcommand\C{\mathbb C}
\newcommand\Q{\mathbb Q}

\newcommand\Z{\mathbb Z}

\renewcommand\t{\mathfrak t}
\newcommand\g{\mathfrak g}

\newcommand\cA{\mathcal A}

\newcommand\cE{\mathcal E}
\newcommand\cF{\mathcal F}

\newcommand\calH{\mathcal H}

\renewcommand\cL{\mathcal L}
\newcommand\cM{\mathcal M}
\newcommand\cN{\mathcal N}

\newcommand\cO{\mathcal O}

\newcommand\cS{\mathcal S}

\newcommand\cZ{\mathcal Z}
\newcommand\udot{^{\bullet}}
\newcommand\E{\mathsf E}
\newcommand\EE{\curly E}
\newcommand\sH{\mathsf H}

\makeatletter
\newcommand{\so}{\ \ext@arrow 0359\Rightarrowfill@{}{\hspace{3mm}}\ }
\makeatother
\newcommand{\rt}[1]{\xrightarrow{\ #1\ }}
\newcommand\To{\longrightarrow}

\newcommand\into{\hookrightarrow}
\newcommand\INTO{\ \ar@{^(->}[r]<-.2ex>}
\newcommand{\Into}{\ensuremath{\lhook\joinrel\relbar\joinrel\rightarrow}}
\newcommand\Mapsto{\ensuremath{\shortmid\joinrel\relbar\joinrel\rightarrow}}

\renewcommand\_{^{}_}

\newcommand{\mat}[4]{\left(\begin{array}{cc} \!\!#1 & #2\!\! \\ \!\!#3 &
#4\!\!\end{array}\right)}

\newfont{\bigtimesfont}{cmsy10 scaled \magstep5}
\newcommand{\bigtimes}{\mathop{\lower0.9ex\hbox{\bigtimesfont\symbol2}}}
\newcommand\dbar{\overline\partial}

\newcommand\VW{\mathsf{VW}}
\newcommand\vw{\mathsf{vw}}
\newcommand\UVW{\widetilde{\mathsf{VW}}}

\newcommand\AJ{\operatorname{AJ}}
\newcommand\At{\operatorname{At}}

\newcommand\rk{\operatorname{rank}}
\newcommand\vir{\operatorname{vir}}
\newcommand\vd{\operatorname{vd}}
\newcommand\tr{\operatorname{tr}}

\newcommand\im{\operatorname{im}}
\newcommand\id{\operatorname{id}}
\newcommand\ev{\operatorname{ev}}

\newcommand\Hom{\operatorname{Hom}}
\newcommand\Sym{\operatorname{Sym}}     
\renewcommand\hom{\curly H\!om}

\newcommand\Ext{\operatorname{Ext}}
\newcommand\ext{\curly Ext}
\newcommand\Aut{\operatorname{Aut}}
\newcommand\Pic{\operatorname{Pic}}
\newcommand\Jac{\operatorname{Jac}}

\newcommand\Spec{\operatorname{Spec}\,}
\newcommand\Hilb{\operatorname{Hilb}}
\newcommand\Cone{\operatorname{Cone}}
\newcommand\Bl{\operatorname{Bl}}

\newcommand\red{\operatorname{red}}

\newcommand\arXiv[1]{\href{http://arxiv.org/abs/#1}{arXiv:#1}}
\newcommand\mathAG[1]{\href{http://arxiv.org/abs/math/#1}{math.AG/#1}}
\newcommand\hepth[1]{\href{http://arxiv.org/abs/hep-th/#1}{hep-th/#1}}

\DeclareRobustCommand{\SkipTocEntry}[3]{}
\makeatletter
\newcommand\@dotsep{4.5}
\def\@tocline#1#2#3#4#5#6#7{\relax
  \ifnum #1>\c@tocdepth % then omit
  \else
    \par \addpenalty\@secpenalty\addvspace{#2}%
    \begingroup \hyphenpenalty\@M
    \@ifempty{#4}{%
      \@tempdima\csname r@tocindent\number#1\endcsname\relax
    }{%
      \@tempdima#4\relax
    }%
    \parindent\z@ \leftskip#3\relax \advance\leftskip\@tempdima\relax
    \rightskip\@pnumwidth plus1em \parfillskip-\@pnumwidth
    #5\leavevmode #6\relax
    \leaders\hbox{$\m@th
      \mkern \@dotsep mu\hbox{.}\mkern \@dotsep mu$}\hfill
    \hbox to\@pnumwidth{\@tocpagenum{#7}}\par
    \nobreak
    \endgroup
  \fi}
\makeatother

\newcommand\beq[1]{\begin{equation}\label{#1}}
\newcommand\eeq{\end{equation}}
\newcommand\beqa{\begin{eqnarray*}}
\newcommand\eeqa{\end{eqnarray*}}

\makeatletter \@addtoreset{equation}{section} \makeatother
\renewcommand{\theequation}{\thesection.\arabic{equation}}
\newtheorem{defn}[equation]{Definition}
\newtheorem{pdefn}[equation]{Preliminary definition}
\newtheorem{thm}[equation]{Theorem}
\newtheorem{lem}[equation]{Lemma}
\newtheorem{cor}[equation]{Corollary}
\newtheorem{prop}[equation]{Proposition}
\newtheorem{rmk}[equation]{Remark}

\title{Vafa-Witten invariants for projective surfaces I: stable case}
\author[Y. Tanaka and R. P. Thomas]{Yuuji Tanaka and Richard P. Thomas}

\begin{document}
\maketitle

\begin{abstract} \noindent
On a polarised surface, solutions of the Vafa-Witten equations correspond to certain polystable Higgs pairs. When stability and semistability coincide, the moduli space admits a symmetric obstruction theory and a $\C^*$ action with compact fixed locus. Applying \emph{virtual localisation} we define invariants constant under deformations.

When the vanishing theorem of Vafa-Witten holds, the result is the (signed) Euler characteristic of the moduli space of instantons. In general there are other, \emph{rational}, contributions. Calculations of these on surfaces with positive canonical bundle recover the first terms of modular forms predicted by Vafa and Witten.
\end{abstract}
\renewcommand\contentsname{\vspace{-9mm}}
\tableofcontents \vspace{-1cm}

%%%%%%%%%%%%%%%%%%%%%%%%%%%%%%%%%%%%%%%%%%%%%%%%%%%%%%%%%%%%%%%%%%%%%%%%%%%

\section{Introduction}

Fix a Riemannian 4-manifold $S$; from the next Section onwards it will be a complex projective surface (with a K\"ahler metric). Fix a compact Lie group $G$ and principal $G$-bundle $P\to S$. We let $\cA_P$ denote the set of all $G$-connections on $P$, and $\Omega^i(\g_P)$ the $\C^\infty$ $i$-forms on $S$ with values in the adjoint bundle of $P$. Finally $\Omega^+\subset\Omega^2$ denotes the self-dual 2-forms with respect to the Riemannian metric.

\subsection{The Vafa-Witten equations}
The Vafa-Witten equations \cite{VW} are
\begin{align}\nonumber
\cA_P\times\Omega^+(\g_P)\times\Omega^0(\g_P)\,\To
&\ \Omega^+(\g_P)\times\Omega^1(\g_P) \\ \label{VW}
(d_A,B,\Gamma)\,\Mapsto&\ \big(F_A^++[B.B]+[B,\Gamma],\,d_A\Gamma+d_A^*B\big).
\end{align}
Here $[B.B]$ is defined by Lie bracket on $\g_P$ and the         contraction $(T^*)^{\otimes2}\otimes(T^*)^{\otimes2}\to (T^*)^{\otimes2}$ given by the (inverse of the) metric on the second and third $T^*$ factors, followed by antisymmetrisation.

\subsection{The equations on a K\"ahler surface}
When $(S,\omega)$ is a K\"ahler surface, we can rewrite $B$ and $\Gamma$ in terms of a $\g_P$-valued $(2,0)$-form $\phi\in\Omega^{2,0}(\g_P\otimes\C)$ and a $\g_P$-valued multiple of $\omega$, resulting in
\begin{eqnarray}
F_A^{0,2} &=& 0, \nonumber \\
F_A^{1,1}\wedge\omega+[\phi,\overline\phi\,] &=&
c\cdot\id_E\omega^2, \label{VWk} \\
\dbar_A\phi &=& 0. \nonumber
\end{eqnarray}
Here $c$ is a topological constant. From now on we will restrict attention to the cases $G=U(r)$ or $G=SU(r)$, so that $P$ is the frame bundle of a hermitian vector bundle $E$. Then $c=\pi i\int_Sc_1(E)\wedge\omega\big/r\!\int_S\omega^2$.
 
Thus, by the first equation, $\dbar_A$ defines an integrable holomorphic structure on $E$ with respect to which $\phi$ is holomorphic. Since the second equation is a moment map for the gauge group action, one can expect an infinite dimensional Kempf-Ness theorem that solutions (modulo $U(r)$ or $SU(r)$ gauge transformations) are equivalent to \emph{stable} holomorphic pairs $(E,\phi)$ (modulo $GL(r,\C)$ or $SL(r,\C)$ gauge transformations) for an appropriate notion of stability. 

\subsection{Hitchin-Kobayashi correspondence}
Such a Donaldson-Uhlenbeck-Yau-type result has been proved for K\"ahler surfaces $(S,\omega)$ in \cite{AG}, and for polarised surfaces $(S,\cO_S(1))$ with integral K\"ahler form $h=c_1(\cO_S(1))$ \cite{Ta1}. We focus on the latter case. Solutions of the $U(r)$ Vafa-Witten equations are equivalent to \emph{slope polystable Higgs pairs}
\beq{Tstar}
(E,\phi), \qquad\phi\in\Hom(E,E\otimes K_S),
\eeq
while for $SU(r)$ we fix $\det E=\cO_S$ and take $\phi$ to be trace-free. The stability condition is that
\beq{stabb}
\frac{c_1(F)\cdot h}{\rk(F)}\ <\ \frac{c_1(E)\cdot h}{\rk(E)}
\eeq
for all proper $\phi$-invariant coherent subsheaves $F\subset E$. Replacing $<$ by $\le$ defines semistability, while polystable pairs are direct sums of stable Higgs pairs of the same slope. The closely related but slightly more refined notion of Gieseker stability is described in Section \ref{Gss}.

Thus the moduli space of solutions of the Vafa-Witten equations has an obvious partial compactification given by taking  Gieseker semistable (hence torsion-free) rank $r>0$ Higgs \emph{sheaves}. In this paper ``stability" always refers to Gieseker stability, and we always assume the Chern classes of $E$ are chosen so that \emph{semistability and stability coincide}. In paper II \cite{TT2} we tackle the semistable case.

\subsection{Spectral construction}
Finally, we can turn such pairs $(E,\phi)$ into compactly supported \emph{stable torsion sheaves $\cE_\phi$ on the Calabi-Yau 3-fold $X$}, where $X$ is the total space of the canonical line bundle $K_S$ of $S$.

Roughly speaking, over each point $x\in S$, we replace $(E_x,\phi_x)$ by the eigenspaces of $\phi_x\in\Hom\big(E_x, E_x\otimes(K_S)_x\big)$ supported on their respective eigenvalues in $(K_S)_x$. The result is an equivalence of categories
$$
{\rm Higgs}_{K_S}(S)\cong\,{\rm Coh}_c(X);
$$
see Section \ref{ss} for full details.

\subsection{Localisation and a $U(r)$ Vafa-Witten invariant}
After fixing Chern classes $(r,c_1,c_2)$ on $S$ for which semistability\,=\,stability, we let $\cN$ denote the moduli space of stable Higgs sheaves $(E,\phi)$ on $S$, or, equivalently, compactly supported stable torsion sheaves $\cE_\phi$ on $X$.

As a moduli space of sheaves on Calabi-Yau 3-fold, the results of \cite{HT, Th} give a perfect obstruction theory on $\cN$. Since $\cN$ is noncompact, the resulting virtual cycle is uninteresting. However, it also carries a $\C^*$ action, given by scaling the Higgs field $\phi$, or, equivalently, the fibres of $X=K_S\to S$. The $\C^*$-fixed locus is compact, so we may apply virtual localisation \cite{GP} to define a numerical invariant \emph{counting} the sheaves $\cE$. This is our \emph{preliminary} $U(r)$ Vafa-Witten invariant
\beq{uvw}
\UVW_{r,c_1,c_2}(S)\ =\ \int_{\big[\cN_{r,c_1,c_2}\big]^{\vir}}\frac1{e(N^{\vir})}\ \in\ \Q,
\eeq
described in Section \ref{prelim}. It is nothing but a local surface DT invariant of the Calabi-Yau 3-fold $X=K_S$.

\subsection{\for{toc}{Virtual technicalities and an $SU(r)$ VW invariant}\except{toc}{Virtual technicalities and an $SU(r)$ Vafa-Witten invariant}}\ \\
\emph{However}, as discussed in Remark \ref{nogood}, $\UVW$ \eqref{uvw} is \emph{zero} unless both $H^{0,1}(S)=0=H^{0,2}(S)$. Most surfaces satisfying this condition satisfy a vanishing theorem, making the Vafa-Witten invariant just the (signed) topological Euler characteristic of the moduli space of stable sheaves on $S$. We would like to be able to handle more interesting cases.

The right fix is to consider the Vafa-Witten equations for the gauge group $SU(r)$ instead of $U(r)$.
Thus we restrict to Higgs pairs $(E,\phi)$ where $E$ has \emph{fixed determinant}\footnote{$SU(r)$ is a bit of a misnomer since for this definition we allow ourselves to fix $\det E=L$ for some \emph{nontrivial} line bundle $L$ on $S$ (e.g. one whose degree is coprime to $\rk(E)$ so that semistability implies stability.) In paper II \cite{TT2} we allow fixed \emph{trivial} $\det E=\cO_S$.} and $\phi$ is \emph{trace-free}. Equivalently, we work with torsion sheaves $\cE$ on $X$ whose ``\emph{centre of mass}" on each fibre (i.e. the sum of its points of support in $K_S$, weighted by multiplicity) is \emph{zero}, and whose pushdown to $S$ has fixed determinant.

Producing the right deformation theory for such things turned out to be unexpectedly complicated. At a single point $(E,\phi)\in\cN$, the complication is roughly the following.

A natural resolution of the torsion sheaf $\cE_\phi$ on $X\rt\pi S$ in terms of $\pi^*E$ and $\pi^*\phi$ \eqref{resolution} gives a long exact sequence
\beq{LEs}
\cdots\To\Hom(E,E\otimes K_S)\To\Ext^1(\cE_\phi,\cE_\phi)\To\Ext^1(E,E)\To\cdots
\eeq
relating the automorphisms, deformations and obstructions of $\cE_\phi$ to those of $(E,\phi)$. The third arrow is easily seen to take deformations of $\cE_\phi$ to the corresponding deformations of its pushdown $E=\pi_*\;\cE_\phi$. What was surprisingly hard (for us) to show is that the second arrow is the one would expect --- namely the derivative of the map from (an open set in) $\Hom(E,E\otimes K_S)$ to $\cN$ that takes $\phi$ to $(E,\phi)$.

Dealing with this issue (in full generality, over an arbitrary family to all orders at the level of perfect obstruction theories) is what takes up all of Section \ref{more}.

The result is that we can express our perfect obstruction theory for $\cN$ in terms of Higgs data on $S$ instead of sheaves on $X$. That done, it allows us to easily fix $\det E\cong L$ and $\tr\phi=0$ to deduce a \emph{symmetric perfect obstruction theory} for the moduli space\footnote{In Theorem \ref{ffinal} we extend this result from $K_S$-Higgs pairs to $\cL$-Higgs pairs. Equivalently we work with torsion sheaves on the total space of $\cL\to S$ in place of $K_S\to S$. The resulting 2-term perfect obstruction theory is not symmetric when $\cL\ne K_S$, however.}
$$
\cN_L^\perp\ =\ \big\{(E,\phi)\colon \det E\cong L,\ \tr\phi=0\big\}.
$$
In fact, using the trace and identity maps on $S$ on the first and last terms of \eqref{LEs} produces a splitting
\beq{parp}
\Ext_X^*(\cE_\phi,\cE_\phi)\ \cong\ H^{*-1}(K_S)\oplus H^*(\cO_S)\,\oplus\,\Ext^*_X(\cE_\phi,\cE_\phi)\_\perp\,.
\eeq
This expresses the deformation-obstruction theory of $\cE_\phi\in\cN$ on the left hand side in terms of, respectively: the deformations and obstructions of $\tr\phi$, the  deformations and obstructions of $\det E$, and the  deformations and obstructions of $\cE_\phi\in\cN^\perp_L$\;; see Theorem \ref{final}.

Thus we get a better definition of the Vafa-Witten invariant by localising this new virtual cycle,
\beq{definitive}
\VW_{r,c_1,c_2}\ :=\ \int_{\big[(\cN^\perp_{r,L,c_2})^{\C^*}\big]^{\vir\ }}
\frac1{e(N^{\vir})}\ \in\ \Q.
\eeq
If the Higgs field vanishes $\phi=0$ for all points $(E,\phi)\in(\cN^\perp_L)^{\C^*}$ then \eqref{definitive} is actually an integer (the virtual signed Euler characteristic of Section \ref{fixsec}), but in general we work in the localised equivariant cohomology $H^*_{\C^*}(B\C^*,\Q)[t^{-1}]\cong\Q[t,t^{-1}]$ and get rational numbers. (The invariant is a constant, rather than a more general Laurent polynomial in the equivariant parameter $t$, because the virtual dimension of the problem is zero.)

$\VW_{r,c_1,c_2}$ \eqref{definitive} is invariant under deformations of $L$, and deformations of $S$ which keep the class $c_1:=c_1(L)\in H^2(S)$ of Hodge type (1,1). Physics seems to predict that it \emph{ought} to also be invariant under deformations of the polarisation $\cO_S(1)$ (or more generally the stability condition defining $\cN$) --- i.e. the invariants' wall crossing should be trivial. We intend to return to this point later using an extension of the wall crossing formula \cite{KL2}. (Whereas the wall crossing formulae of Joyce-Song \cite{JS} and Kontsevich-Soibelman \cite{KS} use weighted Euler characteristics in a crucial way, Kiem-Li use virtual localisation.) When $L=\cO_S$ the closely related invariant $\vw$ defined by Behrend localisation in paper II \cite{TT2} \emph{does} have trivial wall crossing.
\medskip

The integral \eqref{definitive} is over the different components of the fixed locus of the $\C^*$ action scaling the Higgs field. They come in two flavours, which we name according to the conventions of \cite{DPS, GK}.
\begin{enumerate}
\item[\eqnum\label{cpt1}] The ``\emph{instanton branch}" $\phi=0$.
Here we recover $\cM_{\mathrm{asd}}$ or the moduli space $\cM_L$ of stable sheaves of fixed determinant $L$ on $S$. By Proposition \ref{vaneul} the contribution of this locus to our invariant is the virtual signed Euler characteristic of $\cM_{\mathrm{asd}}$ studied in \cite{JT, GK}, generalising \eqref{Euler}. It is an integer.
\item[\eqnum\label{cpt2}] The ``\emph{monopole branch}" $\phi\ne0$. We call the union of these components of the fixed locus $\cM_2$. They correspond to $\C^*$-fixed sheaves $\cE$ supported on scheme theoretic thickenings of the zero section $S\subset X$. When they have rank 1 on their support, they can be described in terms of nested Hilbert schemes; see \cite{GSY1, GSY2} and Section \ref{quintic}. In general they correspond to flags of sheaves on $S$ as studied in the work of Negut (see \cite{Ne}, for instance). They contribute new rational numbers to the Vafa-Witten invariants.
\end{enumerate}

\subsection{Digression: derived algebraic geometry}
The perfect obstruction theory for Higgs sheaves is of independent interest, but adds a great deal of complexity and length to the paper. It also requires us to work with Illusie's full cotangent complex\footnote{Because the exact triangle relating the cotangent complexes of $\cN,\ \curly M$ and $\cN/\curly M$ need not be exact after truncating.} \cite{Ill} instead of the truncated cotangent complex that normally suffices \cite{HT}.

Most readers should therefore ignore most of Sections \ref{potsec}, \ref{prelim}, \ref{more} and skip straight to Section \ref{SU}, taking on trust that the standard perfect obstruction theory \cite{HT, Th} can be refined to the fixed trace and determinant case. But readers who are both derived and stacky will notice that a much quicker solution is to use the theory of derived algebraic stacks and the results of \cite{STV,TVa}.

That is, one shows that the stack $\curly M_L$ of rank $r$ torsion-free coherent sheaves $E$ on $S$ with fixed-determinant $\det E\cong L$ has a natural derived structure, inducing one on its $(-1)$-shifted cotangent bundle\footnote{At the level of points, this fibres over $\curly M_L$ with fibre over $E$ the \emph{dual} of the obstruction space $\Ext^2(E,E)\_0$ at that point. By the duality $\Ext^2(E,E)_0^*\cong\Hom(E,E\otimes K_S)\_0$ we indeed recover \eqref{Tstar}.} $T^*[-1]\curly M_L$. This is the moduli stack of all Higgs pairs. The open substack of \emph{stable} Higgs pairs is then the product of a derived scheme $\cN_L^\perp$ and $T^*[-1]B\C^*$:
$$
\cN_L^\perp\times T^*[-1]B\C^*\ \subset\ T^*[-1]\curly M_L.
$$
Rigidifying (removing the $T^*[-1]B\C^*$ factor) gives a natural derived structure on $\cN_L^\perp$ which is \emph{quasi-smooth}. It therefore gives rise to a perfect obstruction theory on the underlying scheme. (This obstruction theory is also \emph{symmetric} \cite{Ca}.) We are the wrong authors to use this technology honestly or competently, however, so we employ more classical (but lengthy!) techniques familiar to virtual cyclists like the second author.

\subsection{Cotangent field theories}
There is a symmetry between the domain and target of the VW equations \eqref{VW}. If we add a local Coulomb gauge fixing equation $d_{A_0}^*(A-A_0)\in\Omega^0(\g_P)$ instead of dividing by gauge, and notice that $\cA_P$ is an affine space modelled on $\Omega^1(\g_P)$, we see both sides contain the same pieces.

Roughly speaking, the equations come from the anti-self-dual (asd) equations $F_A^+=0$ by passing to their ``$(-1)$-shifted cotangent bundle", at least to first order in the new (cotangent) variables $B,\Gamma$.

The local model is the following. Start with a section $s$ of a vector bundle $E$ over an ambient space $A$ cutting out a moduli space of solutions $\cM$:
$$\xymatrix@R=16pt@C=0pt{
& E \ar[d] \\ s^{-1}(0)\ =\ \cM\ \subset & A. \ar@/_/[u]_s}
$$
The setting could be real or holomorphic, finite or infinite dimensional. In the latter case it is called a \emph{balanced, co-} or \emph{cotangent} field theory in different references; see for example \cite{CMR, DiM, Co, JT} and \cite[Section 2]{VW}. For the real VW equations\footnote{In fact their linearisation about $B=0=\Gamma$.} we start with the asd equations, so the section is $F_A^+$ in the infinite dimensional bundle $\Omega^+(\g_P)$ over the infinite dimensional ambient space $\cA_P$, plus the Coulomb gauge fixing equation (or we replace $\cA_P$ by its quotient by gauge). In Section \ref{morelin} we will work with a \emph{finite dimensional holomorphic} version of this model, but in that case (we will see below) we will have to take $A$ to be a stack rather than a space. \medskip

Next pass to a bigger ambient space: the total space of $E^*\to A$ (in the VW setting this adds the variables $B$ and $\Gamma$). By pairing the section $s$ of $E$ with points of $E^*$ it defines a function $\tilde s$ on this ambient space $E^*$. Its gradient $d\tilde s$ is a section of the (co)tangent bundle of $E^*$.
$$\xymatrix@R=16pt@C=0pt{
& T^*_{E^*} \ar[d] \\ (d\tilde s)^{-1}(0)\ \subset & E^*. \ar@/_/[u]_{d\tilde s}}
$$
The zeros of $d\tilde s$ include the original zeros $s=0$ (inside the zero section $A$ of $E^*$ --- i.e. it includes the asd moduli space where $F_A^+=0=B=\Gamma$), but in general there are more, giving the (dual) obstruction bundle on the asd moduli space (those $(B,\Gamma)$ perpendicular to the image of the linearised asd equation). 

In ``good" cases\footnote{There are different situations in which one can prove such a vanishing theorem, but they all involve $S$ having curvature which is positive in some sense.} \cite[pp 23--25]{VW} the only solutions have $F_A^+=0=B=\Gamma$. In this case the original asd moduli space $\cM_{\mathrm{asd}}$ is smooth with zero obstruction bundle. Thought of as the moduli space of solutions to the Vafa-Witten equations modulo gauge, however, it has obstruction bundle 
the (co)tangent bundle of $\cM_{\mathrm{asd}}$, and so a natural associated integer invariant given by its Euler class
\beq{Euler}
\pm e(\cM_{\mathrm{asd}}).
\eeq
When no vanishing result holds, it seems no one has yet managed to make a rigorous definition of the VW invariant on a general 4-manifold.

\subsection{The complex case is more linear}\label{morelin}
Vafa and Witten's equation is a nonlinear version of the above construction including also the quadratic term $[B.B]+[B,\Gamma]$. But in the projective surface case, the Hitchin-Kobayashi correspondence \emph{removes this nonlinearity at the expense of enlarging $\cM_{\mathrm{asd}}$.}

That is, if we let $\curly M$ denote the moduli \emph{stack} of \emph{all} bundles (or torsion-free sheaves) on $S$ --- not just those which are stable, which form $\cM_{\mathrm{asd}}$ --- then the stable Higgs sheaves form a Zariski open in the $(-1)$-shifted cotangent bundle
$$
\cN\ \subset\ T^*[-1]\curly M.
$$
As a space (ignoring its derived structure or obstruction theory) $T^*[-1]\curly M=\Spec\Sym\udot\mathrm{Ob}$, where Ob is the obstruction sheaf of $\curly M$, so it can be thought of as the total space of the dual obstruction bundle over $\curly M$.  Since the fibre of Ob over $E\in\curly M$ is $\Ext^2(E,E)$, Serre dual to $\Hom(E,E\otimes K_S)$, we see $T^*[-1]\curly M\to\curly M$ is a fibration by Higgs fields $\phi\in\Hom(E,E\otimes K_S)$.

Passing to the open locus which satisfies the stability condition \eqref{stabb} the stabiliser groups drop to $\C^*$. Rigidifying (removing them, and their $(-1)$-shifted cotangents in degree $-2$) gives the scheme $\cN$ with a quasi-smooth derived structure. This induces the symmetric obstruction theory we used to define our invariant. Similarly our $SU(r)$ moduli space is an open set
$$
\cN^\perp_L\ \subset\ T^*[-1]\curly M_L,
$$
inducing the correct symmetric obstruction theory, again after rigidifying.

This enlargement of $\cM_{\mathrm{asd}}$ to $\curly M_L$ is what results in there now being \emph{two} types of fixed locus for the additional $\C^*$ action that scales the Higgs field $\phi$. The first \eqref{cpt1} is $\cM_{\mathrm{asd}}$ but the second --- the monopole branch $\cM_2$ \eqref{cpt2} with $\phi\ne0$ --- need not lie in $T^*[-1]\cM_{\mathrm{asd}}$.

\subsection{Other localisations} There are three other natural ways to localise the virtual cycle of $\cN$ to $\cN^{\C^*}$, thus giving competing definitions of the VW invariant.

Let $\cN$ be a scheme with a symmetric perfect obstruction theory. Behrend \cite{Be} defines a constructible function
$$
\chi^B\colon\cN\To\Z
$$
such that if $\cN$ is compact then the degree of its (zero dimensional) virtual cycle equals the Euler characteristic of $\cN$ weighted by $\chi^B$,
$$
\int_{[\cN]^{\vir}}1\ =\ e\big(\cN,\chi^B\big)\ :=\ 
\sum_{i\in\Z}i\cdot e\big((\chi^B)^{-1}(i)\big).
$$
Since our moduli spaces of Higgs pairs $\cN$ and $\cN_L^\perp$ are noncompact, we instead take the right hand side as a definition. Of course $\chi^B$ is $\C^*$-invariant and the Euler characteristic of any non-fixed orbit is 0, so only the fixed points contribute. This gives the localisation
\beq{kaione}
e(\cN,\chi^B)\ =\ e\Big(\cN^{\C^*},\chi^B\big|_{\cN^{\C^*}}\Big).
\eeq
When $h^{0,1}(S)>0$ this vanishes (since $e(\Jac(S))=0$) so we fix determinant and trace and localise on $\cN_L^\perp$ instead. Thus we consider the Kai-localised $SU(r)$ Vafa-Witten invariant given by\footnote{At least when there are no strictly semistables. More generally we use the Joyce-Song formalism and consider invariants in $\Q$ counting semistables \cite{TT2}.}
\beq{kaidef}
\vw_{r,c_1,c_2}\ :=\ e\Big((\cN_L^\perp)^{\C^*},\chi^B\big|_{(\cN_L^\perp)^{\C^*}}\Big)\ \in\ \Z.
\eeq
When $h^{0,1}(S)=0$ we recover \eqref{kaione}.

Similarly one can also use Kiem-Li's cosection localisation \cite{KL1}. Differentiating the $\C^*$ action on $\cN$ and $\cN_L^\perp$ defines vector fields thereon. Since these have symmetric obstruction theories, their obstruction sheaves are their cotangent sheaves. Pairing with the vector field gives \emph{cosections}
$$
\mathrm{Ob}_{\cN}\To\cO_{\cN} \qquad\mathrm{and}\qquad
\mathrm{Ob}_{\cN^\perp_L}\To\cO_{\cN^\perp_L}.
$$
Their zero loci are the $\C^*$-fixed loci, so Kiem-Li's construction gives localised cycles
$$
\big[\cN\big]^{\mathrm{loc}}\,\in\,A_0\big(\cN^{\C^*}\big)
\qquad\mathrm{and}\qquad
\big[\cN^\perp_L\big]^{\mathrm{loc}}\,\in\,A_0\big((\cN^\perp_L)^{\C^*}\big).
$$
When $h^{0,1}(S)=0$ they coincide, otherwise the first vanishes. This suggests a third VW invariant
$$
\vw'_{r,c_1,c_2}\ :=\ \int_{\big[\cN^\perp_L\big]^{\mathrm{loc}}}1.
$$
Thirdly, we could take the signed topological Euler characteristic of $(\cN^\perp_L)^{\C^*}$. In fact these three rival definitions are closely related. By \cite{Ji, JT} the first two are equal,
$$
\vw\ =\ \vw',
$$
but are different in general from the invariant $\VW$ defined by  virtual localisation. Furthermore, when a ``vanishing theorem" holds so that all $\C^*$-fixed stable Higgs pairs have $\phi=0$ (equivalently $\C^*$-fixed stable torsion sheaves on $X$ are all pushed forward from $S$) they also equal the third definition by ``dimension reduction" \cite{BBS, Da}, \cite[Section 5]{JT}:
$$
\vw\ =\ \vw'\ =\ (-1)^{\vd}e\Big((\cN^\perp_L)^{\C^*}\Big).
$$
While $\vw$ is not in general deformation invariant, it does have advantages over $\VW$: it is often more easily computable, for instance by cut-and-paste methods, and it admits a natural refinement and categorification (as expected in physics). While we expect that $\vw=\VW$ for surfaces with $K_S\le0$ --- we prove this for $\deg K_S<0$ in \cite{TT2} and it is proved for $S$ a K3 surface in \cite{MT} --- in general they differ. In particular when stability\,=\,semistability $\vw$ is an integer but $\VW$ is only rational in general.

So we need to chose between the two invariants somehow. We do this by seeing how they stand up to the modularity predictions of physics.

\subsection{Calculations and modularity} \label{discuss} Vafa-Witten predicted that for fixed rank $r$ and determinant $L$, S-duality should force the generating function
$$
Z_r(S,L)\ :=\ q^{-s}\sum_{n\in\Z}\,\VW_{r,c_1(L),n}(S)q^n
$$
to be, roughly speaking, a weight $w/2=-e(S)/2$ modular form, for some appropriate shift $s$. (The precise functional equations are described in \cite[Equation 5.39]{VW} in rank 2, and \cite[Equation 5.22]{LL} in prime rank.)
% In fact their examples indicated that one should take $s=e(S)/12$, and that the result is a modular form of weight $w/2=-e(S)/2$.

There have been many compelling calculations confirming these predictions on the instanton branch $\cM_{\rm asd}$ \eqref{cpt1}, usually in the presence of a vanishing theorem (for instance when $K_S\le0$) ensuring that it is both smooth and the monopole branch $\cM_2$ \eqref{cpt2} is empty. Recent work of G\"ottsche and Kool \cite{GK} extends these calculations to the much more difficult setting of general type surfaces $S$.

%They also take $L=\cO_S$, which can cause problems for our invariant \ref{definitive} defined by virtual localisation -- it is well defined when $r$ and $n$ are coprime, but need not be in general due to the existence of strictly semistable Higgs pairs. For the invariant defined by weighted Euler characteristics \eqref{kaidef} studied in paper II \cite{TT2} this is not a problem, however.

However, the monopole branch \eqref{cpt2} with nonzero Higgs field has been largely ignored until now.
In Section \ref{quintic} we compute the part of the generating series of Vafa-Witten invariants $\VW$ \eqref{definitive} coming from a series of components with nonzero Higgs field. With some tricks we manage to sum the series in closed form in Proposition \ref{alge}. The result is far from modular --- in fact it is an \emph{algebraic} function of $q$,
\beq{formulah}
c\;(1-q)^{g-1}\left(1+\frac{1-3q}{\sqrt{(1-q)(1-9q)}}\right)^{\!1-g},
\eeq
where $c=(-1)^{p_g(S)+g}\cdot2^{-p_g(S)-1}$ and $g-1=c_1(S)^2$.

In Section \ref{bah} we repeat
the calculation for the bare Euler characteristic version of $\vw$ \eqref{kaidef} (this is reasonable since we expect the Behrend function to be 1 on $\cM_2$). We find something which \emph{is} modular form up to a factor of $(1-q)^{e(S)}$. And calculations with semistable sheaves on K3 surfaces in \cite[Section 5]{TT2} also give modular answers for $\vw$. So for some time this misled us into believing that $\vw$ is the correct invariant.

However, Martijn Kool pointed out that the coefficients of $q^0$ and $q^1$ in \eqref{formulah} match those in a modular form in \cite[Equation 5.38]{VW}. So in the rest of Section \ref{quintic} we compute the contributions of the remaining components of the monopole branch $\cM_2$ \eqref{cpt2} to $\VW$ for $c_2=2,3$. When added to \eqref{formulah} we get agreement also in $q^2$ and $q^3$. That is, we find in \eqref{answ} that
\begin{multline*}
\!\!\sum_{n\in\Z}\VW^{\;\phi\ne0}_{2,K_S,n}(S)q^n\ =\ \\
\quad\frac1{2^{p_g(S)+g}}\prod_{n=1}^\infty(1-q^{2n})^{-12(p_g(S)+1)}(1-q^n)^{2g-2}\left(\sum_{j=0}^\infty q^{j^2+j}\right)^{\!\!1-g} \pmod{q^4}.\!\!
\end{multline*}
The right hand side is the (modular) second term in \cite[Equation 5.38]{VW}.\footnote{Up to an overall sign $(-1)^{\vd}$ which arises from our orientation conventions. Vafa and Witten identify the tangent and cotangent bundles of moduli space using a Riemannian
metric (which is natural from the real point of view), but this changes the natural complex orientation that we use.\label{fnsign}}

This agreement may not sound like much, but given the complexity of the calculations across 15 pages, we feel it cannot be a coincidence. The upshot is that we now believe the virtual localisation invariant $\VW$ of this paper to be the ``correct" one for physics.\footnote{At least in the stable case; for semistable Higgs pairs see \cite{TT2}.} Remarkably, it seems that the generating series of invariants coming from the two types of fixed component (\ref{cpt1}, \ref{cpt2}) can both be written separately in terms of modular-like functions, but the contributions of natural series of subcomponents of $\cM_2$ \eqref{cpt2} cannot.

Even more remarkably, Vafa and Witten predicted the above formula without calculating on the nested Hilbert schemes making up $\cM_2$ \eqref{cpt2}. It seems our geometric techniques are still a lot less powerful then the physics of the early 1990s. The Kiem-Li cosection localisation we use in Section \ref{kiemli} may be a primitive analogue of their perturbation \cite[Equation 5.19]{VW} which they use to localise their calculation to ``cosmic string" contributions. Our results also seem to indicate that the Gieseker compactification of moduli of sheaves is relevant to physics (which usually uses the Uhlenbeck compactification).

In Section \ref{bah} we discuss the modularity properties of the the alternative definition $\vw$ \eqref{kaidef}. These seem to be the \emph{wrong} ones for physics (except when $K_S\le0$). In particular the modular forms in Section \ref{bah} depend on $S$ only through its Euler characteristic $c_2(S)$, whereas Vafa-Witten also predict a dependence on $c_1(S)^2$ when $S$ has general type.

\subsection{Semistable case}
We give a brief overview of the companion paper \cite{TT2}. Since the work of Joyce-Song and Kontsevich-Soibelman uses the Behrend function and weighted Euler characteristics, we can use it --- with some modifications when $h^1(\cO_S)>0$ --- to generalise the invariant $\vw$ to the semistable case.

What came as more of a surprise to us was that modifying the Joycian formalism in a different way --- using virtual localisation in place of weighted Euler characteristics, and replacing Joyce-Song's universal formulae by its logarithm --- seems to work similarly well to give an extension of $\VW$ to the semistable case.

That is, motivated by Mochizuki's work \cite{Mo} and Joyce-Song pairs \cite{JS} we virtually enumerate certain stable pairs $$(\cE_\phi,s)$$ on $X=K_S$ (or equivalently triples $(E,\phi,s)$ on $S$.) Here $\cE_\phi$ is a semistable torsion sheaf with centre of mass zero on the $K_S$ fibres (equivalently $\tr\phi=0$) and $\det\pi_{*\;}\cE_\phi=\det E\cong\cO_S$. Finally $$s\ \in\ H^0(X,\cE_\phi(N))\ \cong\ H^0(S,E(N))$$ is any section which does not factor through a destabilising subsheaf of $\cE_\phi$. Here $N\gg0$ is fixed and sufficiently large that $H^{\ge1}(\cE_\phi(N))=0$ for all semistable sheaves $\cE_\phi$ with the same Chern classes.

There is a symmetric obstruction theory for such pairs, given by combining the $R\Hom\_\perp$ perfect obstruction theory of \eqref{parp} with Joyce-Song's pairs theory. Localising the virtual cycle to the fixed points of the usual $\C^*$ action defines invariants. We conjecture these can be written in terms of universal formulae in $N$; the coefficients of these formulae then define the Vafa-Witten invariants.

We show the conjecture holds in some cases, for instance when stability and semistability coincide: in this case the resulting invariants $\VW$ coincide with the definition \eqref{definitive} of this paper. The conjecture also holds for $\deg K_S<0$, and for K3 surfaces. In the latter case the invariants $\VW$ recover (and generalise) modular forms conjectured by Vafa and Witten.

\subsection{Kapustin-Witten equations on projective surfaces}
It would be natural to try to repeat the same trick for the Kapustin-Witten equations \cite{KW} on a projective surface $S$. On any K\"ahler surface the equations reduce \cite{Ta2} to Simpson's equations, for which there is a Hitchin-Kobayashi correspondence \cite{Si}. Via the spectral construction, solutions can be identified with torsion sheaves on the Calabi-Yau 4-fold $$X\ =\ T^*S.$$
Borisov-Joyce \cite{BJ} describe a \emph{real} analogue of perfect obstruction theory and virtual cycle for sheaves on Calabi-Yau 4-folds $X$, and Cao-Leung \cite{CL} conjecture a way to extend virtual localisation to this setting. We would like to apply this to the natural $\C^*$ action on the fibres of $X\to S$ to define Kapustin-Witten invariants for $S$. But the $\C^*$ action does not preserve the Calabi-Yau form on $X$, so the theory does not immediately apply.

\medskip\noindent\textbf{Acknowledgements.} We are most grateful to Martijn Kool for the crucial observation that --- despite not appearing to give modular forms --- the first few terms of our calculation \eqref{expa} match some formulae in \cite{VW}. Special thanks to Amin Gholampour for pointing out various embarrassing mistakes before anyone else saw them, and a rapid and thorough referee for catching the rest. We also thank Kevin Buzzard, Julien Grivaux, Hiroshi Iritani, Dominic Joyce, Kotaro Kawatani, Davesh Maulik, Rahul Pandharipande, Mauro Porta, Artan Sheshmani, Bal\'azs Szendr\H oi, Yukinobu Toda and Edward Witten for useful conversations.

Y.T. was partially supported by JSPS Grant-in-Aid for Scientific Research numbers JP15H02054 and JP16K05125, and a Simons Collaboration Grant on ``Special holonomy in Geometry, Analysis and Physics". He thanks Seoul National University, NCTS at National Taiwan University, Kyoto University and BICMR at Peking University for support and hospitality during visits in 2015--17 where part of this work was done. 

\subsection{Notation} To keep the notation readable we suppress some obvious pullback maps, and, given a map, we often use the same notation for any basechange thereof. So for instance $\pi$ will denote the standard projection $X\to S$ of \eqref{pidef}, but also the induced map $\cN\times X\to\cN\times S$ (which should properly be called $\id_{\cN}\times\pi$). Therefore, when we let $p\_S$ denote the projection $S\to\Spec\C$, it also denotes the projections $\cN\times S\to\cN$ and $\curly M\times S\to\curly M$.

Where possible we stick to the notation of
$$
(E,\phi) \quad\mathrm{and}\quad \cE=\cE_\phi
$$
for a Higgs pair on $S$ and the associated torsion sheaf on $X=K_S$ respectively. In families these get replaced by
$$
(\E,\Phi) \quad\mathrm{and}\quad \EE=\EE_\Phi
$$
respectively. For instance these could denote the universal Higgs pair over $\cN\times S$ and the universal torsion sheaf over $\cN\times X$ respectively.

Stability always means Gieseker stability, and in this paper we always assume semistability implies stability. We reserve $\cM$ for moduli spaces of stable sheaves on $S$, while $\curly M$ denotes the moduli \emph{stack} of all coherent sheaves on $S$. Similarly $\cN$ is used for moduli spaces of stable Higgs pairs on $S$ or, equivalently, compactly supported stable torsion sheaves on $X=K_S$.

Given a map $\pi$, we denote by $\hom_\pi$ the functor $\pi\circ\hom$ of Homs down the fibres of $\pi$, and by $R\hom_\pi\cong R\pi_*R\hom$ its derived functor.

For a sheaf or complex of sheaves $\cF$, we use $\cF^\vee=R\hom(\cF,\cO)$ to denote its derived dual. For a single sheaf, $\cF^*$ denotes its underived dual $h^0(\cF^\vee)$.

We use $\LL$ to denote Illusie's cotangent complex, and $\TT:=\LL^\vee$ for its dual, the tangent complex.

\section{Spectral construction}\label{ss}
Fix a complex projective algebraic surface $S$ and let
\beq{pidef}
\pi\colon X\To S
\eeq
be the total space of a line bundle $\cL\to S$. (We will mainly be interested in the case $\cL=K_S$.)
We recall the classical spectral construction relating $\cL$-Higgs pairs $(E,\phi)$ on $S$ to compactly supported sheaves $\mathcal E_\phi$ on $X$.

Roughly speaking, over each point $x\in S$, we replace $(E_x,\phi_x)$ by the eigenspaces of $\phi_x\in\Hom(E_x, E_x\otimes \cL_x)$ supported on their respective eigenvalues in $\cL_x$. More precisely we get an equivalence of categories as follows.

\begin{prop}
There is an abelian category $\mathrm{Higgs}_\cL(S)$ of $\cL$-Higgs pairs $(E,\phi)$ on $S$, and an equivalence of categories
\beq{eqcat}
\mathrm{Higgs}_\cL(S)\ \cong\ \mathrm{Coh}_c(X)
\eeq
with the category of compactly supported coherent sheaves on $X$.
\end{prop}

\begin{proof} The maps in Higgs$_\cL(S)$ between $(E,\phi)$ and $(F,\psi)$ are maps $f\colon E\to F$ inducing a commutative diagram
\beq{es}
\xymatrix@R=18pt{
E \ar[r]^-\phi\ar[d]_f & E\otimes \cL \ar[d]_{f\otimes}^{\id_\cL}
\\ F \ar[r]^-\psi & F\otimes \cL.\!}
\eeq
Taking kernels and cokernels of the columns defines kernel and cokernel Higgs pairs.

It is a classical result of Serre (see for instance \cite[Exercise II.5.17]{Ha}) that for affine maps $\pi\colon X\to S$, the functor $\pi_*$ is an equivalence between the abelian category of coherent $\cO_X$-modules and the abelian category of $\pi_*\;\cO_X$-modules on $S$. In our case
\begin{equation}\label{alg}
\pi_*\;\cO_X\ =\ \bigoplus_{i\ge0}\,\cL^{-i}\cdot\eta^i,
\end{equation}
where $\eta$ is the tautological section of $\pi^*\cL$ which is linear on the fibres and cuts out the zero section $S\subset X$.  The affine $X/S$ is Spec of this sheaf of $\cO_S$-algebras, and sheaves $\cE$ on $X$ are equivalent to sheaves of modules $\pi_*\;\cE$ over $\pi_*\;\cO_X$ \eqref{alg}. 

But \eqref{alg} is generated by $\cO_S$ and $\cL^{-1}\cdot\eta$, so a module over it is equivalent to an $\cO_S$-module $E$ together with a commuting action of $\cL^{-1}\cdot\eta$, i.e. an $\cO_S$-linear map
$$\xymatrix{
E\otimes \cL^{-1} \ar[r]^(.58){\pi_*\eta}& E.}
$$
Thus we get an $\cL$-Higgs pair
\beq{higgs2}
(E,\phi)=(\pi_*\;\cE,\pi_*\eta).
\eeq
Conversely, given a Higgs pair $(E,\phi)$ we get an action of $\cL^{-i}\cdot\eta^i$ by
\beq{degi}
\xymatrix{E\otimes \cL^{-i} \ar[r]^(.6){\phi^i}& E.}
\eeq
Summing over $i$ gives an action of $\pi_*\;\cO_X$ \eqref{alg} on $E$. We denote by $\cE_\phi$ the sheaf on $X$ that this $\pi_*\;\cO_X$-module defines.

This equivalence is of course a functor: morphisms of sheaves $\cE$ induce morphisms of their pushdowns $E=\pi_*\;\cE$ which commute with the action of $\eta$, therefore giving commutative diagrams \eqref{es} and so morphisms in Higgs$_\cL(S)$. 

Finally, for $\cE$ coherent, the quasicoherent sheaf $\pi_*\;\cE$ is coherent if and only if $\pi|_{\mathrm{supp}\,\cE}$ is proper, if and only if $\cE$ is compactly supported.
\end{proof}

\subsection{Gieseker stability} \label{Gss}
The $\cL$-Higgs pair $(E,\phi)$ is said to be \emph{Gieseker stable} with respect to an ample line bundle $\cO_S(1)$ if and only if
\beq{stab1}
\frac{\chi(F(n))}{\rk(F)}\ <\ \frac{\chi(E(n))}{\rk(E)} \quad\mathrm{for\ }n\gg0,
\eeq
for every proper $\phi$-invariant subsheaf $F\subset E$ on $S$. Here $\chi$ denotes holomorphic Euler characteristic. Replacing $<$ by $\le$ defines semistability.

Gieseker stability implies slope semistability of $(E,\phi)$, and is implied by slope stability. Recall \eqref{stabb} this is defined by the inequality
$$
\frac{c_1(F)\cdot h}{\rk(F)}\ <\ \frac{c_1(E)\cdot h}{\rk(E)}\,, \quad h:=c_1(\cO(1)).
$$
When the denominator and numerator  of the right hand side (the rank and degree of $E$) are coprime, we further find that
\begin{center}
slope semistable = Gieseker semistable = Gieseker stable = slope stable.
\end{center}

\begin{lem} Under the equivalence of categories \eqref{eqcat}, Gieseker (semi)\-stability of the $\cL$-Higgs pair $(E,\phi)$ with respect to $\cO_S(1)$ is equivalent to Gieseker stability of the sheaf $\cE_\phi$ with respect to $\pi^*\cO_S(1)$.
\end{lem}

\begin{proof}
Under the equivalence, $\phi$-invariant subsheaves $F\subset E$ are equivalent to subsheaves $\cF\subset\cE_\phi$ on $X$. Moreover $\chi(E)=\chi(\pi_*\;\cE_\phi)=\chi(\cE_\phi)$, so Gieseker stability \eqref{stab1} of $(E,\phi)$ is equivalent to
\beq{stab2}
\frac{\chi(\cF(n))}{r(\cF)}\ <\ \frac{\chi(\cE_\phi(n))}{r(\cE_\phi)} \quad\mathrm{for\ }n\gg0,
\eeq
for all proper subsheaves $\cF\subset\cE_\phi$. Here we have suppressed the pullback map $\pi^*$ on $\cO(1)$, and denoted by $r(\cE_\phi):=\int_X c_1(\cE_\phi)\cdot h^2=\rk(E)\int_Sh^2$ the leading coefficient of its Hilbert polynomial $\chi(\cE_\phi(n))$.

But \eqref{stab2} is the usual Gieseker stability for the torsion sheaf $\cE_\phi$ on $X$, so these two notions of Gieseker stability for $(E,\phi)$ and $\cE_\phi$ coincide.
\end{proof}

\subsection{Resolution}
Since $\cE_\phi$ is generated by its sections down $\pi$ we have a natural surjection $\pi^*E=\pi^*\pi_*\;\cE\rt\ev\cE_\phi$. Its kernel is given by the following.\footnote{The result says that if we divide $\pi^*E$ by the minimal submodule to ensure that $\eta$ acts as $\pi^*\phi$ on the quotient, we get $\cE_\phi$.}

\begin{prop}
There is an exact sequence
\begin{equation}\label{resolution}
\xymatrix@=15pt{
0 \ar[r]& \pi^*(E\otimes \cL^{-1}) \ar[rr]^(.6){\pi^*\phi-\eta}&& \pi^*E \ar[r]^(.53)\ev& \cE_\phi \ar[r]& 0.}
\end{equation}
\end{prop}

\begin{proof}
Consider  $X\times_S X$ --- the total space of $\cL^{\oplus2}$ over $S$ --- with its two natural projections $\pi_1,\pi_2$ to $X$ and projection $p$ to $S$. On it is an obvious exact sequence
\begin{equation}\label{pres}
0\To\pi_1^*\cE(-\Delta)\To\pi_1^*\cE\To\Delta_*\cE\To0,
\end{equation}
where $\Delta\cong X\subset X\times_SX$ is the diagonal divisor.

Now $p^*\cL$ has two tautological sections $\eta_1,\eta_2$ whose difference $\eta_1-\eta_2$ cuts out the diagonal. Thus $\cO(\Delta)\cong p^*\cL\cong\pi_2^*\pi^*\cL$ and we can rewrite \eqref{pres} as
$$
\xymatrix@=15pt{
0 \ar[r]& \pi_1^*\cE\otimes\pi_2^*\pi^*\cL^{-1} \ar[rr]^(.63){\eta_1-\eta_2}&& \pi_1^*\cE \ar[r]& \Delta_*\cE \ar[r]& 0.}
$$
Now apply $\pi_{2*}$, using the projection formula and the flat basechange formula $\pi_{2*}\pi_1^*=\pi^*\pi_*$. By \eqref{higgs2} we get \eqref{resolution}.
\end{proof}

\subsection{Deformation theory} \label{defthy}
This resolution \eqref{resolution} of $\cE$ is useful to 
relate the deformation theory of $\cE$, governed by
$$
\Ext^*_X(\cE,\cE),
$$
with that of the Higgs pair $(E,\phi)$, governed by the cohomology groups of the total complex of
$$
\xymatrix@C=35pt{
R\Hom_S(E,E) \ar[r]^(.43){[\ \cdot\ ,\,\phi]} & R\Hom_S(E,E\otimes \cL).}
$$

\begin{prop}
There is an exact triangle
\beq{deftri}
\xymatrix@C=18pt{R\Hom(\cE,\cE) \ar[r]& R\Hom(E,E) \ar[rr]^(.42){\circ\phi-\phi\circ}&& R\Hom(E\otimes \cL^{-1},E).\!\!}
\eeq
\end{prop}

\begin{proof}
Applying $R\hom(\ \cdot\ ,\cE)$ to \eqref{resolution} gives the exact triangle
%we have the quasi-isomorphism
%$$
%\cE\ \simeq\ \big\{\!\!\xymatrix@C=40pt{\pi^*(E\otimes \cL^{-1}) \ar[r]^(.6){\pi^*\phi-\eta}& \pi^*E}\!\!\big\}
%$$
$$
\xymatrix@C=20pt{
R\hom(\cE,\cE) \ar[r]& R\hom(\pi^*E,\cE) \ar[rr]^(.42){\circ\pi^*\phi-\eta}&& R\hom(\pi^*(E\otimes \cL^{-1}),\cE).}
$$
Applying $R\pi_*$ and adjunction gives 
$$
\xymatrix@C=20pt{R\hom_\pi(\cE,\cE) \ar[r]& R\hom(E,\pi_*\;\cE) \ar[rr]^(.42){\circ\phi-\pi_*\eta}&& R\hom(E\otimes \cL^{-1},\pi_*\;\cE).}
$$
Since $\pi_*\eta\colon\pi_*\;\cE\to\pi_*\;\cE\otimes \cL$ is $\phi\colon E\to E\otimes \cL$ by \eqref{higgs2} this is
\beq{locv}
\xymatrix@C=20pt{R\hom_\pi(\cE,\cE) \ar[r]& R\hom(E,E) \ar[rr]^(.44){\circ\phi-\phi\circ}&& R\hom(E,E\otimes \cL)}.
\eeq
Taking $R\Gamma_{\!S}$ gives \eqref{deftri}.
\end{proof}
        
In particular, taking cohomology in \eqref{deftri} gives the long exact sequence \eqref{LEs} of the introduction.

\subsection{In families} We can generalise to \emph{families} of sheaves over $S$ and $X$, and, moreover, we can let $S$ and $X$ themselves vary smoothly (to prove deformation invariance of our invariants). So we now let $S\to B$ be a smooth projective morphism with 2-dimensional fibres, and we let $X\to B$ denote the total space of a line bundle $\cL$ (soon to be the relative canonical bundle $K_{S/B}$) over $S$. On a first viewing, the reader should set $B=\Spec\C$ and forget all about it.

Let $\cN\to B$ denote the moduli space of Gieseker stable Higgs pairs on the fibres of $S\to B$ with fixed rank\footnote{In this paper we always take rank $r>0$. When $r=0$ the obstruction theory \eqref{parp} has to be modified.} and Chern classes $(r,c_1,c_2)$. Equivalently it is a moduli space of compactly supported stable torsion sheaves on the fibres of $X\to B$, with rank 0 and
\begin{eqnarray}
c_1 &=& r[S], \nonumber \\
c_2 &=& \iota_*\Big(\frac{r(r+1)}2c_1(\cL)-c_1\Big), \label{chern} \\
c_3 &=& \iota_*\Big(c_1^2-2c_2-(r+1)c_1\cdot c_1(\cL)+\frac{r(r+1)(r+2)}6c_1(\cL)^2\Big) \hspace{-1cm} \nonumber
\end{eqnarray}
in $H^*_c(X_t,\Z)$. Here $\iota\colon S\into X$ is the zero section, $[S]$ its Poincar\'e dual, and $X_t$ is the fibre of $X\to B$ over any closed point $t\in B$.

Pick a (twisted)\footnote{Although a global universal sheaf may not exist in general due to $\C^*$-automorphisms, one can work locally over $\cN$ (where one always exists) or globally with a \emph{twisted} universal sheaf. We  can ignore this issue since we are only concerned with (derived functors of) $\hom(\cE,\cE)$, which exists globally and glues uniquely
and independently of any choices.} universal sheaf
$$
\EE\ \ \mathrm{over}\ \ \cN\times_BX.
$$
As usual we let $\pi$ denote both the projection $X\to S$ of \eqref{pidef} and (cf. Notation section) any basechange such as $\cN\times_BX\to\cN\times_BS$.
Since $\EE$ is flat over $\cN$ and $\pi$ is affine,
\beq{higgs3}
\E:=\pi_{*\,}\EE\ \ \mathrm{over}\ \ \cN\times_BS
\eeq
is also flat over $\cN$. It is also coherent because $\EE$ is finite over $\cN\times_BS$. Thus it defines a classifying map
\begin{eqnarray} \label{Pi}
\hspace{1cm} \Pi\colon\ \quad\cN\! &\To& \curly M, \\
\cE\, &\Mapsto& \!\pi_{*\;}\cE\quad\mathrm{(or,\ equivalently,}\quad (E,\phi)\Mapsto E), \nonumber
\end{eqnarray}
to the moduli \emph{stack} $\curly M$ of coherent sheaves on the fibres of $S\to B$ of the given rank and Chern classes. We will often not distinguish between the universal sheaf $\E$ over $\curly M\times S$ and its pullback $\E=\Pi^*\E$ over $\cN\times S$. \medskip

By the same working as above, over the families
$$
p\_X\colon\cN\times_BX\to\cN\ \ \mathrm{and}\ \ p\_S\colon\cN\times_BS\to\cN
$$
instead of on single copies of $X$ and $S$, we replace \eqref{deftri} by the exact triangle
\beq{deftri2}
\xymatrix@C=30pt{
R\hom_{p\_X}(\EE,\EE) \ar[r]^{\pi_*}& R\hom_{p\_S}(\E,\E) \ar[r]^(.42){[\ \cdot\ ,\,\phi]} & R\hom_{p\_S}(\E,\E\otimes \cL).}\hspace{-2mm}
\eeq
Presently we will see how this relates the obstruction theories of $\cN$, $\curly M$ and $\Pi\colon\cN\to\curly M$ respectively.
From now on, for simplicity we restrict attention to the case of interest $\cL=K_{S/B}$. (Later in Section \ref{later} we will sketch what small changes have to be made in the general case.) We always work relative to the base $B$, but often suppress mention of it.

\begin{prop}\label{symmetric}
Setting $\cL=K_{S/B}$, the exact triangle \eqref{deftri2} is its own Serre dual.
\end{prop}

\begin{proof}
The claim is that replacing each term in \eqref{deftri2} with its relative Serre dual (down the $X$- and $S$-fibres respectively), and the arrows by their duals,
$$
\xymatrix@C=15pt{
R\hom_{p\_X}(\EE,\EE)[3] & R\hom_{p\_S}(\E,\E\otimes K_{S/B})[2] \ar[l] & R\hom_{p\_S}(\E,\E)[2], \ar[l]}
$$
we get the same exact triangle, just shifted.

For the right hand arrows, this is simple. We concentrate on the left hand arrows. The claim is that the following horizontal arrows
$$
\xymatrix@R=-1pt{
R\hom_{p\_X}(\EE,\EE) \ar[r]^{\pi_*}& R\hom_{p\_S}(\E,\E)
\\ \otimes & \otimes \\
R\hom_{p\_X}(\EE,\EE)[3] \ar[dddd]& R\hom_{p\_S}(\E,\E\otimes K_{S/B})[2] \ar[l]^(.52)\partial\ar[dddd] \\\\\\\\
\cO_{\cN}\!\! & \cO_{\cN}\!\!}
$$
intertwine the vertical pairings given by cup product and trace. Here $\partial$ is the coboundary morphism of the exact triangle \eqref{deftri2}. That is, identifying
\begin{eqnarray} \nonumber
R\hom_{p\_S}(\E,\E\otimes K_{S/B}) &\cong&
R\hom_{p\_S}(\E\otimes K_{S/B}^{-1},\pi_*\;\EE) \\ \label{ide}
&\cong& R\hom_{p\_X}(\pi^*\E\otimes K_{S/B}^{-1},\EE),
\end{eqnarray}
it is cup product with the canonical extension class
\beq{extc}
e\,\in\,\Ext^1_{X\times\cN}(\EE,\pi^*\E\otimes K_{S/B}^{-1})
\eeq
of the resolution \eqref{resolution}. So we want to show that the two parings
$$
\hspace{-1cm}\xymatrix@R=0pt@C=65pt{
\hspace{6mm} R\hom_{p\_X}(\EE,\EE)\stackrel L\otimes R\hom_{p\_X}(\pi^*\E\otimes K_{S/B}^{-1},\EE)[2]\hspace{6mm} \ar@{=}[dd]\ar[rd]^(.73){\tr\_X\big((\ \cdot\ )\circ(\ \cdot\ )\circ e\big)} \\ & \cO_{\cN} \\
R\hom_{p\_X}(\EE,\EE)\stackrel L\otimes R\hom_{p\_S}(\E,\E\otimes K_{S/B})[2] \ar[ru]_(.72){\tr\_S\big(\pi_*(\ \cdot\ )\circ(\ \cdot\ )\big)}}
$$
are equal. Since $\tr(ab)=\tr(ba)$ the upper map may be rewritten
\beq{trs}
\tr\_S\big(e\circ(\ \cdot\ )\circ(\ \cdot\ )\big)
\eeq
when we identify the product of the first two terms in \eqref{trs} as lying in
\begin{eqnarray} \nonumber
R\hom_{p\_X}(\EE,\pi^*\E\otimes K_{S/B}^{-1})[1] &\cong& 
R\hom_{p\_X}(\EE,\pi^!\E) \\ \label{adj}
&\cong& R\hom_{p\_S}(\pi_*\;\EE,\E)\ \cong\ R\hom_{p\_S}(\E,\E)
\end{eqnarray}
(since the dualising complex of $\pi$ is $\pi^*K_{S/B}^{-1}[1]$) and the last in
$$
R\hom_{p\_X}(\pi^*\E\otimes K_{S/B}^{-1},\EE)[2]\ \cong\ 
R\hom_{p\_S}(\E,\E\otimes K_{S/B})[2].
$$
The identification \eqref{adj} is isomorphic to the composition
$$
R\hom_{p\_X}(\EE,\pi^!\E)\rt{\pi_*}
R\hom_{p\_S}(\pi_*\;\EE,\pi_*\pi^!\E)\rt{\epsilon}
R\hom_{p\_S}(\pi_*\;\EE,\E),
$$
where $\epsilon\colon\pi_*\pi^!\to\id$ is the counit. Therefore the product of the first two terms of \eqref{trs} is
$$
\epsilon\circ\pi_*\big(e\circ(\ \cdot\ )\big)\ =\ 
\epsilon\circ\pi_*(e)\circ\pi_*(\ \cdot\ ).
$$
But $\epsilon\circ\pi_*(e)$ is the image of $e$ under the identification
$$
e\,\in\,\Ext^1_{X\times\cN}(\EE,\pi^!\E[-1])\ \cong\ \Hom_{S\times\cN}(\pi_*\;\EE,\E)\ \cong\ \Hom_{S\times\cN}(\E,\E),
$$
which is $\id_E$. Thus \eqref{trs} is $\tr\_S\big(\pi_*(\ \cdot\ )\circ(\ \cdot\ )\big)$, as required.
\end{proof}

\begin{cor} \label{symme}
The exact triangle \eqref{deftri2} fits into the following
commutative diagram of exact triangles with split columns
$$
\xymatrix@R=18pt@C=3pt{
R\hom_{p\_S}(\E,\E\otimes K_{S/B})\_0[-1] \ar@{<->}[d]\ar[r]& R\hom_{p\_X}(\EE,\EE)\_\perp \ar[r]\ar@{<->}[d]& R\hom_{p\_S}(\E,\E)\_0 \ar@{<->}[d] \\
R\hom_{p\_S}(\E,\E\otimes K_{S/B})[-1] \ar@{<->}[d]^{\tr}_\id\ar[r]& R\hom_{p\_X}(\EE,\EE) \ar[r]\ar@{<->}[d]& R\hom_{p\_S}(\E,\E) \ar@{<->}[d]^\tr_\id \\
Rp\_{S*}K_{S/B}[-1] \ar@{<->}[r]& Rp\_{S*}K_{S/B}[-1]\oplus Rp\_{S*}\cO_S \ar@{<->}[r]& Rp\_{S*}\cO_S,\!}
$$
where the suffix 0 denotes trace-free Homs. Letting $R\hom_{p\_X}(\EE,\EE)\_\perp$ denote the co-cone of the central column, it is Serre dual to its own shift by $[3]$.
\end{cor}

\begin{proof}
The left and right hand columns are split because $\tr\circ\id=\rk\ne0$.

The composition
\beq{OM2}
\cO_S\rt{\id}R\hom_\pi(\EE,\EE)\rt{\pi_*}R\hom(\E,\E)\rt{\tr}\cO_S
\eeq
is $\rk(E)\cdot\id$, so the first arrow is split. Apply $Rp\_{S*}$ and use $Rp\_{S*}R\pi_*=Rp\_{X*}$ to give the maps
\beq{maps}
R\hom_{p\_X}(\EE,\EE)\longleftrightarrow Rp\_{S*}\cO_S
\eeq
lifting the corresponding maps on $R\hom_{p\_S}(\E,\E)$. 

The Serre duals of the maps \eqref{maps} give the maps 
\beq{OM}
R\hom_{p\_X}(\EE,\EE)\longleftrightarrow Rp\_{S*}K_{S/B}[-1].
\eeq
By Proposition \ref{symmetric} these also commute with the Serre duals of the identity and trace maps on $R\hom_{p\_S}(\E,\E\otimes K_{S/B})[-1]$. Since these are the trace and identity maps respectively, the result follows.
\end{proof}

This co-cone $R\hom_{p\_X}(\EE,\EE)\_\perp$ will eventually provide the symmetric perfect obstruction theory we need for the moduli space $\cN_L^\perp$ of stable trace-free fixed-determinant Higgs pairs; see Theorem \ref{final}.

\section{Perfect obstruction theory for $U(r)$ theory}\label{potsec}
We will describe the perfect obstruction theories of $\cN$ and $\curly M$ and, for later use, their relationship under $\Pi$ \eqref{Pi}. The exact triangle of cotangent complexes
$$
\Pi^*\LL_{\curly M/B}\rt{\Pi^*}\LL_{\cN/B}\To\LL_{\cN/\curly M}
$$
may fail to be exact if we use truncated cotangent complexes as in \cite{HT}. Therefore we use full cotangent complexes, to which \cite{HT} does not immediately apply. However, since we are deforming sheaves ($\cE$ and $E$) rather than complexes of sheaves, we can use  Illusie's seminal work \cite{Ill}.

\subsection{Atiyah classes}
Illusie \cite[Section III.2.3]{Ill} defines the Atiyah class of a coherent sheaf $\cF$ on a $B$-scheme $Z$ as follows. Using the notation of Gillam \cite{Gi} we let
$$
Z[\cF]:=\Spec\_{\!\!Z}(\cO_Z\oplus\cF)\rt{q}Z
$$
be the trivial square-zero thickening of $Z$ by $\cF$. Here $\cO_Z\oplus\cF$ is an $\cO_Z$-algebra in the obvious way: $(f,s)\cdot(g,t):=(fg,ft+gs)$. It carries a $\C^*$ action, fixing $\cO_Z$ and acting on $\cF$ with weight 1. We get the following exact triangle of $\C^*$-equivariant cotangent complexes on $Z[\cF]$,
\beq{LLLL}
q^*\LL_{Z/B}\To\LL_{Z[\cF]/B}\To\LL_{Z[\cF]/Z}\To q^*\LL_{Z/B}[1].
\eeq
Applying $q_*$ and taking weight 1 parts is an exact functor. Applied to the last two terms of \eqref{LLLL} it yields
\beq{atdef}
\cF\To\cF\otimes\LL_{Z/B}[1]
\eeq
on $Z$. We call this map $\At_{\cF}\in\Ext^1(\cF,\cF\otimes\LL_{Z/B})$ the Atiyah class of $\cF$.\medskip

We apply this to the universal sheaf $\EE$ on $\cN\times_BX$ and project to the first summand of $\LL_{\cN\times_BX/B}=\LL_{\cN/B}\oplus\LL\_{X/B}$ (suppressing some pullback maps) to give the partial Atiyah class
\beq{par}
\At_{\EE\!,\,\cN}\colon\EE\To\EE\otimes\LL_{\cN/B} \quad\mathrm{on\ }\cN\times_BX.
%\At_{\EE\!,\,\cN}\in\Ext^1(\EE,\EE\otimes\LL_{\cN})[1].
\eeq
Similarly applied to $\E$ on $\cN\times_BS$ and projecting $\LL_{\cN\times_BS}$ to $\LL_{\cN/B}$ defines a partial Atiyah class
\beq{parS}
\At_{\E,\,\cN}\colon\E\To\E\otimes\LL_{\cN/B}[1] \quad\mathrm{on\ }\cN\times_BS.
\eeq
The relationship between these two classes turns out to be very simple, and will allow us to relate the deformation theories of $\cE$ and $E$.

\begin{prop}\label{funcAt} Applying $\pi_*$ to \eqref{par},
$$
\pi_*\At_{\EE\!,\,\cN}\colon\pi_{*\,}\EE\To\pi_{*\,}\EE\otimes\LL_{\cN/B}[1]
$$
gives \eqref{parS}. That is, $\pi_{*\!}\At_{\EE\!,\,\cN}=\At_{\E,\,\cN\,}$.
\end{prop}

\begin{proof}
We use the natural $\C^*$-equivariant\footnote{Here we are using the $\C^*$ action which acts with weight 1 on $\EE$ and $\E$ but fixes $X$ and $S$. Only later will we use the different $\C^*$ action on the fibres of $\pi\colon X\to S$.} map
\beq{rho}
\cN\times_BX[\EE]\rt\rho\cN\times_BS[\E]
\eeq
defined by thinking of both as affine schemes over $\cN\times_BS$ and using the pullback map on (sheaves of $\cO_{\cN\times_BS}$-) algebras:
$$\xymatrix@=55pt{
\big(\cO_{\cN\times_BS}\oplus K_{S/B}^{-1}\oplus K_{S/B}^{-2}\oplus\cdots\big)\oplus\E & \cO_{\cN\times_BS}\oplus\E. \ar[l]_(.31){(1,0,0,\cdots,0,1)}&}
$$
(The algebra structure on the left hand side is defined by the action \eqref{degi} of each $K_{S/B}^{-i}$ on $\E$.) Combining the last two terms of \eqref{LLLL} with pullback via $\rho$ \eqref{rho} gives the commutative diagram
$$\xymatrix@=18pt{
\LL_{\cN\times_BX[\EE]/\cN\times_BX} \ar[r]& q^*\LL_{\cN\times_BX}[1] \\
\rho^*\LL_{\cN\times_BS[\E]/\cN\times_BS} \ar[u]^(.45){\rho^*}\ar[r]& \rho^*q^*\LL_{\cN\times_BS}[1]\ar[u]_(.45){\rho^*}.\!}
$$
Taking weight 1 parts gives
$$\xymatrix@R=18pt@C=35pt{
\EE \ar[r]^(.35){\At_{\EE}}& \EE\otimes\LL_{\cN\times_BX}[1] \\
\pi^*\E \ar[u]^(.45)\ev\ar[r]^(.3){\pi^*\!\At_{\E}}& \pi^*\E\otimes\pi^*\LL_{\cN\times_BS}[1]\ar[u]^(.45){\ev\otimes\!}_(.46){\!\pi^*}.\!}
$$
Projecting to the $\LL_{\cN/B}$ factor in both terms on the right hand side gives the commutative diagram
\beq{11}
\xymatrix@R=18pt@C=50pt{
\EE \ar[r]^(.38){\At_{\EE\!,\,\cN}}& \EE\otimes\LL_{\cN/B}[1] \\
\pi^*\E \ar[u]^(.45)\ev\ar[r]^(.37){\pi^*\!\At_{\E,\,\cN}}& \pi^*\E\otimes\LL_{\cN/B}[1]\ar[u]^(.45){\ev\otimes\!}_(.46){\!1}.\!}
\eeq
Apply the adjunction isomorphism
\begin{eqnarray}
\Hom\big(\pi^*\E,\EE\otimes\LL_{\cN/B}[1]\big)\!\! &\cong& \!\!\Hom\big(\E,\pi_{*\;}\EE\otimes\LL_{\cN/B}[1]\big) \nonumber\\
(\ev\otimes1)\circ\pi^*a\ &\ensuremath{\leftarrow\joinrel\relbar\joinrel\shortmid}& \quad a \label{eur}\\
b\circ\ev\ =\ (\ev\otimes1)\circ(\pi^*\pi_{*\,}b)\ &\Mapsto& \ \ \pi_{*\,}b \nonumber
\end{eqnarray}
to $a=\At_{\E,\,\cN}$ and $b=\At_{\EE\!,\,\cN}$. The commutativity of \eqref{11} says that the two maps on the left hand side of \eqref{eur} agree. Therefore the two maps on the right hand side agree. That is, $a=\pi_{*\;}b$ as required.
%Thus the two maps $\At_{\EE\!,\,\cN}\circ\ev$ and $(\ev\otimes1)\circ\At_{E,\,\cN}$ between the lower left hand and upper right hand corners agree. 
%If instead we apply $\pi^*\pi_*$ to the upper row we get a similar diagram
%\beq{2}
%\xymatrix@C=60pt{\EE \ar[r]^(.4){\At_{\EE\!,\,\cN}}& \EE\otimes\LL_{\cN}[1] \\
%\pi^*E \ar[u]_\ev\ar[r]^(.37){\pi^*\pi_{*\!}\At_{\EE\!,\,\cN}}& \pi^*E\otimes\LL_{\cN}[1]\ar[u]^(.49){\ev\otimes\!}_{\!1}}
%\eeq
%differing only in the lower arrow. Since this diagram also commutes, the two maps (\ref{1}, \ref{2}) between the lower left hand and upper right hand corners agree:
%\beq{paree}
%(\ev\otimes1)\circ\pi^*\!\big(\pi_{*\!}\At_{\EE\!,\,\cN}\!\big)\,=\ (\ev\otimes1)\circ\pi^*\!\big(\!\At_{E,\,\cN}\!\big)\,\in\ \Hom(\pi^*E,\EE\otimes\LL_{\cN}).
%\eeq
%Under the adjunction isomorphism
%\beqa
%\Hom(\pi^*E,\EE\otimes\LL_{\cN})\!\! &\cong& \!\!\Hom(E,\pi_*\EE\otimes\LL_{\cN}) \\
%(\ev\otimes1)\circ\pi^*(\psi)\ &\ensuremath{\leftarrow\joinrel\relbar\joinrel\shortmid}& \quad\psi
%\eeqa
%this says that $\pi_{*\!}\At_{\EE\!,\,\cN}=\At_{E,\,\cN}$, as required.
\end{proof}

Consider the partial Atiyah class \eqref{par} as lying in the group
\begin{eqnarray}
\At_{\EE\!,\,\cN} &\in& \Ext^1_{\cN\times_BX}\!\big(\EE,\EE\otimes p_X^*\LL_{\cN/B}\big) \nonumber\\
&=& \Ext^1_{\cN\times_BX}\!\big(R\hom(\EE,\EE),p_X^*\LL_{\cN/B}\big) \label{sd}\\
&\cong& \Ext^2_{\cN}\!\big(Rp\_{X*}R\hom(\EE,\EE\otimes K_{X/B}),\LL_{\cN/B}\big), \nonumber
\end{eqnarray}
using relative Serre duality for the projection $p_X\colon\cN\times_BX\to\cN$ of relative dimension 3. Of course the relative canonical bundle $K_{X/B}\cong\cO$ is trivial, but not \emph{equivariantly} with respect to the $\C^*$ action scaling the fibres of $X\to S$. This acts on $K_{X/B}$ with weight $-1$, so
$$
K_{X/B}\ \cong\ \cO\otimes\t^{-1},
$$
where $\t$ is the standard 1-dimensional representation of $\C^*$. All told we get a map
$$
\At_{\EE\!,\,\cN}\colon R\hom_{p\_X}(\EE,\EE)\t^{-1}[2]\To\LL_{\cN/B}\,.
$$
Using $\tau^{\le0}\;\LL_{\cN/B}=\LL_{\cN/B}$, where $\tau$ is the truncation functor, this gives
\beq{att}
\At_{\EE\!,\,\cN}\colon\tau^{\le0}\big(R\hom_{p\_X}(\EE,\EE)\t^{-1}[2]\big)\To\LL_{\cN/B}\,.
\eeq
Here we have removed $\tau^{\ge3}R\hom_{p\_X}(\EE,\EE)$ from $R\hom_{p\_X}(\EE,\EE)$. By Serre duality and the fact that the sheaves $\EE_t$ are simple (by stability), we have the isomorphisms
\beq{ext3}
\xymatrix{
\Ext^3_X(\EE_t,\EE_t) \ar[r]^(.65){\tr}_(.65){\sim}& \C}
\eeq
for all closed points $t\in\cN$. (Serre duality holds despite the noncompactness of $X$ because supp\,$\EE_t$ is proper. Alternatively,  push forward to the projective completion $\overline X=\PP(\cO_S\oplus K_{S/B})$, which leaves the Ext groups unchanged. Then ordinary projective Serre duality on $\overline X$ gives $\Ext^3(\EE_t,\EE_t)=\Hom(\EE_t,\EE_t\otimes K_{\overline X})^*=\Hom(\EE_t,\EE_t)^*=\C$ since the restriction of $K_{\overline X}$ to $X\supset\ $supp\,$\EE_t$ is trivial.) Therefore $\tau^{\ge3}R\hom_{p\_X}(\EE,\EE)\cong R^3p_*\cO$ by basechange and the Nakayama lemma, so the above truncation
\beq{trunk}\xymatrix{
\tau^{\le2}R\hom_{p\_X}(\EE,\EE)\ =\ \mathrm{Cone}\Big(R\hom_{p\_X}(\EE,\EE) \ar[r]^(.75){\tr}& \cO_{\cN}[-3]\Big)[-1]\hspace{-5mm}}
\eeq
is perfect of amplitude $[0,2]$.

\begin{thm} \label{BF}
The map \eqref{att} is a relative obstruction theory for $\cN/B$ in the sense of \cite[Definition 4.4 and Section 7]{BF}.
\end{thm}

\begin{proof}
Here we can ignore the $\C^*$ action. Fix a morphism of $B$-schemes,
$$
f\colon T_0\To\cN,
$$
and an extension of $B$-schemes $T_0\subset T$ with ideal $I$ such that $I^2=0$.

Compose the pullback of \eqref{att},
\begin{equation} \label{fobs}
f^*\tau^{\le0}\big(R\hom_{p\_X}(\EE,\EE)[2]\big)\To f^*\LL_{\cN/B\,},
\end{equation}
with the natural map $f^*\LL_{\cN/B}\to\LL_{T_0/B}$ followed by the composition\footnote{This composition $\LL_{T_0/B}\to I[1]$ is the Kodaira-Spencer class of $T_0\subset T$.}
\beq{KS}
\xymatrix@=15pt{
\LL_{T_0/B} \ar[r]& \LL_{T_0/T} \ar[rr]^(.35){\tau^{\ge-1}}&& \tau^{\ge-1}\LL_{T_0/T}\ =\ I[1].}
\eeq
This defines an element
\begin{equation} \label{oext2}
o\ \in\ \Ext^{-1}_{T_0}(f^*\tau^{\le2}R\hom_{p\_X}(\EE,\EE),I).
\end{equation}
By \cite[Theorem 4.5]{BF} our Theorem will follow if we can show that $o$ vanishes if and only
if there exists an extension from $T_0$ to $T$
% **affine**
of the map $f$ of
$B$-schemes, and that when $o=0$ the set of extensions is a torsor
under $\Ext^{-2}_{T_0}(f^*\tau^{\le2}R\hom_{p\_X}(\EE,\EE),I)$. \bigskip

Let
$$
\bar f=\id\times f\colon X\times_B T_0\To X\times_B\cN
$$
and let $\bar p\_X$ be the projection
$$
\bar p\_X\colon X\times_B T_0\To T_0.
$$
Since $\bar p\_X$ is flat and $\tau^{\ge3}R\hom_{p\_X}(\EE,\EE)=R^3p\_{X*}\cO[-3]$ is locally free, basechange gives 
\beq{fr}
f^*\tau^{\le2}R\hom_{p\_X}(\EE,\EE)\ \cong\ \tau^{\le2}R\hom_{\bar p\_X}(\bar f^*\EE,\bar f^*\EE).
\eeq
By the obvious functoriality of Atiyah classes, the pullback \eqref{fobs} of the partial Atiyah class is the partial Atiyah class $\At_{\bar f^*\EE\!,\,T_0}$. Therefore, unwinding the working above and the relative Serre duality of \eqref{sd}, we can write \eqref{oext2}
$$
o\ \in\ \Ext^{-1}_{T_0}(f^*\tau^{\le2}R\hom_{p\_X}(\EE,\EE),I)\ \cong\ \Ext^2_{X\times_{\!B}T_0}\big(\bar f^*\EE,\bar f^*\EE\otimes I\big)
$$
as the composition of
$$
\At_{\bar f^*\EE\!,\,T_0}\ \in\ \Ext^1\big(\bar f^*\EE,\bar f^*\EE\otimes\LL_{T_0/B}\big)
$$
with the Kodaira-Spencer class \eqref{KS}.
But in \cite[Proposition III.3.1.8]{Ill}, Illusie shows that precisely this composition is his class $\omega(B,\EE)$, which in \cite[Proposition III.3.1.5]{Ill} he
shows vanishes if and only if $\bar f^*\EE$ deforms from $T_0\times_B X$ to
$T\times_BX$. Furthermore, when it does vanish, he shows such deformations form a torsor under $\Ext^1_{T_0\times_{\!B}X}(\bar f^*\EE,\bar f^*\EE\otimes I)$, which by
\eqref{fr} and relative Serre duality for $\bar p\_X$ is $\Ext^{-2}_{T_0}(f^*\tau^{\le2}R\hom_{p\_X}(\EE,\EE),I)$.

Since deformations of $\bar f^*\EE$ from $X\times_BT_0$ to $X\times_BT$ are in 1-1 correspondence with extensions from $T_0$ to $T$ of
the $B$-map $f$, we are done.
\end{proof}

\begin{cor} \label{2ter}
There is a perfect relative obstruction theory of amplitude $[-1,0]$ for $\cN$,
\beq{pot}
\tau^{[-1,0]}\big(R\hom_{p\_X}(\EE,\EE)[2]\big)\t^{-1}\To\LL_{\cN/B}\,.
\eeq
\end{cor}

\begin{proof} Again, for this proof we can ignore the $\C^*$ action. We have already seen in \eqref{trunk} the complex $\tau^{\le2}R\hom_{p\_X}(\EE,\EE)$ is perfect of amplitude $[0,2]$.
Since stable sheaves are automatically simple we have the isomorphism
\beq{id0}
\xymatrix{
\C \ar[r]^(.25){\id}_(.25){\sim}& \Hom\_X(\EE_t,\EE_t)}
\eeq
for all closed points $t\in\cN$. Therefore, by the standard Nakayama lemma arguments (see for instance \cite[Lemma 4.2]{HT}),
\beq{trun}\xymatrix{
\mathrm{Cone}\Big(\cO_{\cN} \ar[r]^(.37){\id}& \tau^{\le2}R\hom_{p\_X}(\EE,\EE)\Big)}\ =\ \tau^{[1,2]}R\hom_{p\_X}(\EE,\EE)
\eeq
is also perfect, of amplitude [1,2].

By (the dual of) \cite[Proposition 3.2]{STV}\footnote{By \cite[Remark A.1]{STV} their definition of Atiyah class coincides with Illusie's.} we have a commutative diagram
$$
\xymatrix{Rp\_{X*}\cO_{\cN\times_BX}[2] \ar[r]^(.46){\id}\ar[d]_{\At_{\det\EE\!,\,\cN}}& R\hom_{p\_X}(\EE,\EE)[2] \ar[d]^{\At_{\EE\!,\,\cN}} \\
\LL_{\Pic(X/B)/B} \ar[r]^{\det^*}& \LL_{\cN/B\,},\!
}$$
where $\det\colon\cN\to\Pic(X/B)$ takes $\EE$ to its determinant $\cO_X(rS)$. Since this is the constant map,\footnote{We can even do without this fact; since $\Pic(X/B)\to B$ is smooth, $\LL_{\Pic(X/B)/B}$ is concentrated only in degree 0. Therefore the composition \eqref{comp0} is zero in degree $-2$, as required.}
we conclude the composition
\beq{comp0}
\xymatrix@=30pt{\cO_{\cN}[2] \ar[r]^(.33){\id}&
R\hom_{p\_X}(\EE,\EE)[2] \ar[r]^(.65){\At_{\EE\!,\,\cN}}& \LL_{\cN/B},}
\eeq
which factors through $\tau^{\le0}\big(R\hom_{p\_X}(\EE,\EE)[2]\big)$,
is zero. By \eqref{trun} the perfect obstruction theory \eqref{att} therefore factors through a map
$$
\tau^{[-1,0]}\Big(R\hom_{p\_X}(\EE,\EE)[2]\Big)\To\LL_{\cN/B\,}.
$$
Since we have only modified \eqref{att} in degree $-2$, it still induces an isomorphism on cohomology in degree 0 and a surjection in degree $-1$. Thus it is a perfect obstruction theory.
\end{proof}

%We first describe the obstruction theory of $\cN$. In \cite{HT} a \emph{universal Atiyah class} is defined on $\cN\times\cN$:
%\beq{univ}
%\At_{\cN}\colon\Delta_*\cO_{\cN}\To\Delta_*\LL_{\cN}[1].
%\eeq
%Here $\Delta\colon\cN\to\cN\times\cN$ is the diagonal, and $\LL_{\cN}$ is the truncated cotangent complex. We consider it to be a map of Fourier-Mukai kernels, both of which we apply to $\EE$. That is, letting $\pi_i\colon\cN\times\cN\to\cN,\ i=1,2$, denote projection onto the $i$th factor (and its basechange $\cN\times\cN\times X\to\cN\times X$, as usual) we act by
%\beq{act}
%\pi_{2*}\big[\pi_1^*\EE\otimes p_X^*\big(\ \cdot\ \big)\big]
%\eeq
%on \eqref{univ}. Here $p\_X\colon\cN\times X\to\cN$ as before. This defines $\EE$'s Atiyah class\footnote{In fact only the $\LL_{\cN}$ component of its full $\LL_{\cN\times X}=\LL_{\cN}\oplus\LL\_X$-valued Atiyah class.}
%\beq{AtX}
%\EE\To\EE\otimes p_X^*\LL_{\cN}[1]
%\eeq
%on $\cN\times X$.

\section{The $U(r)$ Vafa-Witten invariant}\label{prelim}
From now on we will forget about the base $B$, taking it to be a point for simplicity. It will reappear in Section \ref{SU}, where it will prove deformation invariance of our Vafa-Witten invariants.

\subsection{Localisation} The $\C^*$ action of weight 1 on the fibres of $X=K_S\to S$, corresponding to the grading \eqref{alg} on the sheaf of algebras $\pi_*\cO$, induces a $\C^*$-action on $\cN$, and the obstruction theory \eqref{pot} is naturally $\C^*$-equivariant. By \cite{GP} the $\C^*$-fixed locus $\cN^{\C^*}$ inherits a perfect obstruction theory 
\beq{potf}
\tau^{[-1,0]}\big(R\hom_{p\_X}(\EE,\EE)[2]\t^{-1}\big)^f\To\LL_{\cN^{\C^*}}
\eeq
by taking the fixed (weight 0) part $(\ \cdot\ )^f$ of \eqref{pot}.
%It has virtual dimension
%\beq{vd1}
%\vd\ =\ 2rc_2-(r-1)c_1^2-r^2\chi(\cO_S)+1.
%\eeq
This defines a virtual cycle
$$
\big[\cN^{\C^*}\big]^{\vir}\ \in\ H_*\big(\cN^{\C^*}\big).
$$
The derived dual of the moving (nonzero weight) part of \eqref{pot} is called the virtual normal bundle,
\begin{eqnarray} \nonumber
N^{\vir} &:=& \big(\tau^{[-1,0]}\big(R\hom_{p\_X}(\EE,\EE)[2]\t^{-1}\big)^{\mathrm{mov\;}}\big)^\vee\\
&=& \tau^{[0,1]}\big(R\hom_{p\_X}(\EE,\EE)[1]\big)^{\mathrm{mov\;}}. \label{Nvir}
\end{eqnarray}
Here the second expression follows from $\C^*$ equivariant Serre duality on $X$. If $\cN$ were compact, its perfect obstruction theory \eqref{pot} of virtual dimension zero would define a zero dimensional virtual cycle \cite{BF}. Its length would be a deformation invariant integer which we could calculate by the localisation formula \cite{GP} as
\beq{def1}
\int_{[\cN^{\C^*}]^{\vir}}\frac1{e(N^{\vir})}\,.
\eeq
Here we represent $N^{\vir}$ as a 2-term complex $\{V_0\to V_1\}$ of locally free $\C^*$-equivariant sheaves \emph{with nonzero weights} and define
$$
e(N^{\vir})\ :=\ \frac{c_{\mathrm{top}}^{\C^*}(V_0)}{c_{\mathrm{top}}^{\C^*}(V_1)}\ \in\ 
H^*(\cN^{\C^*}\!,\Z) \otimes_{\Z}\Q[t,t^{-1}],
$$
where $t=c_1(\t)$ is the generator of $H^*(B\C^*)=\Z[t]$ and $c_{\mathrm{top}}^{\C^*}$ denotes the $\C^*$-equivariant top Chern class lying in the \emph{localised} $\C^*$-equivariant cohomology
$$
H_{\C^*}^{*}(\cN^{\C^*}\!,\Z)\otimes_{\Z[t]}\Q[t,t^{-1}]\ \cong\ H^*(\cN^{\C^*}) \otimes\Q[t,t^{-1}].
$$
Since $\cN$ is noncompact but $\cN^{\C^*}$ is compact, we instead take the residue integral \eqref{def1} as a \emph{definition}. In this noncompact setting it is only a \emph{rational} number in general.

\begin{pdefn}\label{pdf}
Let $S$ be a smooth projective complex surface.
% with $h^{0,1}(S)=0=h^{0,2}(S)$
For $(r,c_1,c_2)$ for which all Gieseker semistable Higgs sheaves are Gieseker stable we define
$$
\UVW_{r,c_1,c_2}(S)\ :=\ \int_{[\cN^{\C^*}]^{\vir}}\frac1{e(N^{\vir})}\ \in\ \Q.
$$
\end{pdefn}

\noindent Note this is a \emph{constant} Laurent polynomial in $\Q[t,t^{-1}]$ because $\cN$ has virtual dimension zero. Equivalently, over each component of $\cN^{\C^*}$, the virtual dimension of the obstruction theory \eqref{potf} equals $-\rk(N^{\vir})$. (Both can vary from component to component, but they jump together.) $\UVW(S)$ is just a local DT invariant of $X$.

\begin{rmk}\label{nogood}\normalfont
This is only a sensible definition for surfaces with $h^{0,1}(S)=0=h^{0,2}(S)$. The obstruction sheaf of the obstruction theory \eqref{potf} has a trivial $H^2(\cO_S)$ summand, so the virtual cycle is zero whenever $h^{0,2}(S)>\nolinebreak0$. And when $h^{0,1}(S)>0$ the invariance of the obstruction theory under tensoring by flat line bundles means the integrand is pulled back from a lower dimensional space $\cN/\Jac(S)$, so the integral vanishes. 
\end{rmk}

For more general surfaces we would like to fix the determinant of $E$ and make $\phi$ trace-free, replacing the groups $\Ext^i_S(E,E),\ \Ext^i_S(E,E\otimes K_S)$ by their trace-free counterparts.\footnote{This is \textbf{not} the same as replacing $\Ext^i_X(\cE,\cE)$ by its trace-free version $\Ext^i_X(\cE,\cE)\_0$. Instead one should use the kernel of the map $\Ext^i_{\overline X}(\cE,\cE)\to H^{1,i+1}(\overline X)$ given by cup product with the Atiyah class on $\overline X$ of the pushforward of $\cE$ followed by the trace map; see for instance \cite[Equation 27]{KT1}. It is simpler to interpret this on $S$.} In gauge theory language, we want to replace $U(r)$ Higgs bundles with $SU(r)$. So instead of working on $X$ with the first term of (\ref{deftri}, \ref{deftri2}) we need to work on $S$ with the other two terms of (\ref{deftri}, \ref{deftri2}). Thus we need to relate the deformations of $\cE$ to those of $E$ and $\phi$. Equivalently, we need to express the deformation theory of $\cN$ in terms of that of $\curly M$ and the fibres of $\Pi\colon\cN\to\curly M$ \eqref{Pi}. This is what we do next.

\section{Deformation theory of the Higgs field}\label{more}
By Proposition \ref{funcAt}, the diagram
$$
\xymatrix{
R\hom_{p\_S}(\E,\E)[1] & R\hom_{p\_X}(\EE,\EE)[1] \ar[l]_{\pi_*} \\
\Pi^*\TT_{\curly M} \ar[u]^{\At_{\E,\cN}}& \TT_{\cN} \ar[l]^{\Pi_*}\ar[u]_{\At_{\EE,\cN}}
}$$
commutes. Dualising and taking cones gives, via the exact triangle \eqref{deftri2},
$$
\xymatrix{
R\hom_{p\_S}(\E,\E\otimes K_S)[1] \ar[r]\ar[d]_{\At_{\E,\cN}}& R\hom_{p\_X}(\EE,\EE)[2] \ar[d]^{\At_{\EE,\cN}} \ar[r]& R\hom_{p\_S}(\E,\E)[2] \ar[d]^{\At_{\EE\!,\,\cN/\curly M}} \\
\Pi^*\LL_{\curly M} \ar[r]^{\Pi^*}& \LL_{\cN} \ar[r]& \LL_{\cN/\curly M\,}.\!}\vspace{-17mm}
$$
\beq{ddgg}\vspace{8mm}\eeq
The right hand vertical arrow --- produced by filling in the cones --- is the projection of $\At_{\EE\!,\,\cN}$ from $\LL_{\cN}$ to $\LL_{\cN/\curly M}$, which factors through the top right hand term in the diagram since it is zero on the top left.

Now any fibre of $\Pi\colon\cN\to\curly M$ --- the space of (stable) Higgs fields $\phi$ for a fixed sheaf $E\in\curly M$ --- can be considered to be a space of quotients $\pi^*E\to\cE\to0$ of $\pi^*E$. 
That is, $\cN/\curly M$ is part\footnote{It is the Zariski open in Quot consisting of all quotients $\pi^*E\to\cE$ which induce an isomorphism $E\to\pi_*\;\cE$ on sections.} of a relative Quot scheme of $\pi^*\E$. Since the kernel of the quotient $\pi^*E\to\cE$ is $\pi^*E\otimes K_S^{-1}$ \eqref{resolution}, its deformations and obstructions are governed by
\beq{gov}
\Hom\_X(\pi^*E\otimes K_S^{-1},\cE)\ \cong\ \Hom\_S(E,E\otimes K_S)
\eeq
and
\beq{ext1ob}
\Ext^1_S(E,E\otimes K_S)
\eeq
respectively.
By Serre duality these are cohomologies of $(R\Hom(E,E)[2])^\vee$, so this is compatible with the last vertical arrow of the above diagram. So we need to check that the arrow is what we expect, i.e. that the diagram induces the usual obstruction theory for the relative Quot scheme $\cN/\curly M$.

\subsection{Deformations of quotients} \label{qdef}
Illusie describes a \emph{reduced} Atiyah class for quotients (or, more generally, maps of modules) and relates it to his Atiyah class of (\ref{LLLL}, \ref{atdef}). In our situation we again consider $\cN\times X[\EE]$ with its projection to $\cN\times X\big/\curly M\times X$. We now also consider $\cN\times X[\pi^*\E]$ with the embedding
$$
\cN\times X[\EE]\ \Into\ \cN\times X[\pi^*\E]
$$
induced from the surjection $\pi^*\E\to\EE\to0$. It has ideal
\beq{idealr}
\pi^*\E\otimes K_S^{-1},
\eeq
the kernel \eqref{resolution} of $\pi^*\E\to\EE$.

These maps of spaces induce a commutative diagram of
exact triangles
\beq{LLL}
\xymatrix@=15pt{
\LL_{\cN\times X[\EE]/\curly M\times X} \ar[r]\ar[d]&
\LL_{\cN\times X[\EE]/\cN\times X} \ar[r]\ar[d]&
\LL_{\cN\times X/\curly M\times X}[1] \ar[d] \\
\LL_{\cN\times X[\EE]/\curly M\times X[\pi^*\E]} \ar[r]&
\LL_{\cN\times X[\EE]/\cN\times X[\pi^*\E]} \ar[r]&
\LL_{\cN\times X[\pi^*\E]/\curly M\times X[\pi^*\E]}[1]}
\hspace{-4mm}
\eeq
on $X\times\cN[\EE]$. As in (\ref{LLLL}, \ref{atdef}), taking the degree 1 part of the pushdown to $\cN\times X$ of the right hand square gives
\beq{reddg}
\xymatrix@=15pt{
\EE \ar[r]\ar[d]& \EE\otimes\LL_{\cN/\curly M}[1] \ar@{=}[d] \\
\pi^*\E\otimes K_S^{-1}[1] \ar[r]& \EE\otimes\LL_{\cN/\curly M}[1],\!}
\eeq
where the bottom left hand term is the ideal \eqref{idealr}.
The lower horizontal arrow defines the reduced Atiyah class\footnote{Also known as the second fundamental form of $\pi^*\E\otimes K_S^{-1}\into\pi^*\E$.} of the quotient $\pi^*\E\to\EE$,
\begin{eqnarray} \label{redat}
\At^{\red}_{\Phi} &\in& \Hom\Big(\pi^*\E\otimes K_S^{-1},
\,\EE\otimes\LL_{\cN/\curly M}\Big) \\ \nonumber
&\cong& \Hom\Big(\E\otimes K_S^{-1},
\E\otimes\LL_{\cN/\curly M}\Big),
\end{eqnarray}
where the isomorphism follows from $\pi_{*\;}\EE\cong\E$.

\begin{prop} \label{CHF}
The right hand arrow of \eqref{ddgg} is $\At^{\red}_{\Phi}$.
\end{prop}

\begin{proof} 
This is just the statement that \eqref{reddg} commutes, since, by construction, the upper horizontal arrow is the Atiyah class $\At_{\EE\!,\,\cN/\curly M}$ of $\EE$.
\end{proof}

That is, the obstruction theory on $\cN/\curly M$ induced from considering $\cN$ to be a moduli space of sheaves $\cE$ on $X$ (i.e. the right hand arrow of \eqref{ddgg}) is the same as the standard obstruction theory for quotients $\pi^*E\to\cE\to0$ provided by the reduced Atiyah class.

\subsection{Deformations of Higgs fields}
The alternative description of any fibre of $\cN\to\curly M$ as a (linear!) space of Higgs fields $\phi\in\Hom(E,E\otimes K_S)$ means it has tangent space
\beq{tho}
\Hom(E,E\otimes K_S)
\eeq
and similarly an obstruction space
\beq{extob}
\Ext^1(E,E\otimes K_S).
\eeq
Putting these together in a family we will show --- in \eqref{arrow} for instance --- this description induces a perfect relative obstruction theory
\beq{arroh}
R\hom_{p\_S}(\E,\E\otimes K_S)^\vee\To\LL_{\cN/\curly M\;}.
\eeq
Of course (\ref{tho}, \ref{extob}) are the same as the groups (\ref{gov}, \ref{ext1ob}) we got by considering the deformation-obstruction theory of quotients $\pi^*E\to\cE_\phi\to0$. So to prove these are really the same obstruction theories we would like to show the arrow in \eqref{arroh} is $\At^{\red}_{\Phi}$ \eqref{redat}. Proposition \ref{CHF} would then imply that the  perfect obstruction theory given by the right hand arrow of \eqref{ddgg} is the same as the one induced by the linear structure on the fibres of $\cN/\curly M$. This latter description will allow us to pass more easily to the obstruction theory of trace-free Higgs fields (which is harder to describe from the Quot point of view).

\subsection{Matrices} We start by showing the equivalence of the two obstruction theories in the case of Higgs fields over a point, i.e. matrices. (Everything else we need will follow from this case by working in families.) So we fix \emph{vector spaces} $E$  and $\cL$ of ranks $r$ and $1$ respectively, and let $H$ denote the moduli space of Higgs fields $\phi$ on $E$. It is the vector space
$$
H\,:=\ \Hom(E,E\otimes \cL).
$$
By the spectral construction, each $\phi\in H$ is equivalent to a length-$r$ torsion sheaf $\cE_\phi$ on $\cL$ with an exact sequence
$$
0 \To \pi^*(E\otimes \cL^{-1}) \rt{\pi^*\phi-\eta} \pi^*E \rt{\ev} \cE_\phi \To 0.
$$
Here $\pi$ is the projection from $\cL$ to a point, and $\eta$ is the tautological section of $\pi^*\cL$ on $\cL$.

At any point $\phi\in H$ consider the maps
\beq{compo}
\xymatrix@=30pt{
H\,\cong\,\TT_H|_\phi \ar[r]^(.25){\At^{\red}_\phi}& \Hom\big(\pi^*(E\otimes \cL^{-1}),\cE\big)\,\cong\,\Hom(E,E\otimes \cL)\,=\,H,}\hspace{-2mm}
\eeq
where the the first isomorphism is the identification of a linear space with its tangent space. We want to show the composition \eqref{compo} is the identity to deduce that the two descriptions of $H$ --- as a linear space, or as a space of quotients --- give rise to the same description of the tangent space.

\begin{lem}\label{lemme}
The composition \eqref{compo} is the identity.
\end{lem}

\begin{proof}
Let $R$ denote the dual numbers $\C[t]/(t^2)$. Fix any tangent vector $\dot\varphi\in H\cong T_\phi H$, or equivalently a map $\Spec R\to H$. This corresponds to the family $\phi+t\dot\varphi$ of Higgs fields over $\Spec R$, and so the family of quotients
\beq{vary}
\xymatrix@=16pt{
0 \ar[r]& \pi^*(E\otimes \cL^{-1})\otimes R \ar[rrr]^(.56){\pi^*(\phi+t\dot\varphi)-\eta}&&& \pi^*E\otimes R \ar[r]& \cE_t \ar[r]& 0}
\eeq
defined by the cokernel of the second arrow.

Illusie \cite[Section IV.3.2]{Ill} shows any such family (flat over $R$) defines a deformation class in
\beq{class}
\Hom\big(\pi^*(E\otimes \cL^{-1}),\,\cE\otimes(t)\big)\ \cong
\ \Hom\big(\pi^*(E\otimes \cL^{-1}),\,\cE\big)
\eeq
described as follows. We compare the constant quotient (the right hand horizontal arrow $\ev\otimes1_R$ in the diagram below) with the varying quotient \eqref{vary} (the left hand horizontal arrow):
$$
\xymatrix@R=16pt@C=35pt{
\pi^*(E\otimes \cL^{-1})\cdot t \ar[d]\ar[rrrd]^0 &&& \cE_\phi\cdot t \ar[d] \\
\pi^*(E\otimes \cL^{-1})\otimes R \ar[rr]^(.57){\pi^*(\phi+t\dot\varphi)-\eta}\ar[d]&&
\pi^*E\otimes R \ar[r]& \cE_\phi\otimes R \ar[d] \\
\pi^*(E\otimes \cL^{-1}) \ar@{-->}[rrru]&&& \cE_\phi}
$$
Since $t^2=0$ the upper diagonal arrow is zero, so the horizontal composition induces the lower diagonal arrow. This has composition zero with the projection to $\cE_\phi$ so lifts uniquely to a map to $\cE_\phi\cdot t$.

Splitting $\cE_\phi\otimes R=\cE_\phi\cdot t\oplus\cE_\phi$ and $\pi^*E\otimes R=\pi^*E\cdot t\oplus\pi^*E$, we see from the diagram that this map is $\dot\varphi$. Since Illusie also shows the class is $\At^{\red}_\phi\cdot\,\dot\varphi$, this proves the Lemma.
\end{proof}

From this universal case, we can deduce the same result for any family of Higgs fields $\phi$ on the vector space $E$, parametrised by any (possibly singular) base $\calH$. Let $f\colon\calH\to H$ be its classifying map and let $\Phi$ denote the universal Higgs field.

\begin{cor} \label{core}
The map $$\At^{\red}_{f^*\Phi}\colon\TT_{\calH}\To f^*\big(R\Hom(E,E\otimes \cL)\otimes\cO_H\big)\,=\ H\otimes\cO_{\calH}$$ is the same as the derivative $$Df\colon\TT_{\calH}\To f^*\TT_H=H\otimes\cO_{\calH}.$$
\end{cor}

\begin{proof} We first observe this is true for the universal family $H$ itself (with $f,\ Df$ the identity). Since $H$ is smooth its tangent complex $\TT_{H}$ is a vector bundle, and to show an endomorphism of a vector bundle on a smooth space is the identity it is sufficient to check it on restriction to any point. Therefore Lemma \ref{lemme} proves the claim.

That is, $\At^{\red}_\Phi=\id$.
The obvious functoriality $f^*\At^{\red}_\Phi=\At^{\red}_{f^*\Phi}$ therefore gives $\At^{\red}_{f^*\Phi}=f^*\id=Df$, as required.
\end{proof}

\subsection{Higgs bundles} We can now deduce the same result in a \emph{family} of vector spaces, i.e. when $E$ is a vector bundle.

Suppose that $\curly M$ parametrises only \emph{vector bundles} on $S$ (or shrink it so that it does). Let $X$ be the total space of $K_S$ as usual.
%use $\pi$ to denote both projections $X\to S$ and any
%basechange of it such as $\curly M\times X\to\curly M\times S$.
From the universal bundle $\E$ on $\curly M\times S$ we form the vector bundle
$$
\sH\,:=\ \hom(\E,\E\otimes K_S)\rt\rho\curly M\times S.
$$
Over $\sH\times_{\curly M\times S}\curly M\times X\rt\pi\sH$ we get a universal Higgs field $\Phi$ and, by the spectral construction, a universal quotient
$$
0\To\pi^*(\rho^*\E\otimes K_S^{-1})\To\pi^*(\rho^*\E)\To\EE\To0.
$$
Using the smooth linear structure on the fibres of $\sH\to\curly M\times S$ gives the first isomorphism in the sequence of maps
\begin{multline} \label{compid}
\hspace{2mm}\xymatrix{
\rho^*\sH\ \cong\ \TT_{\sH/\curly M\times S}\ \ar[r]^(.4){\At^{\red}_\Phi}& 
\ \pi_*\hom(\pi^*(\rho^*\E\otimes K_S^{-1}),\,\EE)} \\
\cong\ \hom(\rho^*\E,\rho^*\E\otimes K_S)\ \cong\ \rho^*\sH.
\hspace{14mm}
\end{multline}

\begin{lem} \label{compl} 
The composition \eqref{compid} is the identity.
\end{lem}

\begin{proof}
Lemma \ref{lemme} shows that restricted to any closed point of $\sH$ the composition is the identity. If $\curly M$ is reduced this proves the claim.

In general it is sufficient to check the claim locally, since maps of sheaves (rather than complexes of sheaves) are equal if and only if they are equal locally. So, shrinking $\curly M$ and $S$ if necessary, we can assume that both $\E$ and $K_S$ are trivial. Then applying Corollary \ref{core} to $f$ the projection $\sH\cong H\times\curly M\times S\to H$ proves that \eqref{compid} is the identity.
\end{proof}

Thinking of $\cN/\curly M$ as a moduli space of sections of $\sH\to\curly M\times S$, the graph of the universal section $\Phi$ gives an embedding
\beq{aitch}
\xymatrix{
\cN\times S\ \ar@{^(->}[r]^(.53)\Phi& \,\Pi^*\sH,}
\eeq
where $\Pi\colon\cN\times S\to\curly M\times S$ is the projection. Its normal bundle is the fibrewise tangent bundle of $\Pi^*\sH\to\cN\times S$, which by the linear structure is just $\Pi^*\sH$:
\beq{just}
N_\Phi\ =\ \Cone\Big(\TT_{\cN\times S}\rt{D\Phi}\Phi^*\TT_{\Pi^*\sH}\Big)\ \cong\ \Phi^*\TT_{\Pi^*\sH/\cN\times S}\ \cong\ \Pi^*\sH.
\eeq
Therefore, by \cite[Proposition 6.2]{BF} applied to $(\curly M\times S)$-maps (i.e. sections) from $\curly M\times S$ to $\sH$ we get the linear relative obstruction theory for $\cN$ promised in \eqref{arroh} by starting with the map
\beq{arra}
\xymatrix@=30pt{
p_S^*\TT_{\cN/\curly M}\ \cong\ \TT_{\cN\times S/\curly M\times S} \ar[r]^-{D\Phi}& \Phi^*\TT_{\Pi^*\sH/\cN\times S}\ =\ N_\Phi\ \cong\ \Pi^*\sH}
\eeq
and applying adjunction:
\beq{arrow}
\TT_{\cN/\curly M}\,\To\,Rp\_{S*}(N_\Phi)\ \cong\ Rp\_{S*}\big(\Pi^*\sH\big)\ \cong\ R\hom_{p\_S}(\E,\E\otimes K_S).
\eeq

\begin{prop} \label{ni} The relative obstruction theory \eqref{arrow} is the same as the one given by the right hand arrow of \eqref{ddgg}.
\end{prop}

\begin{proof}
By Proposition \ref{CHF} the right hand arrow of \eqref{ddgg} is given by the reduced Atiyah class $\At^{\red}_\Phi$ of the universal quotient $\pi^*\E\to\EE_\Phi$ on $\cN\times X$. It is induced from the composition of the arrow \eqref{arra} with 
the pullback by $\Phi^*\Pi^*$ of \eqref{compid}. By Lemma \ref{compl} the latter is the identity, so we are done.
\end{proof}

\subsection{Trace}\label{tracesec}
Taking the trace of the section $\Phi$ \eqref{aitch} of $\sH$ we get a section $\tr\Phi$ of $K_S$ instead:
\beq{trsec}
\xymatrix{
\cN\times S\ \ar@{^(->}[r]^(.55)\Phi& \,\Pi^*\sH \ar[r]^(.4)\tr& \cN\times K_S.}
\eeq
Replacing the above analysis (\ref{just}, \ref{arra}, \ref{arrow}) for the section \eqref{aitch}  (the graph of $\Phi$) by the (simpler) analysis for the section \eqref{trsec} (the graph of $\tr\Phi$) gives a relative obstruction theory
$$
\TT_{\curly M\times H^0(K_S)/\curly M}\,\To\,Rp\_{S*}K_S
$$
for the moduli space $\curly M\times \Gamma(K_S)\to\curly M$ of pairs $(E,\sigma),\ \sigma\in \Gamma(K_S)$, together with a commutative diagram
\beq{Hitch}
\xymatrix{
R\hom_{p\_S}(\E,\E\otimes K_S) \ar[r]^-\tr& Rp\_{S*}K_S \\
\TT_{\cN/\curly M} \ar[r]^-{D(\tr\Phi)}\ar[u]^{\At^{\red}_\Phi}& 
(\tr\Phi)^*\TT_{\curly M\times \Gamma(K_S)/\curly M} \ar[u]} 
\eeq
compatible with the map
\begin{eqnarray}\label{tracemap}
\cN &\To& \curly M\times \Gamma(K_S), \\ \nonumber
(E,\phi) &\Mapsto& (E,\tr\phi).
\end{eqnarray}
Moreover, identifying $\TT_{\curly M\times\Gamma(K_S)/\curly M}$ with $\Gamma(K_S)\otimes\cO$ using the obvious linear structure, the construction of the right hand arrow of \eqref{Hitch} (i.e. the analogue of \eqref{just} for the section $\tr\Phi$) shows that it is the canonical embedding
$$
\Gamma(K_S)\otimes\cO\rt{H^0}R\Gamma(K_S)\otimes\cO\ \cong\ Rp\_{S*}K_S.
$$
Taking co-cones in \eqref{Hitch} gives the commutative diagram of exact triangles
$$
\hspace{7mm} \xymatrix@R=18pt@C=1pt{
R\hom_{p\_S}(\E,\E\otimes K_S)\_0 \ar[rrr]&&& R\hom_{p\_S}(\E,\E\otimes K_S) \ar[r]^-\tr& Rp\_{S*}K_S \\
\TT_{\cN/\curly M\times \Gamma(K_S)} \ar[u]^{\At^{\red}_0}\ar[rrr]&&& \TT_{\cN/\curly M} \ar[r]^-{\tr\Phi}\ar[u]^{\At^{\red}_\Phi}& 
(\tr\Phi)^*\TT_{\curly M\times \Gamma(K_S)/\curly M} \ar[u]}
\vspace{-15mm}$$
\beq{pech} \vspace{6mm} \eeq
where $\At^{\red}_0$ is the trace-free component of $\At^{\red}_\Phi$ in the splitting of the top row. The (dual of the) left hand map defines an obstruction theory by the following standard Lemma.

\begin{lem}\label{ABCL}
Given maps of spaces $A\to B\to C$ and a map of distinguished triangles on $A$,
$$
\xymatrix@R=18pt{
E\udot \ar[d]& F\udot \ar[l]\ar[d]& G\udot \ar[l]\ar[d] \\
\LL_{A/B} & \LL_{A/C} \ar[l]& \LL_{B/C}, \ar[l]}
$$
if the central and right hand vertical arrows are relative obstruction theories, then so is the left hand vertical arrow.
\end{lem}

\begin{proof} Taking sheaf cohomology we get the maps of long exact sequences
$$
\xymatrix@=15pt{
0& h^0(E\udot) \ar[l]\ar[d]& h^0(F\udot) \ar[l]\ar@{=}[d]& h^0(G\udot) \ar[l]\ar@{=}[d]& h^{-1}(E\udot) \ar[l]\ar[d]& h^{-1}(F\udot) \ar[l]\ar@{->>}[d]& h^{-1}(G\udot) \ar[l]\ar@{->>}[d] \\
0& \Omega_{A/B} \ar[l]& \Omega_{A/C} \ar[l]& \Omega_{B/C} \ar[l]& \LL^{-1}_{A/B} \ar[l]& \LL^{-1}_{A/C} \ar[l]& \LL^{-1}_{B/C}\,, \ar[l]}
$$
where $\LL^{-1}:=h^{-1}(\LL)$. The obstruction theories give isomorphisms on $h^0$ and surjections on $h^{-1}$ as shown by the arrows marked $\xymatrix@1@C=16pt{\ar@{=}[r]&}$ and $\xymatrix@1@C=16pt{\ar@{->>}[r]&}$ respectively. By a simple diagram chase, the first vertical arrow is an isomorphism and the fourth is a surjection.
\end{proof}

Fitting \eqref{pech} in with \eqref{ddgg} we get the diagram of exact triangles
$$
\hspace{7mm} \xymatrix@R=18pt@C=22pt{
R\hom_{p\_S}(\E,\E\otimes K_S)\_0 \ar@{<->}[r]\ar[d]&
R\hom_{p\_S}(\E,\E\otimes K_S) \ar@{<->}[r]^-\tr_-\id\ar[d]& R\Gamma(K_S) \ar@{=}[d] \\
R\hom_{p\_X}(\EE,\EE)^0[1] \ar@{<->}[r]\ar[d]&
R\hom_{p\_X}(\EE,\EE)[1] \ar@{<->}[r]\ar[d]& R\Gamma(K_S)\\
R\hom_{p\_S}(\E,\E)[1] \ar@{=}[r]& R\hom_{p\_S}(\E,\E)[1]}
\vspace{-23mm}$$
\beq{doc} \vspace{15mm} \eeq
(where $R\hom_{p\_X}(\EE,\EE)^0$ is defined by taking cones in the central row) receiving Atiyah class maps from the diagram
$$
\xymatrix@=18pt{\TT_{\cN/\curly M\times\Gamma(K_S)}
\ar[r]\ar[d]& \TT_{\cN/\curly M} \ar[r]\ar[d]&
\TT_{\Gamma(K_S)} \ar@{=}[d] \\
\TT_{\cN/\Gamma(K_S)} \ar[r]\ar[d]&
\TT_{\cN} \ar[r]\ar[d]& \TT_{\Gamma(K_S)} \\
\TT_{\curly M} \ar@{=}[r]& \TT_{\curly M}}
\vspace{-21mm}$$
\beq{doc2} \vspace{13mm} \eeq
with everything commutative. Along the top rows the maps are obstruction theories, as we just showed following \eqref{pech}. In the middle row the maps on the central and right hand terms are obstruction theories by Theorem \ref{BF} and Proposition \ref{ni} respectively. Therefore the map on the first terms of the middle rows
\beq{nite}
\TT_{\cN/\Gamma(K_S)}\To R\hom_{p\_X}(\EE,\EE)^0[1]
\eeq
is also an obstruction theory, by Lemma \ref{ABCL} again.

\subsection{Determinant}\label{detsec}
Having dealt with the trace of Higgs fields, we now also deal with the determinant of Higgs bundles on $S$.
This is more standard. We start with the left hand column of \eqref{doc}, making it the central column of the following commutative diagram of exact triangles
$$
\hspace{8mm}\xymatrix@=18pt{
R\hom_{p\_S}(\E,\E\otimes K_S)\_0 \ar@{=}[r]\ar[d]&
R\hom_{p\_S}(\E,\E\otimes K_S)\_0 \ar[d] \\
R\hom_{p\_X}(\EE,\EE)\_\perp[1] \ar@{<->}[r]\ar[d]&
R\hom_{p\_X}(\EE,\EE)^0[1] \ar@{<->}[r]\ar[d]& R\Gamma(\cO_S)[1] \ar@{=}[d] \\
R\hom_{p\_S}(\E,\E)\_0[1] \ar@{<->}[r]& R\hom_{p\_S}(\E,\E)[1] \ar@{<->}[r]& R\Gamma(\cO_S)[1].\!}
\vspace{-24mm}$$
\beq{doc3} \vspace{15mm} \eeq
Here the central row is defined in Corollary \ref{symme}. We then get the left hand column by taking co-cones. We claim this diagram is mapped to commutatively by the diagram of exact triangles\footnote{Here, as before, $\det\colon\cN\to\Pic(S)$ is the map which on points takes $(E,\phi)\mapsto\det E$.}
$$
\xymatrix@=18pt{
\TT_{\cN/\curly M\times\Gamma(K_S)} \ar@{=}[r]\ar[d]&
\TT_{\cN/\curly M\times\Gamma(K_S)} \ar[d] \\
\TT_{\cN/\Gamma(K_S)\times\Pic(S)} \ar[r]\ar[d]&
\TT_{\cN/\Gamma(K_S)} \ar[r]^{\det_*}\ar[d]& \TT_{\Pic(S)} \ar@{=}[d] \\
\TT_{\curly M/\Pic(S)} \ar[r]& \TT_{\curly M} \ar[r]^{\det_*}& \TT_{\Pic(S)}.\!}
\vspace{-23mm}$$
\beq{doc4} \vspace{15mm} \eeq
The maps between central columns come from the Atiyah class maps between \eqref{doc2} and \eqref{doc}. Between the right hand columns we similarly take the Atiyah class of the line bundle $\det\EE$. These give a commutative square by (the dual of) \cite[Proposition 3.2]{STV}. Taking co-cones gives the maps on the left hand column.

The splitting of the central row of \eqref{doc3} defines a splitting of the Atiyah class $\At_{\EE\!,\,\cN}$ into components $\big(\!\At_{\EE\!,\,\cN}^\perp\,,\,\At_{\det\EE\!,\,\cN}\big)$.

\begin{prop} \label{close}
Over the open locus of Higgs \emph{bundles} in $\cN$, the map
$$
\At_{\EE\!,\,\cN}^\perp\colon R\hom_{p\_X}(\EE,\EE)\_\perp[2]\t^{-1}\To\LL_{\cN/\Gamma(K_S)\times\Pic(S)}
$$
defined above is a 2-term symmetric perfect relative obstruction theory for $\cN\big/\Gamma(K_S)\times\Pic(S)$.
\end{prop}

\begin{proof}
Consider the maps between the middle rows of \eqref{doc3} and \eqref{doc4}. Between the central terms the map is the obstruction theory \eqref{nite}. Between the right hand terms it is trivially an obstruction theory. Therefore the map between the left hand terms is also a (relative) obstruction theory by Lemma \ref{ABCL}.

That $R\hom_{p\_X}(\EE,\EE)\_\perp[2]$ is perfect of amplitude $[-1,0]$ follows from basechange just as in Corollary \ref{2ter}. Its symmetry was noted in Corollary \ref{symme}. Recording the $\C^*$ action explicitly as in Corollary \ref{2ter} explains the $\t^{-1}$. 
\end{proof}

%The map on top left corners plus the maps on the central column defines commuting maps on the whole top left squares. Taking cones defines the map on the right hand columns. In particular 
%$$
%R\hom_{p\_X}(\EE,\EE)\_{\perp}[1]
%$$
%is \emph{defined} as a cone by the central row. Its map to $\LL_{\cN^\perp/\Pic(S)}$ defines a perfect obstruction theory since all the other maps do.
%
%Let
%$$
%\cN^\perp\ \subset\ \cN
%$$
%be the fibre over 0 of \eqref{tracemap}, i.e. the moduli space of of Higgs bundles $(E,\phi)$ with \emph{trace-free} $\phi\in\Hom(E,E\otimes K_S)\_0$\;. Restricting the right hand arrow of \eqref{pech} to it gives a perfect relative obstruction theory
%$$\xymatrix{
%\TT_{\cN^\perp/\curly M} \ar[r]^(.33){\At_0^{\red}}& R\hom_{p\_S}(\E,\E\otimes K_S)\_0\,.}
%$$
%Dualising \eqref{pech} (thus swapping $\id$ and $\tr$), using the splitting of its top row, and combining with \eqref{ddgg} gives the following.
%
%\begin{cor}\label{-K} Consider the composition
%\begin{align*}
%R\hom_{p\_X}(\EE,\EE)[2]\rt{\pi_*}R\hom_{p\_S}(\E,\E)[2]\rt{\tr}
%&R\hom_{p\_S}\cO[2] \\ \cong\ &R\Gamma(\cO_S)\otimes\cO_{\cN}[2]
%\end{align*}
%on $\cN$. On restriction to $\cN^\perp\subset\cN$ its co-cone maps to $\LL_{\cN^\perp}$ to give a perfect obstruction theory
%$$
%\mathrm{Cone}\Big(R\hom_{p\_X}(\EE,\EE)\To R\Gamma(\cO_S)\otimes\cO_{\cN^\perp}\Big)[2]\To\LL_{\cN^\perp}. \vspace{-7mm}
%$$
%$\hfill\square$\vspace{2mm}
%\end{cor}

\subsection{Higgs sheaves}
When the rank of the vector space jumps in the family --- i.e. when $E$ is a sheaf rather than a bundle --- things are more complicated and we need to replace sheaves by locally-free resolutions.

We first shrink $\curly M$ to the image of $\Pi\colon\cN\to\curly M$. This ensures boundedness since $\cN$ is quasi-projective.
As $\curly M$ parametrises only torsion-free sheaves, these all have homological dimension $\le1$ on $S$. Therefore the universal sheaf $\E$ also has homological dimension $\le1$.

Fix an ample line bundle $\cO(1)$ on $S$ (and use the same notation $\cO(1)$ for its pullbacks to $\curly M\times S$ and $\cN\times S$ as usual). Then for sufficiently large $N\gg0$ we get a surjection
\beq{surge}
\E_1:=\big(p_S^*p\_{S*}\E(N)\big)(-N)\rt{\ev}\E\To0 \qquad\mathrm{on}\ \curly M\times S.
\eeq
Since $\E$ is flat over $\curly M$, for $N\gg0$ we deduce that this $\E_1$ is locally free. Letting $\E_2$ be the kernel of \eqref{surge}), it is also locally free since $\E$ has homological dimension $\le1$. So from now on we fix this two-term locally free resolution 
\beq{rezz}
0\To \E_2\rt{d}\E_1\To\E\To 0.
\eeq
Its advantage is its functoriality, which implies that any (twisted, Higgs field) endomorphism of $E$ induces compatible endomorphisms of the $E_i$. The universal case is over $\cN/\curly M$, where the universal Higgs field $\Phi$ on (the pullback to $\cN\times S$ of) $\E$ induces canonical maps
$$
\xymatrix{
0 \ar[r]& \ker(\ev) \ar[r]^(.33)d\ar[d]^{\Phi_2}& \big(p_S^*p\_{S*}\E(N)\big)(-N) \ar[r]^(.75){\ev}\ar[d]_{\Phi_1:=}^{p_S^*p\_{S*}\Phi}& \E \ar[r]\ar[d]^\Phi& 0 \\
0 \ar[r]& \ker(\ev) \ar[r]^(.33)d& \big(p_S^*p\_{S*}\E(N)\big)(-N) \ar[r]^(.75){\ev}& \E \ar[r]& 0,\!}
$$
i.e. $\Phi_i\in\Hom(\E_i,\E_i\otimes K_S)$ such that
\beq{eq}
d\circ\Phi_2=\Phi_1\circ d.
\eeq
Moreover any maps $\Phi_i$ satisfying \eqref{eq} induce a map $\Phi$ which is \emph{unique}: other choices differ by an element of
\beqa
\Hom(\E_1,\E_2) &=& \Hom\big((p_S^*p\_{S*}\E(N))(-N),\E_2\big) \\
&\cong& \Hom\big(p\_{S*}(\E(N)),\,p\_{S*}(\E_2(N))\big)\ =\ 0,
\eeqa
since the choice of $\E_1$ ensures that $p\_{S*}(\E_2(N))=0$. We can phrase this as saying that $\cN$ is cut out of
\beq{embd}
\cN\ \subset\ \cN_1\times_{\curly M}\cN_2,
\eeq
by the equation \eqref{eq}. Here $\cN_i\to\curly M$ denotes the moduli space of Higgs fields $\phi_i$, with fibre over $E\in\curly M$ given by $\Hom(E_i,E_i\otimes K_S)$ (where the $E_i$ are the restriction to $S\times\{E\}$ of the $\E_i$ \eqref{rezz}).
\medskip

Using \eqref{rezz}, $R\hom(\E,\E)$ is computed by the total complex of
\beq{total}\xymatrix{
\E_1^*\otimes \E_2 \ar[r]^(.48){d^*\otimes1}\ar[d]_{1\otimes d}& \E_2^*\otimes \E_2 \ar[d]^{1\otimes d} \\ 
\E_1^*\otimes \E_1 \ar[r]^(.48){d^*\otimes1}&
\E_2^*\otimes \E_1.\!}
\eeq
Since the top left corner injects into $\E_1^*\otimes \E_1$ this is quasi-isomorphic to
$$
\frac{\E_1^*\otimes \E_1\ \oplus\ \E_2^*\otimes \E_2}{(1\otimes d,\,d^*\otimes1)\E_1^*\otimes \E_2}
\xymatrix@=60pt{\ar[r]^(.38){(d^*\otimes1,\,-1\otimes d)}&
\E_2^*\otimes \E_1}\vspace{-3mm}
$$
$$\hspace{-3cm}\xymatrix{
\ar[d]<1ex>^{\tr:=(\tr,-\tr)} \\
\cO_{\curly M\times S} \ar[u]^{\id:=(1,1)}}
$$
with the given trace and identity maps. The identity $\tr\circ\id=\rk\ne0$ gives the usual splitting
$$
R\hom(E,E)\ \cong\ R\hom(E,E)\_0\,\oplus\,\cO_{\curly M\times S}.
$$

Over $\cN$ the universal Higgs field $\Phi$ defines a section $\tr\Phi$:
$$
\cN\times S\rt\Phi\hom(\E,\E\otimes K_S)\rt{\mathcal H^0}R\hom(\E,\E\otimes K_S)\rt{\tr}K_S,
$$
where as usual we have abbreviated $\pi^*\E$ to $\E$ and also suppressed the pullback on $K_S$. So we have all the same ingredients as we had in the Higgs bundles case, except we have not yet shown that the composition
\beq{one}\xymatrix{
%p_S^*\TT_{\cN/\curly M}\ \cong\ 
\TT_{\cN\times S/\curly M\times S} \ar[r]^(.42){\At^{\red}}& R\hom(\E,\E\otimes K_S) \ar[r]^(.7){\tr}& K_S}
\eeq
is the same as\footnote{In other words, we will not extend all of Proposition \ref{ni} from bundles to sheaves to give the full equivalence of the two natural perfect relative obstruction theories for $\cN/\curly M$, as it is enough for our purposes to deal only with the trace.} 
\beq{two}\xymatrix@=36pt{
\TT_{\cN\times S/\curly M\times S} \ar[r]^(.6){D(\tr\Phi)}& K_S.}
\eeq
\medskip

To prove the equality of \eqref{one} and \eqref{two}
we repeat our earlier analysis for Higgs fields on vector bundles to both $(\E_1,\phi_1)$ and $(\E_2,\phi_2)$. Section \ref{qdef} gives reduced Atiyah classes \eqref{redat}
$$
\At^{\red}_{\phi_i}\in R\Hom\Big(\E_i\otimes K_S^{-1},\E_i\otimes\LL_{\cN_i/\curly M}\Big), \quad i=1,2,
$$
which we consider as maps
\beq{i}
\At^{\red}_{\phi_i}\colon p_S^*\TT_{\cN_i/\curly M}\To R\hom(\E_i\otimes K_S^{-1},\E_i)\ \cong\ \E_i^*\otimes\E_i\otimes K_S.
\eeq
They determine the reduced Atiyah class $\At^{\red}_\Phi$ of $\E$ as follows.

\begin{prop}
Using the resolution \eqref{rezz} to write $R\hom$ as a total complex \eqref{total}, the reduced Atiyah class of $\E$,
\beq{1}\xymatrix{
p_S^*\TT_{\cN/\curly M}\ \ar[rr]^-{\At^{\red}_\Phi}&& \ R\hom(\E\otimes K_S^{-1},\E).}\hspace{6mm}
\eeq
becomes\vspace{-7mm}
\beq{2}\begin{array}{c}
\\ \xymatrix@=30pt{
p_S^*\TT_{\cN/\curly M}\ 
\ar[rr]^-{\left(\begin{array}{cc} \scriptstyle{0} & \!\!\scriptscriptstyle{\At^{\red}_{\phi_2}}\!\!\!\! \\ \!\!\scriptscriptstyle{\At^{\red}_{\phi_1}}\!\!\!& \scriptstyle{0}\end{array}\right)}&&} \\ \end{array}\!\!{\footnotesize\xymatrix{
\E_1^*\otimes \E_2\otimes K_S \ar[r]^(.48){d^*\otimes1}\ar[d]^{1\otimes d}& \E_2^*\otimes \E_2\otimes K_S \ar[d]^{1\otimes d} \\ 
\E_1^*\otimes \E_1\otimes K_S \ar[r]^(.48){d^*\otimes1}&
\E_2^*\otimes \E_1\otimes K_S,\!}}
\eeq
where $\At_{\phi_i}^{\red}$ denotes the reduced Atiyah class \eqref{i} restricted to $\cN$ by \eqref{embd}.
\end{prop}

Note that \eqref{2} really defines a map in the derived category to the total complex on the right hand side because $d\circ\At^{\red}_{\phi_2}=\At^{\red}_{\phi_1}\circ\,d$ by \eqref{eq} and the functoriality of Atiyah classes under the pullback to $\cN\subset\cN_1\times_{\curly M}\cN_2$.

\begin{proof}
We follow the definition of reduced Atiyah class in \eqref{LLL}. The maps of sheaves on $\cN\times X$,
$$
\xymatrix@R=18pt{
\pi^*\E_2 \ar[r]\ar[d]& \pi^*\E_1 \ar[r]\ar[d]& \pi^*\E \ar[d]\\
\cE_2 \ar[r]& \cE_1 \ar[r]& \EE,\!}
$$
induce maps of spaces
$$\xymatrix@R=18pt{
\cN\times X[\pi^*\E_2] &\ar[l] \cN\times X[\pi^*\E_1] &\ar[l]
\cN\times X[\pi^*\E] \\
\cN\times X[\cE_2] \ar[u]&\ar[l]\ar[u] \cN\times X[\cE_1] &\ar[l]\ar[u] \cN\times X[\EE],\!}
$$
where the vertical maps are embeddings with ideals $\pi^*\E_2\otimes K_S^{-1},\,\pi^*\E_1\otimes K_S^{-1}$ and $\pi^*\E\otimes K_S^{-1}$ respectively. These induce maps of cotangent complexes
{\small
$$\xymatrix@=13pt{
\LL_{\cN\times X[\cE_2]/\cN\times X[\pi^*\E_2]} \ar[r]\ar[d]& \LL_{\cN\times X[\cE_1]/\cN\times X[\pi^*\E_1]} \ar[r]\ar[d]&
\LL_{\cN\times X[\EE]/\cN\times X[\pi^*\E]} \ar[d] \\
\LL_{\cN\times X[\pi^*\E_2]/\curly M\times X[\pi^*\E_2]}[1]\! \ar[r]& \!\LL_{\cN\times X[\pi^*\E_1]/\curly M\times X[\pi^*\E_1]}[1]\! \ar[r]& \!\LL_{\cN\times X[\pi^*\E]/\curly M\times X[\pi^*\E]}[1].\!\!}
$$}
Taking weight 1 parts of their pushdowns to $\cN\times X$ gives
$$\xymatrix@R=18pt{
\pi^*\E_2\otimes K_S^{-1}[1] \ar[r]\ar[d]^{\At^{\red}_{\phi_2}}& \pi^*\E_1\otimes K_S^{-1}[1] \ar[r]\ar[d]^{\At^{\red}_{\phi_1}}& \pi^*\E\otimes K_S^{-1}[1] \ar[d]^{\At^{\red}_\Phi} \\
\LL_{\cN/\curly M}\otimes\cE_2[1] \ar[r]& 
\LL_{\cN/\curly M}\otimes\cE_1[1] \ar[r]& 
\LL_{\cN/\curly M}\otimes\EE[1],\!}
$$
which is equivalent, by adjunction and $\pi_*\;\cE_i=E_i,\ \pi_*\;\EE=\E$, to the following diagram of horizontal exact triangles on $\cN\times S$:
$$\xymatrix@R=18pt{
\TT_{\cN/\curly M}\otimes \E_2 \ar[r]\ar[d]^{\At^{\red}_{\phi_2}}& \TT_{\cN/\curly M}\otimes \E_1 \ar[r]\ar[d]^{\At^{\red}_{\phi_1}}& \TT_{\cN/\curly M}\otimes\E \ar[d]^{\At^{\red}_\Phi} \\
\E_2\otimes K_S \ar[r]& \E_1\otimes K_S \ar[r]& \E\otimes K_S.\!}
$$
The commutativity of this diagram gives the claimed result.
\end{proof}

In particular
\beqa
\tr\circ\At^{\red}_\Phi &=& \tr\circ\At^{\red}_{\phi_1}-\tr\circ\At^{\red}_{\phi_2} \\
&=& D\big(\!\tr(\phi_1)\big)-D\big(\!\tr(\phi_2)\big)\ =\ D(\tr\Phi)
\eeqa
by applying (\ref{Hitch}, \ref{tracemap}) to both Higgs \emph{bundles} $(E_i,\phi_i)$. Hence \eqref{one} equals \eqref{two} and Sections \ref{tracesec} and \ref{detsec} go through for Higgs sheaves as well as bundles. Thus Proposition \ref{close} holds for Higgs sheaves too:

\begin{thm} \label{thear}
There is a 2-term symmetric perfect relative obstruction theory
$$
\At_{\EE\!,\,\cN}^\perp\colon R\hom_{p\_X}(\EE,\EE)\_\perp[2]\t^{-1}\To\LL_{\cN/\Gamma(K_S)\times\Pic(S)}
$$
for $\cN\big/\Gamma(K_S)\times\Pic(S)$. \hfill$\square$
\end{thm}

\section{The $SU(r)$ Vafa-Witten invariant}\label{SU}
Restricting the relative obstruction theory of Theorem \ref{thear} to the fibre $\cN_L^\perp$ over
$$
(L,0)\ \in\ \Pic(S)\times\Gamma(K_S)
$$
we get an absolute obstruction theory.

\begin{thm}\label{final}
The moduli space $\cN_L^\perp$ of stable Higgs sheaves $(E,\phi)$ with $\det E\cong L$ and trace-free $\phi\in\Hom(E,E\otimes K_S)\_0$ admits a 2-term symmetric perfect obstruction theory
$$
R\hom_{p\_X}(\EE,\EE)\_{\perp\,}[1]\t^{-1}\To\LL_{\cN_L^\perp}\,.
\vspace{-8mm}
$$
$\hfill\square$
\end{thm}

Localising as in Section \ref{prelim}, we can now give a general definition of our Vafa-Witten invariant for any surface $S$ (which agrees with Preliminary Definition \ref{pdf} when $h^{0,1}(S)=0=h^{0,2}(S)$).

\begin{defn}\label{SUdef}
Let $S$ be a smooth projective complex surface, and fix $(r,c_1,c_2)$ with $r>0$ for which all Gieseker semistable Higgs sheaves are Gieseker stable. We define
\beq{defvw}
\VW_{r,c_1,c_2}(S)\ :=\ \int_{\big[(\cN_{r,L,c_2}^\perp(S))^{\C^*}\big]^{\vir\ }}\frac1{e(N^{\vir})}\ \in\ \Q,
\eeq
where $L$ is any line bundle on $S$ with $c_1(L)=c_1$.
\end{defn}

\begin{rmk}\normalfont This is deformation invariant under deformations of $S$ for which $c_1$ remains of type (1,1). More precisely, suppose $\cS\to B$ is a smooth family of projective surfaces over a connected base $B$ with a global class $c_1\in H^{1,1}(\cS)$ whose restriction to any fibre $\cS_b,\ b\in B,$ we also denote by $c_1\in H^{1,1}(\cS_b)$.\footnote{Equivalently the classifying map from $B$ to the moduli stack of surfaces $S$ has image in the \emph{Noether-Lefschetz locus} of $(S,c_1)$.} Let $\cL$ be any line bundle over $\cS$ with $c_1(\cL)=c_1$. Then just as in Section \ref{potsec} we may do everything relative to $B$ so that $\cN_{\cL}^\perp\to B$ has a perfect relative obstruction theory, inducing one on its $\C^*$-fixed locus by \cite{GP}. Thus
$$
\int_{\big[(\cN_{\cL_b}^\perp(\cS_b))^{\C^*}\big]^{\vir\ }}\frac1{e(N^{\vir})}
$$
is independent of $b\in B$ by \cite[Proposition 7.2]{BF} and conservation of number \cite[Theorem 10.2]{Fu}.

The invariant need \emph{not} be deformation invariant, however, under deformations of $S$ for which the Hodge type of $c_1(L)$ does not remain of Hodge type $(1,1)$.
In this situation there are no sheaves with $\det E=L$ on the deformed surface, so the invariant becomes zero.
\end{rmk}

\subsection{More general Higgs pairs} \label{later}
In Theorem \ref{final} we can replace $K_S$ by any line bundle $\cL\to S$ with
$\deg \cL\ge\deg K_S$
at the expense of dropping the word ``\emph{symmetric}". That is, we consider the moduli space $\cN_{r,L,c_2}^\perp$ of stable $\cL$-Higgs pairs
$$
(E,\phi),\quad\det E\,\cong\,L,\quad \phi\in\Hom(E,E\otimes \cL)\_0
$$
on $S$ of rank $r$ and second Chern class $c_2$. Equivalently, it is the moduli space of stable torsion sheaves $\cE_\phi$ on the total space
$$
X\ =\ \mathrm{Tot}\,\cL\To S
$$
with centre of mass zero on the fibres of $X\to S$ and $\det\pi_*\;\cE_\phi\cong L$.

\begin{thm} \label{ffinal}
Fix $(r,L,c_2)$ such that semistability implies stability. Suppose $\cL$ satisfies
\begin{itemize}
\item $\deg \cL>\deg K_S$, or
\item $\deg \cL=\deg K_S$ but $\cL^{\otimes r}\not\cong K_S^{\otimes r}$, or
\item $\cL=K_S$.
\end{itemize}
Then $\cN_{r,L,c_2}^\perp$ admits a 2-term perfect obstruction theory
$$
R\hom_{p\_X}(\EE,\EE\otimes K_X)\_{\perp\,}[1]\t^{-1}\To\LL_{\cN_{r,L,c_2}^\perp}\,.
$$
\end{thm}

We briefly sketch the modifications that have to be made to the proof of Theorem \ref{final} to derive this. Proposition \ref{symmetric} is the first place we set $\cL=K_S$. For general $\cL$ it is modified to say that the exact triangle \eqref{deftri2}
$$
\xymatrix@C=30pt{
R\hom_{p\_X}(\EE,\EE) \ar[r]^{\pi_*}& R\hom_{p\_S}(\E,\E) \ar[r]^(.42){[\ \cdot\ ,\,\phi]} & R\hom_{p\_S}(\E,\E\otimes \cL)}
$$
is Serre dual (up to a shift) to the exact triangle
\begin{multline*}
\xymatrix@C=25pt{
R\hom_{p\_X}(\EE,\EE\otimes K_{X/B}) \ar[r]^-{\pi_*}& R\hom_{p\_S}(\E,\E\otimes K_{S/B}\otimes \cL^{-1})} \\
\xymatrix@C=35pt{\ar[r]^-{[\ \cdot\ ,\,\phi]} & R\hom_{p\_S}(\E,\E\otimes K_{S/B})}
\end{multline*}
made in the same way: by applying $Rp\_{X*}$ to the exact triangle \eqref{locv} and its twist by $K_{X/B}\cong\pi^*(K_{S/B}\otimes \cL^{-1})$ respectively.

Corollary \ref{symme} then gets modified by replacing $K_{S/B}$ by $\cL$. In the proof we note the triangle \eqref{OM2} can be twisted by the line bundle $K_X\cong \cL^{-1}\otimes K_S$ to become
$$
\cL^{-1}K_S\!\rt{\!\id\!}\!R\hom_\pi(\EE,\EE\otimes K_X)\!\rt{\!\!\pi_*\!\!}\!R\hom(\E,\E\otimes \cL^{-1}K_S)\!\rt{\!\tr\!}\!\cL^{-1}K_S,
$$
allowing us to replace \eqref{OM} by 
$$
R\hom_{p\_X}(\EE,\EE)\longleftrightarrow Rp\_{S*\,}\cL[-1].
$$
This allows us to remove $Rp\_{S*\;}\cL[-1]\oplus Rp\_{S*\;}\cO_S$ from $R\hom_{p\_X}(\EE,\EE)$ to form $R\hom_{p\_X}(\EE,\EE)\_\perp$ much as before. Serre duality inserts an extra $K_X$ in both \eqref{att},
\beq{last}
\At_{\EE\!,\,\cN}\colon R\hom_{p\_X}(\EE,\EE\otimes K_X)\t^{-1}[2]\To\LL_{\cN/B},
\eeq
and \eqref{pot}. To make Corollary \ref{2ter} run we have to ensure the trace-free part of the left hand side of \eqref{last} is 2-term. To do this we modify \eqref{ext3} by
\beq{ghy}
\Ext^3_X(\EE_t,\EE_t)\ \cong\ \Hom\_X(\EE_t,\EE_t\otimes K_X)\ =\ 0
\eeq
when $\cL\ne K_S$. When $\deg \cL>\deg K_S$ then $\deg K_X<0$, so \eqref{ghy} follows from stability. When $\deg \cL=\deg K_S$, any nonzero element of $\Hom\_X(\EE_t,\EE_t\otimes K_X)$ must be an isomorphism by stability. By pushdown we get an isomorphism $\E_t\to\E_t\otimes \cL^{-1} K_S$ whose determinant is an isomorphism $\cL^{\otimes r}\to K_S^{\otimes r}$, contradicting our assumption and so proving \eqref{ghy}.

The rest of the proof of Theorem \ref{final} can then be copied almost verbatim, with a few twists by line bundles, to deduce Theorem \ref{ffinal}.

\section{Vanishing theorem and the two fixed loci}\label{fixsec}

The integral \eqref{defvw} is over two types of fixed components, as mentioned in (\ref{cpt1}, \ref{cpt2}).

\subsection{The first fixed locus}
The instanton branch has $\phi=0$ and so gives the moduli space $\cM_L$ of Gieseker stable\footnote{Recall we are fixing $(r,c_1,c_2)$ for which stability\,=\,semistability.} sheaves on $S$ of fixed determinant $L$. Consider the dual of the top row of the diagram of Corollary \ref{symme}. Its coboundary map is $[\ \cdot\ ,\phi]$ by \eqref{deftri2} which vanishes when $\phi=0$. Therefore the exact triangle splits the obstruction theory of Theorem \ref{final}
$$
R\hom_{p\_X}(\EE,\EE)\_\perp[2]\t^{-1}\ \cong\ R\hom_{p\_S}(\E,\E\otimes K_S)\_0[1]\ \oplus\ R\hom_{p\_S}(\E,\E)\_0[2]\t^{-1}
$$
into fixed and moving parts. By \cite{GP} the former defines a perfect obstruction theory for $\cM_L$,
\beq{fixie}
E\udot\,:=\ R\hom_{p\_S}(E,E\otimes K_S)\_0[1]\,\To\,\LL_{\cM_L}.
\eeq
Here $E\udot$ is quasi-isomorphic to a 2-term complex of $\C^*$-fixed locally free sheaves $E^{-1}\to E^0$. Similarly, dualising the moving parts gives the virtual normal bundle
$$
N^{\vir}\ =\ R\hom_{p\_S}(E,E\otimes K_S)\_0\;\t\ =\ E\udot\otimes\t[-1].
$$
Therefore the contribution of the fixed locus $\cM_L$ to the invariant \eqref{def1} is
\beqa
\int_{[\cM_L]^{\vir}}\frac1{e(N^{\vir})} &=& \int_{[\cM_L]^{\vir}}\frac{c^{\C^*}_s\!(E^0\otimes\t)}{c^{\C^*}_r\!(E^{-1}\otimes\t)} \\
&=& \int_{[\cM_L]^{\vir}}\frac{c_s(E^0)+tc_{s-1}(E^0)+\ldots}{c_r(E^{-1})+tc_{r-1}(E^{-1})+\ldots}\,,
\eeqa
where $r$ and $s$ are the ranks of $E^{-1}$ and $E^0$ respectively.
The integrand is homogeneous of degree $s-r$ equal to the virtual dimension
\beq{vd}
\vd\ =\ 2rc_2-(r-1)c_1^2-(r^2-1)\chi(\cO_S)
\eeq
of the perfect obstruction theory \eqref{fixie} for $\cM_L$. Therefore only the $t^0$ coefficient has the correct degree $\vd$ over $\cM_L$ to have nonzero integral against $[\cM_L]^{\vir}$, and we may set $t=1$ in the above to give
$$
e_2(\cM_L)=\int_{[\cM_L]^{\vir}}\left[\frac{c_\bullet(E^0)}{c_\bullet(E^{-1})}\right]_{\vd}\,.
$$
Here $c_\bullet(\,\cdot\,)$ denotes the total Chern class. This is 
\beq{vircot}
\int_{[\cM_L]^{\vir}}c_{\vd}\big(E\udot\big)\ \in\ \Z,
\eeq
the top (in a virtual, or derived sense) Chern class of the virtual cotangent bundle $E\udot$ of $\cM_L$. This is the Ciocan-Fontanine-Kapranov/Fantechi-G\"ottsche
signed Euler characteristic of $\cM_L$ studied in \cite{JT}.

\subsection{Vanishing theorem}
Under certain positive curvature hypotheses, the above integer is the \emph{only} contribution to $\VW_{r,c_1,c_2}(S)$.

\begin{prop} \label{vaneul}
If $\deg K_S\le0$ then any stable $\C^*$-fixed Higgs pair $(E,\phi)$ has Higgs field $\phi=0$. If $r,c_1,c_2$ are chosen so that semistability implies stability, then $\VW_{r,c_1,c_2}(S)$ is the signed virtual Euler characteristic \eqref{vircot} of the moduli space $\cM_L(S)$ of stable sheaves on $S$ with fixed determinant $L$.
\end{prop}

\begin{proof}
Since both $\ker\phi$ and $\im\phi$ are $\phi$-invariant subsheaves (of $E$ and $E\otimes K_S$ respectively) Gieseker stability gives the following inequalities on reduced Hilbert polynomials,
\beq{lt}
p\_E(n)\ <\ p\_{\;\im\phi}(n)\ <\ p\_{E\otimes K_S}(n) \quad \forall\,n\gg0,
\eeq
unless $\im\phi$ is 0 or all of $E\otimes K_S$. But $\deg K_S\le0$ implies that $p\_{E\otimes K_S}(n)\le p\_E(n)$ for $n\gg0$. So either $\phi=0$ or $\phi$ is an isomorphism. Since $\phi$ is $\C^*$-fixed it has determinant 0, so it cannot be an isomorphism.
\end{proof}

\medskip\noindent\textbf{Remark.}
If either $\deg K_S<0$ or $K_S\cong\cO_S$ then for any $E\in\cM_L$,
$$
\Ext^2(E,E)\_0\ \cong\ \Hom(E,E\otimes K_S)^*_0\ =\ 0.
$$
In the first case this is because $E$ is stable and $\deg E\otimes K_S<\deg E$ so the only map $E\to E\otimes K_S$ is zero. In the second case it is because stable sheaves $E$ are simple. Therefore the obstruction space vanishes at any $E\in\cM_L$, so $\cM_L$ is smooth of dimension $\vd$ \eqref{vd}. In particular, the signed virtual Euler characteristic \eqref{vircot} is just the signed \emph{topological} Euler characteristic
$$
\VW_L\ =\ \int_{\cM_L}c_{\vd}\big(\Omega_{\cM_L}\big)\ =\ (-1)^{\vd}e(\cM_L).
$$

The same proof with non-strict inequalities in \eqref{lt} gives the following.

\begin{prop} 
If $\deg K_S<0$ then any semistable $\C^*$-fixed Higgs pair $(E,\phi)$ has Higgs field $\phi=0.\hfill\square$
\end{prop}

%\begin{proof}
%Let $\cE=\cE_\phi$ be a stable $\C^*$-fixed torsion sheaf on $X=K_S$. Let $r$ be the minimum integer such that $\cE$ is supported scheme theoretically on $rS$. Consider the maps
%\beq{compstab}
%\cE\otimes K_S^{1-r}\To\iota_*(\cE|_S\otimes K_S^{1-r})\To\cE,
%\eeq
%the first being restriction to the zero section $S$, the second induced by the inclusion
%$$
%\iota_*\;K_S^{1-r}\Into\cO_{rS}
%$$
%of the ideal sheaf of $(r-1)S$ inside $rS$.  The first is onto, and the second (and therefore the whole composition \eqref{compstab}) is a surjection onto the kernel of $\cE\to\cE|_{(r-1)S}$.
%
%If $r>1$ and $\deg K_S\le0$, the left hand side of \eqref{compstab} has degree at least as big as that of the right hand side. By stability, this forces the map to be zero.  Thus the kernel of $\cE\to\cE|_{(r-1)S}$ is zero and
%$\cE$ is supported on $(r-1)S$. This contradicts our choice of $r$, so in fact $r=1$.
%\end{proof}
%
%\begin{rmk} The same proof shows that if $\deg K_S<0$ then any \emph{semi}stable Higgs pair also has $\phi=0$.
%\end{rmk}

\subsection{The second fixed locus}\label{spt}
Connected components of the monopole branch $\cM_2$ \eqref{cpt2}, with $\phi\ne0$, are more interesting. Since the $\C^*$-fixed stable sheaves $\cE_\phi$ are simple, we can make them $\C^*$-equivariant by  \cite[Proposition 4.4]{Ko}.\footnote{In \cite[Proposition 5.1]{TT2} we prove the same holds for general (not necessarily simple) $\C^*$-fixed Higgs pairs, which we apply in the semistable case.} Equivalently, $E$ carries a $\C^*$-action $\psi\colon\C^*\to\Aut(E)$ such that
\beq{cstar}
\psi_t\circ\phi\circ\psi_t^{-1}\ =\ t\phi.
\eeq
Therefore $E=\oplus_i E_i$ splits into weight spaces $E_i$ on which $t$ acts as $t^i$. That is, with respect to this splitting we can write $\psi_t=\mathrm{diag}(t^i)$. By \eqref{cstar}, this action acts on the Higgs field with weight $1$. Conversely, a Higgs pair $(E,\phi)$ with a $\C^*$ action on $E$ whose induced action on $\phi$ has weight $1$ clearly defines a fixed point of our original $\C^*$ action.

Since $\phi$ decreases weight, it maps the lowest weight torsion subsheaf to zero. This subsheaf is therefore $\phi$-invariant, and so zero by stability.
In particular each of the $E_i$ are torsion-free, and so in particular have rank $>0$. Thus $\phi$ acts blockwise through maps
\beq{nezt}
\phi_i\colon E_i\To E_{i-1}\otimes K_S,
\eeq
and we get flags of torsion-free sheaves on $S$ of the sort studied by Negut \cite{Ne}.
When the $E_i$ have rank 1 they are twists of ideal sheaves by line bundles (since they are torsion-free by stability), and the maps \eqref{nezt} define nesting of these ideals.
Therefore one can express these components of the $\C^*$-fixed locus $\cM_2$ in terms of nested Hilbert schemes (of curves and points) on $S$. In particular one gets the virtual cycles on nested Hilbert schemes discovered recently in \cite{GSY1, GSY2}, at least when $h^{0,1}(S)=0=h^{0,2}(S)$.

We explain an extended example with $\rk(E)=2$ in detail next, leaving the general case to \cite{GSY1}. Since the $E_i$ are torsion free and rank $>0$, they have rank 1, they are ideal sheaves tensored with line bundles, and there are only two of them:
$$
E\ =\ E_i\oplus E_j.
$$
Without loss of generality $i>j$. Since the Higgs field has weight 1, it takes weight $k$ to weight $k-1$. It is also nonzero, so we must have $j=i-1$ and the only nonzero component of $\phi$ maps $E_i$ to $E_{i-1}$.

Let $\t$ denotes the one dimensional $\C^*$ representation of weight 1. By tensoring $E$ by $\t^{-i}$ (i.e. multiplying the $\C^*$ action on $E$ by $\lambda^{-i}\cdot\id_E$) we may assume without loss of generality that $i=0$ and $j=-1$. Considering $\phi$ as a weight 0 element of $\Hom(E,E\otimes K_S)\otimes\t$, we have
\beq{rk1}
E\ =\ E_0\oplus E_{-1}\ \ \mathrm{and}\ \ \phi\ =\ \mat{0}{0}{\Phi}{0}\mathrm{\ \ for\ some\ \ }\Phi\colon E_0\To E_{-1}\otimes K_S\cdot\t.
\eeq
We now calculate invariants from these on certain general type surfaces.

\section{Calculations on surfaces with positive canonical bundle} \label{quintic}
Let $(S,\cO_S(1))$ be a smooth, connected polarised surface with
\begin{itemize}
\item $h^1(\cO_S)=0$, and
\item a smooth nonempty connected canonical divisor $C\in|K_S|$, such that
\item $L=\cO_S$ is the only line bundle satisfying $0\le\deg L\le\frac12\deg K_S$,
\end{itemize}
where degree is defined by $\deg L=c_1(L)\cdot c_1(\cO_S(1))$. Examples include
\begin{itemize}
\item the generic quintic surface in $\P3$,
\item the generic double cover of $\P2$ branched over an octic,
\item general type surfaces with $\Pic(S)=\Z\cdot K_S$, and
\item the blow up of K3 in a point. 
\end{itemize}
More generally, we expect many general type surfaces with $h^1(\cO_S)=0$ and $p_g>0$ to have enough deformations to move the Hodge structure on $H^2(S)$ so the only integral (1,1) classes are rational multiples of $c_1(S)$. Such surfaces therefore satisfy the above conditions if $c_1(S)\in H^2(S,Z)$ is primitive.

By adjunction the genus $g$ of $C$ is
$$
g=1+c_1(S)^2
$$
while the bundles $K_S$ and $K_S^2$ have sections
$$
h^0(K_S)\ =\ p_g(S)\ =\ \frac1{12}\big(c_1(S)^2+c_2(S)\big)-1
$$
and
\beq{alf}
h^0(K_S^2)\ =\ P_2(S)\ =\ p_g(S)+g\ =\ \frac1{12}\big(13c_1(S)^2+c_2(S)\big).
\eeq
By the obvious long exact sequences and Serre duality it is easy to show they have \emph{no higher cohomology}. \medskip

We consider $\C^*$-fixed rank 2 Higgs pairs $(E,\phi)$ of \emph{fixed determinant $K_S$} in the monopole branch $\cM_2$ \eqref{cpt2}. They were described in \eqref{rk1} above.
In particular, $E_{-1}\subset E$ is a $\phi$-invariant subsheaf, so by semistability
$$
\deg E_{-1}\ \le\ \deg E_0\ =\ \deg K_S-\deg E_{-1}.
$$
But the existence of the nonzero map $\Phi\colon E_0\to E_{-1}\otimes K_S$ implies
$$
\deg E_{-1}+\deg K_S\ \ge\ \deg E_0\ =\ \deg K_S-\deg E_{-1}.
$$
Together these give $0\le\deg E_{-1}\le\frac12\deg K_S$, which by our assumptions on $S$ implies that $\det E_{-1}=\cO_S$ and $\det E_0=K_S$. Therefore
$$
E_0\ =\ I_0\otimes K_S,\quad E_{-1}\ =\ I_1\!\cdot\!\t^{-1},
$$
for some ideal sheaves $I_i$.
Since $\Phi$ defines a nonzero map $I_0\to I_1$ we must have $I_0\subseteq I_1$. That is, denoting by $Z_j$ the 0-dimensional subscheme defined by the ideal $I_j$,
\beq{nest}
Z_1\ \subseteq\ Z_0.
\eeq
This can be done in families to prove the following scheme-theoretic description of the $\C^*$-fixed loci.

\begin{lem} \label{n2} Fix $r=2,\ c_1=-c_1(S),\ c_2\in\Z$. If $c_1(S)^2\ge0$, then for $c_2<0$ the $\C^*$-fixed locus is empty. For $c_2\ge0$ it is the disjoint union of $\cM_{r,c_1,c_2}$ and the nested Hilbert schemes
\beq{nested}
\cM_2\ \cong\ \bigsqcup_{i=0}^{\lfloor c_2(E)/2\rfloor}S^{[i,\,c_2-i]}
\eeq
of subschemes $Z_1\subseteq Z_0\subset S$ of lengths
$$
|Z_1|\ =\ i,\qquad |Z_0|+|Z_1|\ =\ c_2(E).\vspace{-7mm}
$$
$\hfill\square$
\end{lem}

In general $S^{[i,j]}$ is singular, but connected.
In particular we have the cases
\begin{itemize}
\item $c_2(E)<0$. Then $\cM_2$ is empty and --- if $c_1(S)^2\ge0$ ---  $\cM_L$ is also empty by the Bogomolov inequality for stable sheaves. So $\VW_{r,\;K_S,\;c_2<0}=0$.
\item $c_2(E)=0$. Then $\cM_2$ is the single point with $E=K_S\oplus\cO\!\cdot\!\t^{-1}$ and $\phi\colon E\to E\otimes K_S\!\cdot\!\t$ is $\mat{0}{0}{1}{0}$. This Higgs pair already appears in \cite[Equation 2.70]{VW}.
\item $c_2(E)=1$. Then $\cM_2
\cong S$ with $x\in S$ corresponding to $E=I_x\otimes K_S\oplus\cO\!\cdot\!\t^{-1}$ and Higgs field $\phi=\mat{0}{0}{\iota_x}{0}$, where $\iota_x$ is the (twist by $K_S$ of) the inclusion $I_x\into\cO$.
\item $c_2(E)=2$. Now $\cM_2$ has two components,
$$
\cM_2\ \cong\ S^{[2]}\ \sqcup\ S.
$$
The first $S^{[2]}:=\Hilb^2S$ is similar to the previous example, with length-2 subschemes $Z\subset S$ replacing $\{x\}\subset S$. The second is a copy of $S$, with $x\in S$ corresponding to $E=I_x\otimes K_S\oplus I_x\!\cdot\!\t^{-1}$ and Higgs field $\phi=\mat{0}{0}{1}{0}$.
\end{itemize}

\subsection{Horizontal terms}
We will calculate first for $i=0$ in \eqref{nested}, i.e. $Z_1$ empty ($I_1=\cO$) and
$$
E=I_0\otimes K_S\ \oplus\ \cO\!\cdot\!\t^{-1}.
$$
Denoting $|Z_0|=c_2(E)$ by $n$, this is a connected component $S^{[n]}=S^{[0,n]}$ of the $\C^*$-fixed locus, as noted above.
Even this case turns out to be both hard and interesting.
It describes \emph{all} $\C^*$-fixed points when $c_2\le1$ and $c_1(S)^2\ge0$, by the Bogomolov inequality.

The deformation theory of $(E,\phi)$ with $\det E=K_S$ fixed and $\tr\phi=0$ is governed by $R\Hom\_X(\cE_\phi,\cE_\phi)\_\perp$, which is the co-cone on
\beq{cone}
\xymatrix@C=35pt{
R\Hom(E,E)\_0 \ar[r]^(.4){[\ \cdot\ ,\,\phi]} & R\Hom(E,E\otimes K_S\otimes\t)\_0\;.}
\eeq
Here we have been more explicit about the $\C^*$ action, inserting the character $\t$ so that the Higgs field $\phi$ becomes of weight 0.

The splitting of $E$ induces a splitting of $\Hom(E,E)$ which we denote, in the obvious notation, by
$$
\mat{\Hom(I_0K_S,I_0K_S)}{\Hom(\cO\!\cdot\!\t^{-1},I_0K_S)}{\Hom(I_0K_S,\cO\!\cdot\!\t^{-1})}{\Hom(\cO\!\cdot\!\t^{-1},\cO\!\cdot\!\t^{-1})}\,=\
\mat{\C\cdot\id_{I_0}}{H^0(I_0K_S)\t}{0}{\C\cdot\id_{\cO}}.
$$
Similarly, $\Hom(E,E\otimes K_S\!\cdot\!\t)$ splits as
$$
\mat{\Hom(I_0K_S,I_0K_S^2\!\cdot\!\t)}{\Hom(\cO\!\cdot\!\t^{-1},I_0K_S^2\!\cdot\!\t)\!}{\Hom(I_0K_S,K_S)}{\Hom(\cO\!\cdot\!\t^{-1},K_S)}=
\mat{\!H^0(K_S)\t}{H^0(I_0K_S^2)\t^2\!}{\C\cdot\iota}{H^0(K_S)\t}\!.
$$
Since $\phi=\mat00\iota0$, the map $[\ \cdot\ ,\phi]$ between them acts by
$$
\mat as0b\ \Mapsto\ \mat{s\iota}0{(b-a)\iota}{-\iota s}.
$$
Setting $b=-a$ gives the map on trace-free groups. We find the map $\Hom(E,E)\_0\to\Hom(E,E\otimes K_S\otimes\t)\_0$ is injective and has cokernel
\beq{coker}
\frac{H^0(K_S)}{\iota H^0(I_0K_S)}\cdot\t\ \oplus\ H^0(I_0K_S^2)\t^2.
\eeq

\medskip Next we compute $\Ext^1(E,E)=\Ext^1(E,E)\_0$ as
{\small $$
\mat{\Ext^1(I_0K_S,I_0K_S)}{\Ext^1(\cO\!\cdot\!\t^{-1},I_0K_S)}{\Ext^1(I_0K_S,\cO\!\cdot\!\t^{-1})}{\Ext^1(\cO\!\cdot\!\t^{-1},\cO\!\cdot\!\t^{-1})}\,=\
\mat{T_{Z_0}S^{[n]}}{H^1(I_0K_S)\t}{H^1(I_0K_S^2)^*\t^{-1}}0\!,
$$}
where in the bottom left entry we have used Serre duality on $S$.

Similarly $\Ext^1(E,E\otimes K_S\!\cdot\!\t)=\Ext^1(E,E\otimes K_S)\_0\cdot\t$ is
{\small $$
\mat{\Ext^1(I_0K_S,I_0K_S^2\!\cdot\!\t)}{\Ext^1(\cO\!\cdot\!\t^{-1},I_0K_S^2\!\cdot\!\t)}{\Ext^1(I_0K_S,K_S)}{\Ext^1(\cO\!\cdot\!\t^{-1},K_S)}\,=\
\mat{(T^*_{Z_0}S^{[n]})\t}{H^1(I_0K_S^2)\t^2}{H^1(I_0K_S)^*}0\!.
$$}
The map $[\ \cdot\ ,\phi]$ between them acts by
$$
\mat vsf0\ \Mapsto\ \mat{s\iota}0{-\iota v}{-\iota s}.
$$

\begin{lem} \label{zero} This map vanishes.
\end{lem}

\begin{proof}
In the exact sequence
\beq{serre}
\Hom(I_0,\cO_{Z_0})\To\Ext^1(I_0,I_0)\rt\iota\ \Ext^1(I_0,\cO)
\eeq
the first arrow is an isomorphism: it is the identity map from $T_{Z_0}S^{[n]}$ to itself. Therefore the second arrow $\iota$ is zero, and $\iota v=0$.

Similarly we can see $s\iota$ from
\begin{align*}
\Ext^1(\cO,I_0K_S)&\rt{\iota^*}\Ext^1(I_0,I_0K_S), \\
s \qquad &\,\Mapsto\qquad s\iota.
\end{align*}
But $\iota^*$ is the Serre dual of the second arrow in \eqref{serre}, which we saw was zero.

Finally $\iota s$ lies in $H^1(K_S)$, which vanishes since $h^1(\cO_S)=0$.
\end{proof}

So we are now ready to analyse the deformation-obstruction theory of $(E,\phi)$. Let $T^i$ be the cohomology groups of the cone \eqref{cone}. We have the exact sequence
{\small \begin{multline*}
\!\!\!\!0\to T^0\to\Hom(E,E)\_0\rt{[\,\cdot\,,\phi]}\Hom(E,E\otimes K_S)\_0\t\to T^1\to\Ext^1(E,E)
\rt{[\,\cdot\,,\phi]}\quad \\ \quad\Ext^1(E,E\otimes K_S)\t\to T^2\To\Ext^2(E,E)\_0\rt{[\,\cdot\,,\phi]}
\Ext^2(E,E\otimes K_S)\_0\t\to T^3\to0.
\end{multline*}}

\noindent We have shown that the third arrow is injective, so $T^0=0$. This says that $(E,\phi)$ has no trace-free infinitesimal automorphisms, as we knew already from stability. By Serre duality, the third-from-last map is surjective and $T^3=0$.
From \eqref{coker} and Lemma \ref{zero} the sequence splits into
{\small $$
0\to\frac{H^0(K_S)}{\iota H^0(I_0K_S)}\!\;\cdot\!\;\t\ \oplus\ H^0(I_0K_S^2)\t^2\To T^1\To
%T_{Z_0}\Hilb^nS\ \oplus\ H^1(I_0K_S)\t\ \oplus\ H^1(I_0K_S^2)^*\t^{-1}
\mat{T_{Z_0}S^{[n]}}{H^1(I_0K_S)\t\!}{\!H^1(I_0K_S^2)^*\t^{-1}}0\to0
$$}

\noindent and its Serre dual for $T^2$:
\beq{fir} \vspace{-1mm} \eeq
{\small $$
0\to\!\mat{(T^*_{Z_0}S^{[n]})\t}{H^1(I_0K_S^2)\t^2\!}{\!H^1(I_0K_S)^*}0\!
\To T^2\To\!\left(\frac{H^0(K_S)}{\iota H^0(I_0K_S)}\right)^{\!\!*\!}\ \oplus\ H^0(I_0K_S^2)^*\t^{-1}\!\to0.
$$}

\subsection{The virtual cycle} \label{kiemli}
In particular we see the fixed (weight 0) part of $T^1$ is just $T_{Z_0}S^{[n]}$, as expected.

The weight 1 part of $T^1$ comes from putting together $H^0(K_S)\big/\iota H^0(I_0K_S)$ and $H^1(I_0K_S)$. Together we claim these give $\Gamma(K_S|_{Z_0})$. Although we don't need this --- we only require the K-theory class of $T^1$ to compute the invariant \eqref{defvw} --- for completeness we briefly describe how to understand it.\medskip

\noindent{\textbf{Digression.}} We do this by working with the $\C^*$-invariant sheaf $\cE_\phi$ on $X$ which is equivalent, via the spectral construction, to the Higgs pair $(E,\phi)$. It is 
$$
\cE_\phi\ =\ I_{Z_0\subset 2S}\otimes K_S,
$$
the pushforward from $2S$ to $X$ of the ideal sheaf of $Z_0\subset2S$, all twisted by $K_S$. (Here $2S\subset X$ is the thickening of the zero section $S\subset X$ defined by the ideal $I_{S\subset X}^2$.) Let $s_0$ denote the tautological section of $\pi^*K_S$ on $X=K_S$, vanishing on the zero section $S$, and fix any section $s$ of $K_S$. The support of $\cE_\phi$ is $2S=\{s_0^2=0\}$; a weight-1 deformation of this (parametrised by $t$) is given by
\beq{t2}
s_0^2=t^2s^2,
\eeq
splitting $2S$ into the two sections $s_0=\pm ts$.
Taking the structure sheaf of this deformation and twisting by both $K_S$ and the ideal sheaf
\beq{t3}
\pi^*I_{Z_0}+(s_0-ts)
\eeq
gives a corresponding deformation of $\cE_\phi$. However, away from $Z_0$ it is the \emph{trivial} deformation to first order. That is, working over $\Spec\C[t]/(t^2)$ we see that \eqref{t2} is just $s_0^2=0$ --- i.e. the same $2S$ as at $t=0$. The deformation \emph{does} move $Z_0$ to first order, however, if and only if $s\not\in I_{Z_0\,}$, as can be seen from \eqref{t3}. This describes the $H^0(K_S)\big/H^0(I_0K_S)$ part of the weight-1 part of the first order deformations $T^1$.

Since the deformation is zero away from $Z_0$ we do not need $s$ to be a global section of $K_S$. Any local section defined near $Z_0$ also defines the above deformation --- glued to zero away from $Z_0$. This describes \emph{all} of the weight-1 part of the first order deformations $T^1$ as $\Gamma(K_S|_{Z_0})$ --- both the deformations $H^0(K_S)\big/H^0(I_0K_S)$ coming from global sections $s$, and the quotient $H^1(I_0K_S)$ by these. \medskip

Now $\Gamma(K_S|_{Z_0})$ is the fibre over $Z_0\in S^{[n]}$ of the rank $n$ bundle $K_S^{[n]}$ on $S^{[n]}$, and this bundle is globally the weight 1 part of $T^1$. Dually, the (fixed part of the) obstruction bundle on our smooth $\C^*$-fixed moduli space $\cM_2$ is $\big(K_S^{[n]}\big)^*$.

Therefore the virtual cycle that $\cM_2$ inherits from \cite{GP} is the Euler class of this obstruction bundle:
\beqa
\big[\cM_2\big]^{\vir} &=& e\big(\big(K_S^{[n]}\big)^{*\,}\big)\cap S^{[n]}\\
&=& (-1)^ne\big(K_S^{[n]}\big)\cap S^{[n]}.
\eeqa
Fix the section of $K_S$ cutting out the smooth canonical curve $C$. It induces a section of $K_S^{[n]}$ with zero locus $C^{[n]}=\Sym^n C\subset S^{[n]}$. Since this has the correct dimension, it is Poincar\'e dual to $e\big(K_S^{[n]}\big)$, and we find that
\beq{vc}
\big[\cM_2\big]^{\vir}\ =\ (-1)^n\big[C^{[n]}\big]\ \subset\ S^{[n]}\ =\ \cM_2.
\eeq

\subsection{The virtual normal bundle} Reading off the virtual normal bundle $N^{\vir}$ from the moving parts of the above computations \eqref{fir} gives
$$
\Gamma(K_S|_{Z_0})\t\ \oplus\ 
R\Gamma(I_0K_S^2)\t^2\ \oplus\ 
R\Gamma(I_0K_S^2)^\vee\t^{-1}[-1]\ \oplus\ 
T^*_{Z_0}S^{[n]}\t[-1]
$$
at $Z_0\in\cM_2$. Notice the final term has weight 1 --- if Serre duality did not shift weights in this way, this term would have weight 0 and would be the fixed obstruction bundle. Therefore the virtual class $[\cM_2]^{\vir}$ would be a signed Euler characteristic as before. As it is, the $\cM_2$ contribution is more interesting.

We only care about its K-theory class, for which it simplifies things to express $R\Gamma(I_0K_S^2)$ as $H^0(K_S^2)-H^0(K_S^2|_{Z_0})$. As $Z_0$ moves through $S^{[n]}$, the first term is a fixed $\C^{P_2}$, where the plurigenus $P_2(S)=h^0(K_S^2)=p_g(S)+g$
is the constant \eqref{alf}. The second term is the fibre of the vector bundle $(K_S^2)^{[n]}$. That is, in $\C^*$-equivariant K-theory, $N^{\vir}$ is
$$
\big[K_S^{[n]}\big]\t\ +\ 
(\t^2)^{\oplus P_2}\ -\ 
\big[(K_S^2)^{[n]}\big]\t^2\ -\ 
(\t^{-1})^{\oplus P_2}\ +\ 
\big[\big((K_S^2)^{[n]}\big)^*\big]\t^{-1}\ -\ 
\big[T^*_{S^{[n]}}\big]\t.
$$
Therefore
\begin{eqnarray}
\frac1{e(N^{\vir})}
% &=& \frac{e\Big((K_S^2)^{[n]}\t^2\,\oplus\,(\t^{-1})^{\oplus P_2}\,\oplus\,T^*_{S^{[n]}}\t\Big)}
% {e\Big(K_S^{[n]}\t\,\oplus\,(\t^2)^{\oplus P_2}\,\oplus\,\big((K_S^2)^{[n]}\big)^*\t^{-1}\Big)} \nonumber \\
&=& \frac{e\big((K_S^2)^{[n]}\t^2\big)e\big((\t^{-1})^{\oplus P_2}\big)e\big(T^*_{S^{[n]}}\t\big)}
{e\big(K_S^{[n]}\t\big)e\big((\t^2)^{\oplus P_2}\big)e\big(\big((K_S^2)^{[n]}\big)^*\t^{-1}\big)} \nonumber \\
&=& \frac{(2t)^nc_{\frac1{2t}}\big((K_S^2)^{[n]}\big)\cdot(-t)^{P_2}\cdot t^{2n}c_{\frac1t}\big(T^*_{S^{[n]}}\big)}
{t^nc_{\frac 1t}\big(K_S^{[n]}\big)\cdot(2t)^{P_2}\cdot(-1)^nt^nc_{\frac1t}\big((K_S^2)^{[n]}\big)} \nonumber \\
&=& (-2)^{n-P_2\,}t^n\frac{c_{\frac1{2t}}\big((K_S^2)^{[n]}\big)c_{-\frac1t}\big(T_{S^{[n]}}\big)}
{c_{\frac 1t}\big(K_S^{[n]}\big)c_{\frac1t}\big((K_S^2)^{[n]}\big)}\,, \label{1n}
\end{eqnarray}
where $c_s(E):=1+sc_1(E)+\cdots+s^rc_r(E)$ for a bundle $E$ of rank $r$; when $s=1$ this is $c_\bullet$ (the total Chern class).
 
When we take the degree $n$ part of this and integrate over the virtual cycle $(-1)^n\big[C^{[n]}\big]$ \eqref{vc}, only the $t^0$ part contributes because \eqref{1n} has total degree $n$ (that is $-\rk(N^{\vir})=n$ equals the virtual dimension of $\cM_2$). So we can set $t=1$ to get the same answer. We can also use the fact that $C^{[n]}$ is cut out by a transverse section of $K_S^{[n]}$, so its normal bundle is the restriction of $K_S^{[n]}$. Thus $T_{S^{[n]}}|_{C^{[n]}}= T_{C^{[n]}}\oplus K_S^{[n]}|_{C^{[n]}}$ in K-theory, and
\beq{sofar}
\int_{[\cM_2]^{\vir}}\frac1{e(N^{\vir})}\ =\ (-2)^{-P_2}2^n\int_{C^{[n]}}
\frac{c_{\frac12}\big((K_S^2)^{[n]}\big)c_{-1}\big(T_{C^{[n]}}\big)c_{-1}\big(K_S^{[n]}\big)}
{c_\bullet\big(K_S^{[n]}\big)c_\bullet\big((K_S^2)^{[n]}\big)}\,.
\eeq

\subsection{Tautological classes} To evaluate this integral we express it in terms of two tautological cohomology classes on $C^{[n]}$. Fix a basepoint $c_0\in C$. The first class is
$$
\omega\ :=\ \mathrm{PD}\big[C^{[n-1]}\big]\ \in\ H^2\big(C^{[n]},\Z\big),
$$
where $C^{[n-1]}\into C^{[n]}$ is the map taking $Z\mapsto Z+c_0$.

For the second we use the Abel-Jacobi map
\beq{AJ}
\AJ\colon C^{[n]}\To\Pic^nC, \qquad Z\Mapsto\cO(Z).
\eeq
Tensoring with powers of $\cO(c_0)$ makes the $\Pic^nC$ isomorphic for all $n$, so we may pull back the theta divisor from $\Pic^{g-1}C$. Its cohomology class
$$
\theta\ \in\ H^2(\Pic^nC,\Z)\ \cong\ \Hom(\Lambda^2 H^1(C,\Z),\Z)
$$
maps $\alpha,\beta\in H^1(C,\Z)$ to $\int_C \alpha\wedge \beta$. We also use $\theta$ to denote its pullback $\AJ^*\theta$, giving our second tautological class
\beq{theta}
\theta\ \in\ H^2(C^{[n]},\Z).
\eeq

The integrals of these classes are given by
\begin{eqnarray} \nonumber
\int_{C^{[n]}}\frac{\theta^i}{i!}\,\omega^{n-i} &=& \int_{C^{[i]}}\frac{\theta^i}{i!}\ =\ 
\int_{\Pic^iC}\!\AJ_*\!\big[C^{[i]}\big]\cup\frac{\theta^i}{i!}
\ =\ \int_{\Pic^iC}\frac{\theta^{g-i}}{(g-i)!}\cdot\frac{\theta^i}{i!}\\ &=&{g\choose i}.\label{Poinc}
\end{eqnarray}
The first equality follows from repeated intersection with the divisor $C^{[n-1]}\subset C^{[n]}$; the last two are Poincar\'e's formula \cite[Section I.5]{ACGH}.

In our application we also need to know the two identities
$$
c_t(T_{C^{[n]}})\ =\ (1+\omega t)^{n+1-g}
\exp\left(\frac{-t\theta}{1+\omega t}\right)
$$
and
$$
c_t(L^{[n]})\ =\ (1-\omega t)^{n+g-1-\deg L}\exp\left(\frac{t\theta}{1-\omega t}\right)
$$
from, for example, \cite[Section VIII.2]{ACGH}.

Plugging these into \eqref{sofar} makes the integrand $(-2)^{-P_2}2^n$ times by
$$
\frac{(1-\frac\omega2)^{n+1-g}\exp\left(\frac{\theta}{2-\omega}\right)(1-\omega)^{n+1-g}
\exp\left(\frac{\theta}{1-\omega}\right)(1+\omega)^n\exp\left(\frac{-\theta}{1+\omega}\right)}
{(1-\omega)^n\exp\left(\frac{\theta}{1-\omega}\right)(1-\omega)^{n+1-g}\exp\left(\frac{\theta}{1-\omega}\right)}\,.
$$
Therefore our invariant is
\beq{sofar2}
(-2)^{-p_g(S)-1}(-1)^n\!\!\int_{C^{[n]}}\!(\omega-2)^{n+1-g}\frac{(1+\omega)^n}{(1-\omega)^n}
\,\exp\!\left(\!\frac{\theta}{2-\omega}-\frac{\theta}{1+\omega}-\frac{\theta}{1-\omega}\!\right)\!.
\eeq
%Or:
%$$
%\frac{(-2)^{n-P_2}\sum_{k=0}^{\infty} (1-\omega)^{-n-k}\frac{(-\theta)^k}{k!}\sum_{k=0}^{\infty} (1-\omega)^{5-n-k}\frac{(-\theta)^k}{k!}}{\sum_{k=0}^{\infty} (1-\frac\omega2)^{5-n-k}\frac{(-\theta/2)^k}{k!}
%\sum_{k=0}^{\infty} (1-\omega)^{n+1-g-k}\frac{\theta^k}{k!}
%\sum_{k=0}^{\infty} (1+\omega)^{-n-k}\frac{\theta^k}{k!}}\
%$$
%$$
%=\ -2^{-5}(\omega-2)^{n-5}\frac{(1+\omega)^n}{(1-\omega)^n}
%\frac{\sum_{a=0}^{\infty} (\omega-1)^{-a}\frac{\theta^a}{a!}\sum_{b=0}^{\infty} (\omega-1)^{-b}\frac{\theta^b}{b!}}{\sum_{k=0}^{\infty} (1-\frac\omega2)^{5-n-k}\frac{(-\theta/2)^k}{k!}
%\sum_{k=0}^{\infty} (1-\omega)^{n+1-g-k}\frac{\theta^k}{k!}
%\sum_{k=0}^{\infty} (1+\omega)^{-n-k}\frac{\theta^k}{k!}}\
%$$

To simplify this integral we notice from \eqref{Poinc} that
$$
\int_{C^{[n]}}\frac{\theta^i}{i!}\,\omega^{n-i}\ =\ 
\int_{C^{[n]}}{g\choose i}\omega^i\,\omega^{n-i},
$$
so whenever $\theta^i/i!$ is integrated against only powers of $\omega$ we may replace it by ${g\choose i}\omega^i$. In particular,
$$
\exp(\alpha\theta)\ =\ \sum_{i=0}^\infty\alpha^i\frac{\theta^i}{i!}\ \sim\ \sum_{i=0}^\infty\alpha^i{g\choose i}\omega^i\ =\ (1+\alpha\omega)^g
$$
for $\alpha$ a power series in $\omega$. The $\sim$ becomes an equality if we integrate both sides over $C^{[n]}$ against a power series in $\omega$. In \eqref{sofar2} this gives
\begin{multline*}
%(-2)^{-p_g(S)-1}(-1)^n
\frac{(-1)^{n+p_g(S)+1}}{2^{p_g(S)+1}}
\int_{C^{[n]}}(\omega-2)^{n+1-g}\frac{(1+\omega)^n}{(1-\omega)^n}
\left(1+\frac{\omega}{2-\omega}-\frac{\omega}{1+\omega}-\frac{\omega}{1-\omega}\right)^g \\
=\ (-2)^{-p_g(S)-1}(-1)^n\int_{C^{[n]}}(\omega-2)^{n+1-2g}\frac{(1+\omega)^{n-g}}{(1-\omega)^{n+g}}\big(4\omega-2\big)^g.
\end{multline*}
Therefore, in terms of generating series,
\begin{align} \nonumber
(-2)^{p_g(S)-g+1}\sum_{n=0}^\infty(-1)^nq^n&\int_{C^{[n]}}\frac1{e(N^{\vir})}\\ =&\ \sum_{n=0}^\infty q^n\int_{C^{[n]}}(\omega-2)^{n+1-2g}\frac{(1+\omega)^{n-g}}{(1-\omega)^{n+g}}\big(1-2\omega\big)^g.
\label{gens}
\end{align}

\subsection{The answer in closed form}
The series \eqref{gens} is the \emph{diagonal} \cite[Section 6.3]{St2} of the double generating series
\beq{dubl}
\sum_{i,n=0}^\infty x^it^n\int_{C^{[i]}}(\omega-2)^{n+1-2g}\frac{(1+\omega)^{n-g}}{(1-\omega)^{n+g}}\big(1-2\omega\big)^g.
\eeq
That is, if we write \eqref{dubl} as $\sum_{i,n}a_{in}x^it^n$, then our generating series \eqref{gens} is $\sum_na_{nn}q^n$.

We can evaluate \eqref{dubl} by first summing over $n$ to give
$$
\sum_{i=0}^\infty x^i\int_{C^{[i]}}\frac{(1-2\omega)^g}{(\omega-2)^{2g-1}(1-\omega^2)^g}\left(1-t\,\frac{(\omega-2)(1+\omega)}{1-\omega}\right)^{\!-1}.
$$
Since the integrand is independent of $i$, and $\int_{C^{[i]}}\omega^j=\delta_{ij}$, the operator $\sum_{i=0}^\infty x^i\int_{C^{[i]}}$ simply acts by setting $\omega$ to be $x$. So we get
$$
\frac{(1-2x)^g}{(x-2)^{2g-1}(1-x^2)^g}\ 
\frac{1-x}{1-x-t(x^2-x-2)}\,.
$$
To find the diagonal \eqref{gens} of this series we substitute $t=q/x$ and consider the integral \cite[Section 6.3]{St2}
$$
\frac1{2\pi i}\oint\frac{(1-2x)^g}{(x-2)^{2g-1}(1-x^2)^g}\ 
\frac{1-x}{1-x-\frac qx(x^2-x-2)}\ \frac{dx}x
$$
around a small loop containing only those poles \emph{which tend to the origin as $q\to0$}. Thus \eqref{gens} is the residue of
$$
\frac{(1-2x)^g}{(x-2)^{2g-1}(1-x^2)^g}\ 
\frac{-(1-x)}{(1+q)x^2-(1+q)x-2q}
$$
at the root
\beq{x0}
x\_0\,:=\ \frac12\left(1-\sqrt{1+\frac{8q}{1+q}}\,\right)
\eeq
of the quadratic $(1+q)x^2-(1+q)x-2q$ in $x$. This is
$$
\frac{(1-2x\_0)^g}{(x\_0-2)^{2g-1}(1+x_0)^g(1-x_0)^{g-1}}\ \frac{-1}{(1+q)(x\_0-x\_1)}\,,
$$
where $x\_1=\frac12\left(1+\sqrt{1+\frac{8q}{1+q}}\,\right)$ is the other root. Substituting in $x\_1-x\_0=\sqrt{1+\frac{8q}{1+q}}$ and $(1-2x\_0)^g=\big(1+\frac{8q}{1+q}\big)^{g/2}$, and multiplying out the denominator this becomes
$$
\frac{(-1)^{g-1}}{(x_0^2-x_0-2)^g(x_0^2-3x_0+2)^{g-1}}\ 
\frac{(1+9q)^{\frac{g-1}2}}{(1+q)^{\frac{g+1}2}}\,.
$$
Since $x_0$ was chosen as a root of $x_0^2-x_0-\frac{2q}{1+q}=0$, and is given by the formula \eqref{x0}, we find the terms in the denominator are
$$
\ x_0^2-x_0-2\ =\ -\frac2{1+q} \quad\mathrm{and}\quad
x_0^2-3x_0+2\ =\ \frac{1+3q+\sqrt{(1+q)(1+9q)}}{1+q}\,.\
$$
Rearranging and substituting into \eqref{gens} finally gives the following answer.

\begin{prop} \label{alge}
\beq{justj}
\sum_{n=0}^\infty q^n\int_{C^{[n]}}\frac1{e(N^{\vir})}\ \ =\ \ c\;(1-q)^{g-1}\left(1+\frac{1-3q}{\sqrt{(1-q)(1-9q)}}\right)^{\!1-g},
\eeq
where $c=(-1)^{p_g(S)+g}\cdot2^{-p_g(S)-1}$. Expanding in powers of $q$ gives
\begin{multline}
(-2)^{-p_g(S)-g}\Big(1-2(g-1)q+(g-1)(2g-11)q^2 \\
-\frac23(g-1)\big(2g^2-31g+126\big)q^3+\cdots\Big). \label{expa}
\end{multline}
\end{prop}

Even after multiplying by a shift $q^{-s}$, this is clearly  not modular \cite{BE}. \medskip

%For instance, when $S$ is the blow up of K3 in a point, so $p_g(S)=1$ and $g=0$, the generating series is
%$$
%-\frac1{4(1+q)}\left(1+\frac{1+3q}{\sqrt{(1+q)(1+9q)}}\right).
%$$

%In contrast, the (\emph{unweighted}) Euler characteristics of the nested Hilbert schemes are studied in \cite{TT2}. Their generating function turns out to be
%$$
%\sum_{i\le j}e\big(S^{[i,j]}\big)q^{i+j}\ =\
%(1-q)^{e(S)}\left(\prod_{\,n=1}^\infty\frac1{1-q^n}\right)^{2e(S)}.
%$$
%After multiplying by Vafa-Witten's shift $q^{-s}=q^{-\frac{e(S)}{12}}$ this becomes
%\beq{modu}
%q^{-s}Z(q)\ =\ (1-q)^{e(S)}\cdot\eta(q)^{-2e(S)}.
%\eeq
%So \emph{up to the rational function} $(1-q)^{e(S)}$ this is the weight $-e(S)$ modular form $\eta(q)^{-2e(S)}$. We take this as a first hint that the invariant relevant to physics is not the one of this paper but that of \cite{TT2}. However, we do not understand the significance of the rational function $(1-q)^{e(S)}$, nor why the weight is twice Vafa-Witten's prediction. They may be related to the fact that we took Higgs pairs of fixed determinant $\cO_S(1)$ rather than $\cO_S$. Finally, we only know how to put in the correct Behrend function weighting in a very few cases \cite{TT2}.

\subsection{Euler characteristics of nested Hilbert schemes}\label{bah}
If instead of the virtual theory we use Euler characteristics\footnote{It is reasonable to hope the Behrend function is 1 at all $\C^*$ fixed points; cf. \cite[Proposition 5.9]{TT2} and \cite[Corollary 5.8]{MT}.} then by G\"ottsche's famous calculation we get the generating series
$$
\left(\prod_{n=1}^\infty\frac1{1-q^n}\right)^{\!e(S)}
$$
which \emph{is} modular after multiplying by the shift $q^{-e(S)/24}$. In fact we can do better and sum over \emph{all} nested Hilbert schemes --- corresponding to all partitions, not just the horizontal terms --- and again get something close to a modular form. By \cite[Equation 1.128]{St1} the generating series of nested partitions is
$$
\sum_{\mu\le\lambda}q^{|\mu|+|\lambda|}\ =\ (1-q)\left(\prod_{\,n=1}^\infty\frac1{1-q^n}\right)^{\!2}.
$$
Therefore by similar arguments to G\"ottsche's we find the generating series of the bare Euler characteristic versions of the $\vw$ invariants is
$$
\sum_{i\le j}e\big(S^{[i,j]}\big)q^{i+j}\ =\
(1-q)^{e(S)}\left(\prod_{\,n=1}^\infty\frac1{1-q^n}\right)^{\!2e(S)}.
$$
After multiplying by the shift $q^{-s}=q^{-\frac{e(S)}{12}}$ this becomes
\beq{stan}
q^{-s}\sum_{i\le j}e\big(S^{[i,j]}\big)q^{i+j}\ =\ (1-q)^{e(S)}\cdot\eta(q)^{-2e(S)}.
\eeq
So \emph{up to the rational function} $(1-q)^{e(S)}$ this is the weight $-e(S)$ modular form (with character) $\eta(q)^{-2e(S)}$.

We do not understand the significance of the rational factor, the modular form, nor the weight (which is twice the standard Vafa-Witten prediction). Note that a similar formula has recently been obtained for nested Hilbert schemes with 3 steps \cite{Bo}, relevant to $SU(3)$ Vafa-Witten theory.

For some time we found this calculation suggestive that $\vw$ is the ``correct" Vafa-Witten invariant, but as explained in Section \ref{discuss} we no longer think this. For one thing, the formula \eqref{stan} depends only on $c_2(S)$, whereas the predictions in \cite[Equation 5.38]{VW} also involve $c_1^2(S)$.

The solution seems to be that we have to add in the non-horizontal terms --- the contributions of $S^{[i,j]}$ with nonzero $i$ --- to \eqref{justj}. Then we appear to get the modular forms predicted by Vafa and Witten, as we show in the next few Sections.

\subsection{Vertical terms}
Having dealt with the length\;$(Z_1)=0$ component of $\cM_2$, we now turn to the other extreme: the length\;$(Z_1)=$ length\;$(Z_0)$ component. In the description of Lemma \ref{n2}, the Higgs field $\Phi\colon I_0\to I_1$ must be an isomorphism, so we get the Hilbert scheme
$$
S^{[n,n]}\ \cong\ S^{[n]}
$$
of Higgs pairs
$$
E\ =\ I_Z\otimes K_S\ \oplus\ I_Z\cdot\t^{-1}, \qquad \phi\ =\ \mat{0}{0}{1}{0}\colon E\to E\otimes K_S\cdot\t,
$$
where $Z\subset S$ is a 0-dimensional subscheme of length $n=c_2(E)/2$. The torsion sheaf $\cE_\phi$ on $X$ that this corresponds to is the twist by $\pi^*K_S$ of 
\beq{etor}
\cF_Z\ :=\ (\pi^*I_Z)\otimes\cO_{2S}.
\eeq

To work out the deformation theory it is convenient to work with the description \eqref{etor}. By the exact sequence
$$
0\To\pi^*I_Z(-2S)\To\pi^*I_Z\To\cF_Z\To0
$$
we find
$$
R\Hom(\cF_Z,\cF_Z)\To R\Hom(\pi^*I_Z,\cF_Z)\To R\Hom(\pi^*I_Z,\pi_*\;\cF_Z(2S)).
$$
The section of $\cO(2S)$ cutting out $2S\subset X$ annihilates $\cF_Z$, so the second arrow is zero. By adjunction and $\pi_*\;\cF=I_Z\,\oplus\,I_Z\otimes\!K_S^{-1}\cdot\t^{-1}$ we therefore find
\begin{multline*}
R\Hom_X(\cF_Z,\cF_Z)\ \cong\ R\Hom_S(I_Z,I_Z)\oplus R\Hom_S(I_Z,I_Z\otimes K_S^{-1})\t^{-1} \\ \oplus\ 
R\Hom_S(I_Z,I_Z\otimes K_S^2)\t^2[-1]\oplus R\Hom_S(I_Z,I_Z\otimes K_S)\t[-1].
\end{multline*}
The $SU(2)$ perfect obstruction theory $R\Hom_X(\cF_Z,\cF_Z)\_\perp[1]$ (the derived dual of Theorem \ref{final}) comes from taking trace-free parts of the first and last terms, so
$\Hom\_\perp=0=\Ext^3_\perp$,
\begin{align} \nonumber
\Ext^1_X(\cF_Z,\cF_Z)\_\perp\ =\ &\Ext^1_S(I_Z,I_Z)\oplus\Ext^1_S(I_Z,I_Z\otimes K_S^{-1})\t^{-1} \\ &\oplus\ 
\Hom_S(I_Z,I_Z\otimes K_S^2)\t^2\oplus\Hom_S(I_Z,I_Z\otimes K_S)\_0\t,\label{extperp}
\end{align}
and $\Ext^2_\perp$ is its dual tensored with $\t$ (by Serre duality and $\omega_X\cong\cO_X\cdot\t^{-1}$).
The first term of \eqref{extperp} is just $T_ZS^{[n]}$, the fixed part of the deformations. The last term vanishes, so by duality so does the fixed part of the obstructions. Therefore
$$
\big[S^{[n,n]}\big]^{\vir}\ =\ \big[S^{[n]}\big]
$$
and we can read off the virtual normal bundle:
\begin{multline*}
\hspace{-3mm} N^{\vir}\ =\ \Ext^1_S(I_Z,I_Z\otimes K_S^{-1})\t^{-1}\oplus
\Hom_S(I_Z,I_Z\otimes K_S^2)\t^2 \\
\quad -\Big[\!\Ext^1_S(I_Z,I_Z\otimes K_S)\t\oplus\Ext^1_S(I_Z,I_Z\otimes K_S^2)\t^2 \oplus\ 
\Ext^2_S(I_Z,I_Z\otimes K_S^{-1})\t^{-1}\Big].
\end{multline*}
Integrating the reciprocal of its equivariant Euler class over $S^{[n]}$ is a project for future work, but in the case $n=1$ it is easy enough. We get
$$
\int_{S}\frac1{e(N^{\vir})}\ =\ \int_S\frac
{e(\Omega_S\,\t)e(T_S\otimes K_S^2\;\t^2)e(H^0(K_S^2)^*\;\t^{-1})}
{e(T_S\otimes K_S^{-1}\;\t^{-1})e(H^0(K_S^2)\;\t^2)}\,.
$$
Here we have repeatedly used the isomorphism $\Ext^1(I_p,I_p)\cong T_pS$ and its Serre dual, for any point $p\in S$. We have also computed $\Ext^1(I_p,I_p\otimes K_S^2)=T_pS\otimes K_S^2$ from the local-to-global spectral sequence $\oplus_{i+j=k\,}H^i(\ext^j)\so\Ext^k$ (using that $H^1(\hom)=H^2(\hom)=0$ in this case). It follows from Serre duality that $\Ext^1(I_p,I_p\otimes K_S^{-1})=T_pS\otimes K_S^{-1}$ (proving directly the vanishing of the differential $H^0(\ext^1)\to H^2(\hom)$ would be more troublesome in this case).

In the notation of \eqref{1n} this is
$$
\int_S\frac{t^2c_{\frac1t}(\Omega_S)\cdot(2t)^2c_{\frac1{2t}}(T_S\otimes K_S^2)\cdot(-t)^{P_2}}
{(-t)^2c_{-\frac1t}(T_S\otimes K_S^{-1})\cdot(2t)^{P_2}}\,.
$$
Since only the $t^0$ term contributes, we may set $t=1$ to give
$$
(-2)^{-P_2}\int_S\frac{(1-c_1+c_2)\cdot4\Big(1+\frac12c_1(T_S\otimes K_S^2)+\frac14c_2(T_S\otimes K_S^2)\Big)}
{1-c_1(T_S\otimes K_S^{-1})+c_2(T_S\otimes K_S^{-1})}\,,
$$
where $c_i:=c_i(S)$. Substituting in the identities 
\beqa
c_1(T_S\otimes K_S^2) &=& c_1+2(-2c_1)\ =\ -3c_1, \\
c_2(T_S\otimes K_S^2) &=& c_2+c_1(-2c_1)+(-2c_1)^2\ =\ c_2+2c_1^2, \\
c_1(T_S\otimes K_S^{-1}) &=& c_1+2(c_1)\ =\ 3c_1, \\
\mathrm{and}\quad c_2(T_S\otimes K_S^{-1}) &=& c_2+c_1(c_1)+(c_1)^2\ =\ c_2+2c_1^2
\eeqa
gives 
\begin{eqnarray} \nonumber
\int_{S}\frac1{e(N^{\vir})}&=&\! 
(-2)^{-P_2}\int_S\frac{(1-c_1+c_2)\big(4-6c_1
+c_2+2c_1^2\big)}{1-3c_1+c_2+2c_1^2} \\ \nonumber
&=&\! (-2)^{-P_2}\!\int_S\!\big(4-10c_1
+5c_2+8c_1^2\big)\big(1+3c_1-c_2-2c_1^2+9c_1^2\big) \\
&=&\! (-2)^{-\chi(K_S^2)}\big(c_2+6c_1^2\big). \label{vert2}
\end{eqnarray}

\noindent\textbf{Remark.}
Finally we note the above calculations can be generalised further to the rank $r$ case, where the ``vertical" component of the $\C^*$-fixed moduli space
$$
S^{[1,1,\cdots,1]}\ \cong\ S
$$
parametrises $\C^*$-fixed torsion sheaves
$$
(\pi^*I_p)\otimes\cO_{rS},\qquad p\in S.
$$
The corresponding Higgs pairs have rank $r$, determinant $K_S^{-r(r-1)/2}$ and $c_2=r+r(r-1)(r-2)(3r-1)/24$. A similar analysis to the above shows that when $g\ge2$ the invariant is
\begin{multline*}
\frac{(-1)^{P_2+\cdots+P_r}}{r^{P_r}}\left(\prod_{i=1}^{r-1}i^i\right)^{\!g-1}\times \\
\left[rc_1(S)^2\left(r-1-2(r-1)\sum_{i=1}^{r-1}\frac1i+2r\left(\sum_{i=1}^{r-1}\frac1i\right)^{\!2\,}\right)+c_2(S)\right].
\end{multline*}

\subsection{One mixed term}
Finally we compute the contribution from the other component with $c_2=3$, namely the nested Hilbert scheme
$$
S^{[2,1]}\ =\ \big\{Z_1\subset Z_0\subset S\colon|Z_0|=2,\,|Z_1|=1\big\}.
$$
The torsion sheaf $\cE$ corresponding to $Z_1\subset Z_0\subset S$ (with ideal sheaves $I_1\supset I_0$ respectively) is an extension
\beq{Beau}
0\To\iota_*I_1\otimes K_S^{-1}\t^{-1}\To\cE\To\iota_*I_0\To0.
\eeq
Here $\iota\colon S\into X$ is the zero section, and multiplication by the tautological section of $\pi^*K_S$ vanishing on $S\subset X$ acts on $\cE$ by the inclusion $\iota_*I_0\subset\iota_*I_1$.
From \eqref{Beau} we see that \emph{in equivariant K-theory}, the class of $R\Hom(\cE,\cE)[1]$ is the same as that of
\begin{multline*}
-R\Hom(\iota_*I_0,\iota_*I_0)-R\Hom(\iota_*I_1,\iota_*I_1)
\\ -R\Hom(\iota_*I_1\otimes K_S^{-1},\iota_*I_0)\t-
R\Hom(\iota_*I_0,\iota_*I_1\otimes K_S^{-1})\t^{-1}.
\end{multline*}
Using $R\Hom_X(\iota_*A,\iota_*B)=R\Hom_S(A,B)\,\oplus\,R\Hom_S(A,B\otimes K_S)\t[-1]$ this becomes
\begin{multline*}
-\langle I_0,I_0\rangle+\langle I_0,I_0\otimes K_S\rangle\t
-\langle I_1,I_1\rangle+\langle I_1,I_1\otimes K_S\rangle\t \\
-\langle I_1,I_0\otimes K_S\rangle\t
+\langle I_1,I_0\otimes K_S^2\rangle\t^2
-\langle I_0,I_1\otimes K_S^{-1}\rangle\t^{-1}
+\langle I_0,I_1\rangle,
\end{multline*}
where $\langle\ \cdot\ ,\ \cdot\ \rangle$ denotes the K-theory class of $R\Hom_S(\ \cdot\ ,\ \cdot\ )$.

To get the $SU(2)$ deformation-obstruction theory of Theorem \ref{final} we remove a copy of $H^*(\cO_S)$ from $\langle I_0,I_0\rangle+\langle I_1,I_1\rangle$ (via trace), and a copy of $H^*(K_S)\t$ from $\langle I_0,I_0\otimes K_S\rangle\t+\langle I_1,I_1\otimes K_S\rangle\t$.
In particular the fixed part of the obstruction theory has the same K-theory class as
\beq{fixi}
-\langle I_0,I_0\rangle-\langle I_1,I_1\rangle\_0
+\langle I_0,I_1\rangle,
\eeq
(where the suffix 0 denotes trace-free),
and the virtual normal bundle has the same K-theory class as
\begin{align}\nonumber
\langle I_0,I_0\otimes K_S\rangle\t
&+\langle I_1,I_1\otimes K_S\rangle\_0\;\t \\
&-\langle I_1,I_0\otimes K_S\rangle\t
+\langle I_1,I_0\otimes K_S^2\rangle\t^2
-\langle I_0,I_1\otimes K_S^{-1}\rangle\t^{-1}. \label{vN}
\end{align}
More generally this works over all of $S^{[2,1]}$ instead of at a single point. Let $\cE$ denote the universal sheaf on $X\times S^{[2,1]}$, let $I_i$ denote the universal ideal sheaves on $S\times S^{[2,1]}$, and let $\pi\_X,\,\pi\_S$ be the projections from these spaces to $S^{[2,1]}$. Then the (Serre dual) obstruction theory of $S^{[2,1]}$ of Theorem \ref{final},
$$
(E\udot)^\vee\ :=\ R\hom_{p\_X}(\EE,\EE)\_{\perp\,}[1]
$$
has a class in $K^{\C^*\!}(S^{[2,1]})$ which is given by the same formula when we use $\langle\ \cdot\ ,\ \cdot\ \rangle$ to denote the K-theory class of $R\hom_{\pi\_S}(\ \cdot\ ,\ \cdot\ )$.
Thus \eqref{fixi} is $(E\udot)^{\vee,\mathrm{\;fix}}$ and \eqref{vN} is $(E\udot)^{\vee,\mathrm{\;mov}}$.
%For simplicity of notation we will write expressions at a single point, but they generalise to the universal family in the same way.

We use the notation
$$
\xymatrix@C=10pt{
S^{[2,1]} \ar@{=}[r]& \Bl_{\Delta_S}(S\times S)
\ar[rr]^-p\ar[dl]^{\pi_1}\ar[dr]^{\pi_2}\ar[d]^{\Bl}&& S^{[2]} \\
S_1 & S_1\times S_2 & S_2}
$$
for $S^{[2,1]}\cong\Bl_{\Delta_S}(S\times S)$. 
Here $S^{[2]}$ parametrises $Z_0$ and $S_1=S$ parametrises $Z_1$. The support of $I_{Z_1}/I_{Z_0}$ is a single point parametrised by $S_2$. We explain and justify these claims as follows.

Consider $S_1$ (respectively $S_2$) as parametrising points $Z_1\subset S$ (respectively $Z_2\subset S$). There are corresponding universal points $\cZ_i\subset S_i\times S$. Pulling them back by $\pi_1\times\id_S,\,\pi_2\times\id_S$ we get (by a small abuse of notation) universal points $$\cZ_1,\ \cZ_2\ \subset\ \Bl_{\Delta_S}(S\times S)\times S.$$
Set
$$
\cZ_0\ =\ \cZ_1\cup\cZ_2\ \subset\ \Bl_{\Delta_S}(S\times S)\times S.
$$
Since $\cZ_1\cap\cZ_2$ is a divisor in both $\cZ_i$ (it is a copy of the exceptional divisor $E\subset\Bl_{\Delta_S}(S\times S)$) it follows that $\cZ_0$ is a \emph{flat} family of subschemes of $S$ over $\Bl_{\Delta_S}(S\times S)$. Together with the subfamily $\cZ_1\subset\cZ_0$ we get a flat family of flags of subschemes of $S$ of lengths 1 and 2 respectively, with a classifying map $\Bl_{\Delta_S}(S\times S)\to S^{[2,1]}$. To construct the inverse map, note $\cZ_0/S^{[2,1]}$ defines a classifying map $S^{[2,1]}\to S^{[2]}$, while the subscheme $\cZ_1\subset\cZ_0/S^{[2,1]}$ defines classifying map from $S^{[2,1]}$ to the total space of the universal subscheme over $S^{[2]}$, which is its double cover $\Bl_{\Delta_S}(S\times S)$.

From $\cZ_1\cap\cZ_2\cong E$ we get the exact sequence
$$
0\To\cO_{\cZ_2}(-E)\To\cO_{\cZ_0}\To\cO_{\cZ_1}\To0
$$
and so
\beq{use}
0\To I_0\To I_1\To\cO_{\cZ_2}(-E)\To0.
\eeq
Using this \eqref{fixi} becomes
\begin{multline*}
-\langle I_1,I_1\rangle+\langle\cO_{\cZ_2}(-E),I_1\rangle+\langle I_1,\cO_{\cZ_2}(-E)\rangle-\langle\cO_{\cZ_2}(-E),\cO_{\cZ_2}(-E)\rangle \\ -\langle I_1,I_1\rangle\_0
+\langle I_1,I_1\rangle-\langle\cO_{\cZ_2}(-E),I_1\rangle,
\end{multline*}
which we can write as
\begin{multline*}
-\langle I_1,I_1\rangle\_0+\langle I_1,\cO_{\cZ_2}(-E)\rangle-\langle\cO_{\cZ_2},\cO_{\cZ_2}\rangle \\
=\ T_{S_1}+\langle I_1,\cO_{\cZ_2}(-E)\rangle-\cO+T_{S_2}-K_{S_2}^{-1},
\end{multline*}
using Serre duality to compute the last term $\ext^2_{\pi\_S}(\cO_{\cZ_2},\cO_{\cZ_2})=K_{S_2}^{-1}$. Finally, writing $I_1=\cO-\cO_{\cZ_1}$ in K-theory, we obtain
\beq{inter}
T_{S_1}+T_{S_2}+\cO(-E)-\langle\cO_{\cZ_1},\cO_{\cZ_2}(-E)\rangle-\cO-K_{S_2}^{-1},
\eeq
where we have suppressed some obvious (flat) pullback maps.

To compute the fourth term we can work on $S_1\times S_2$ (and later pull back to $\Bl_{\Delta_S}(S\times S)$ by using Bondal-Orlov's basechange \cite[Lemma 1.3]{BO}). The universal subschemes $\cZ_1,\,\cZ_2\subset(S_1\times S_2)\times S$ intersect transversally in the small diagonal $\delta_S\subset S^{\times3}$. Therefore $R\hom(\cO_{\cZ_1},\cO_{\cZ_2})\cong\cO_{\delta_S}\otimes\Lambda^2N[-2]$, where $N$ is the normal bundle to $\cZ_1$. Pushing down to $S_1\times S_2$ we get $\cO_{\Delta_S}\otimes\Lambda^2N[-2]$. By a standard Koszul resolution argument the pullback of $\cO_{\Delta_S}$ to $\Bl_{\Delta_S}(S\times S)$ has $h^0=\cO_E$ and $h^{-1}=\cO_E(E)\otimes\Lambda^2N^*$, so using $\Lambda^2N\cong K_S^{-1}$ we get
\beq{get}
\langle\cO_{\cZ_1},\cO_{\cZ_2}(-E)\rangle\ =\ \cO_E\otimes K_S^{-1}(-E)-\cO_E.
\eeq
Plugging this into \eqref{inter} gives
\beq{media}
\big[(E\udot)^{\vee,\mathrm{\;fix}}\big]\ =\ T_{S^{[2,1]}}-\mathrm{Ob}_{S^{[2,1]}}\ =\ T_{S_1}+T_{S_2}-\cO_E\otimes K_S^{-1}(-E)-K_{S_2}^{-1}.
\eeq
There is an obvious exact sequence on $S^{[2,1]}\cong\Bl_{\Delta_S}(S_1\times S_2)$,
$$
0\To T_{\Bl_{\Delta_S}(S_1\times S_2)}\rt{\Bl_*}\Bl^*
T_{S_1\times S_2}\To\frac{\Bl^*N_{\Delta_S}}{\cO_E(E)}\To0.
$$
Since $N_{\Delta_S}\cong T_S$ has rank 2, its quotient by $\cO_E(E)$ is $\Lambda^2N_{\Delta_S}(-E)\cong K_S^{-1}|\_E(-E)$. Therefore $T_{S_1\times S_2}=T_{S^{[2,1]}}+K_S^{-1}|\_E(-E)$ in K-theory, so \eqref{media} simplifies to give
$$
\mathrm{Ob}_{S^{[2,1]}}\ =\ K_{S_2}^{-1},
$$
at least at the level of K-theory. And $S^{[2,1]}$ is smooth, so its virtual cycle is just the top Chern class of its obstruction bundle:
$$
\big[S^{[2,1]}\big]^{\vir}\ =\ c_1(S_2)\cap\big[S^{[2,1]}\big].
$$
In particular
\beq{integral}
\int_{[S^{[2,1]}]^{\vir\ }}\frac1{e(N^{\vir})}\ =\ \int_{S^{[2,1]}}\frac{c_1(S_2)}{e(N^{\vir})}\,.
\eeq
\medskip

Using \eqref{use} we rewrite \eqref{vN} as
\begin{align*}
N^{\vir}\ &=\ \langle I_1,I_1\otimes K_S\rangle\t
-\langle I_1,\cO_{\cZ_2}(-E)\otimes K_S\rangle\t
-\langle\cO_{\cZ_2}(-E),I_1\otimes K_S\rangle\t \\
&+\langle\cO_{\cZ_2}(-E),\cO_{\cZ_2}(-E)\otimes K_S\rangle\t
+\langle I_1,I_1\otimes K_S\rangle\_0\;\t
-\langle I_1,I_1\otimes K_S\rangle\t \\
&+\langle I_1,\cO_{\cZ_2}(-E)\otimes K_S\rangle\t
+\langle I_1,I_1\otimes K_S^2\rangle\t^2
-\langle I_1,\cO_{\cZ_2}(-E)\otimes K_S^2\rangle\t^2 \\
&-\langle I_1,I_1\otimes K_S^{-1}\rangle\t^{-1}
+\langle\cO_{\cZ_2}(-E),I_1\otimes K_S^{-1}\rangle\t^{-1}.
\end{align*}
Cancelling, using Serre duality, and substituting $[I_1]=[\cO]-[\cO_{\cZ_1}]$ gives
\begin{align*}
&\Big(\!-\cO(E)+\langle\cO_{\cZ_1},\cO_{\cZ_2}(-E)\rangle^\vee+\langle\cO_{\cZ_2},\cO_{\cZ_2}\rangle^\vee
+\langle I_1,I_1\rangle_0^\vee\Big)\t \\
&+\Big(\langle I_1,I_1\otimes K_S^2\rangle
-K_{S_2}^2(-E)+\langle\cO_{\cZ_1},\cO_{\cZ_2}(-E)\otimes K_S^2\rangle\Big)\t^2 \\
&\Big(\!-\langle I_1,I_1\otimes K_S^{-1}\rangle
+K_{S_2}^{-2}(E)
-\langle\cO_{\cZ_2}(-E),\cO_{\cZ_1}\otimes K_S^{-1}\rangle\Big)\t^{-1}.
\end{align*}
Substituting in \eqref{get} yields
\begin{align*}
&\Big(\!-\cO(E)-K_S|\_E(2E)+\cO_E(E)+\cO-\Omega_{S_2}+K_{S_2}-\Omega_{S_1}\Big)\t \\
&+\Big(\cO^{\oplus P_2}-T_{S_1}\otimes K_{S_1}^2
-K_{S_2}^2(-E)
+K_S|\_E(-E)-K_S^2|\_E\Big)\t^2 \\
&+\Big(T_{S_1}\otimes K_{S_1}^{-1}-\cO^{\oplus P_2}
+K_{S_2}^{-2}(E)
+K_S^{-1}|\_E(2E)-K_S^{-2}|\_E(E)\Big)\t^{-1}.
\end{align*}
Finally this can be rewritten
\begin{align*}
N^{\vir}\ =\ &-\Big(K_{S_i}(2E)-K_{S_i}(E)+\Omega_{S_1\times S_2}-K_{S_2}\Big)\t \\
&-\Big(T_{S_1}\otimes K_{S_1}^2
+K_{S_2}^2
-K_{S_j}(-E)+K_{S_j}(-2E)\Big)\t^2+(\t^2)^{\oplus P_2} \\
&+\Big(T_{S_1}\otimes K_{S_1}^{-1}+K_{S_2}^{-2}
+K_{S_k}^{-1}(2E)-K_{S_k}^{-1}(E)\Big)\t^{-1}-(\t^{-1})^{\oplus P_2}.
\end{align*}
for any $i,j,k\in\{1,2\}$. Now we use $e(F\otimes\t^w)=(wt)^rc_{1/wt}(F)$ for any complex $F$ of rank $r$, and note that $N^{\vir}$ has rank $-3$, so only terms in $t^0$ contribute to the integral \eqref{integral}. We may therefore set $t=1$ to give
\begin{multline*}
\int_{S^{[2,1]}}\frac{c_1(S_2)\cdot
\big(1+2[E]-c_1(S_i)\big)}
{\big(1+[E]-c_1(S_i)\big)\big(1-c_1(S_2)\big)} \\
\hspace{-15mm}\cdot\frac{\big(1-c_1(S_1)+c_2(S_1)\big)\big(1-c_1(S_2)+c_2(S_2)\big)}{\big(1-\frac12c_1(S_j)-\frac12[E]\big)} \\
\hspace{25mm}\cdot\frac{2^3\big(1-\frac32c_1(S_1)+\frac14c_2(S_1)+\frac12c_1(S_1)^2\big)\big(1-c_1(S_2)\big)}
{(-1)^3\big(1-3c_1(S_1)+c_2(S_1)+2c_1(S_1)^2\big)\big(1-2c_1(S_2)\big)} \\
\cdot\frac{\big(1-\frac12c_1(S_j)-[E]\big)\big(1-c_1(S_k)-[E]\big)(-1)^{P_2}}{\big(1-c_1(S_k)-2[E]\big)2^{P_2}}\,.
\end{multline*}
Multiplying out gives a polynomial in $[E]$. The constant $[E]^0$ term can be easily integrated on $S_1\times S_2$. The $[E]^1$ term is an integral on $E\cong\PP(T_S)\to\Delta_S$ of Chern classes pulled back from $\Delta_S$, so it is zero. The $[E]^{\ge2}$ terms can be evaluated using the Grothendieck formula $[E]\big|_E^2-c_1(S)\cdot[E]\big|\_E+c_2(S)=0$ \emph{on the projective bundle $E\cong\PP(T_S)\to\Delta_S$}. Pushing down to $\Delta_S$ using the projection formula --- and the fact that $[E]\big|_E\in H^2(E)$ pushes down to $-1\in H^0(\Delta_S)$ --- gives
%\begin{align*}
%\int_{S^{[2,1]}}[E]^2c_1^2(S)\ &=\ c_1(S)^2\ =\ 
%\int_{S^{[2,1]}}[E]^3c_1(S), \\
%\int_{S^{[2,1]}}[E]^2c_2(S)\ =\ c_2(S),&\qquad 
%\int_{S^{[2,1]}}[E]^4\ =\ c_1(S)^2-c_2(S).
%\end{align*}
$$
\int_{S^{[2,1]}}[E]^2c_1^2(S)\ =\ -c_1(S)^2\ =\ 
\int_{S^{[2,1]}}[E]^3c_1(S).
$$
After much cancellation the final result is the following.

\begin{prop}
$$
\int_{\big[S^{[2,1]}\big]^{\vir}}\ \frac1{e(N^{\vir})}\ =\ 
(-2)^{-P_2(S)}c_1(S)^2\Big(\!-12c_1(S)^2-2c_2(S)+62\Big).
$$
\end{prop}

\subsection{Comparison with Vafa-Witten prediction}

We can rewrite the previous formula in the notation of \cite{VW} as
\beq{mix21}
(-2)^{-\nu-g+1}(g-1)\big(-24\nu-10g+72\big),
\eeq
where $\nu=\chi(\cO_S)=\frac1{12}(c_1(S)^2+c_2(S))=p_g(S)+1$ and
$g=c_1(S)^2+1$ is the canonical genus.
Adding \eqref{mix21} to the horizontal \eqref{expa} and vertical \eqref{vert2} terms gives the full generating series of monopole branch  contributions up to degree 3:
\begin{multline*}
\sum_{n\in\Z}\VW^{\;\phi\ne0}_{2,K_S,n}(S)q^n\ =\ \\
(-2)^{-\nu-g+1}\Big[1-2(g-1)q+\big(2(g-1)(g-3)+12\nu\big)q^2 \\ 
-\frac43(g-1)\big(18\nu+g^2-8g+9\big)q^3\Big]+O(q^4).
\end{multline*}
\vskip-14mm
\beq{answ} \vspace{2mm}
\eeq
On the other hand, the second term on the first line of \cite[Equation 5.38]{VW} is, in their notation,
\beq{VW538}
\left(\frac14G(q^2)\right)^{\nu/2}\left(\frac{\theta_1}{\eta^2}\right)^{\!1-g},
\eeq
where by \cite[Equation 5.16]{VW},
\beqa
G(q) &=& q^{-1}\prod_{n=1}^\infty(1-q^n)^{-24}, \\
\theta_1(q) &=& \sum_{n\in\Z+\frac12}q^{n^2}\ =\ 
2q^{1/4}\sum_{j=0}^\infty q^{j(j+1)}, \\
\frac1{\eta(q)} &=& q^{-1/24}\prod_{n=1}^\infty(1-q^n)^{-1}. \eeqa
Therefore \eqref{VW538} is $q^{(1-g)/6-\nu}$ times by
$$
2^{-\nu+1-g}\prod_{n=1}^\infty(1-q^{2n})^{-12\nu}(1-q^n)^{2g-2}\left(\sum_{j=0}^\infty q^{j^2+j}\right)^{\!1-g} 
.$$
Ignoring terms of $O(q^4)$ this is
$$
2^{-\nu+1-g}(1-q^2)^{-12\nu}(1-q)^{2g-2}(1-q^2)^{2g-2}(1-q^3)^{2g-2}(1+q^2)^{1-g}
$$
which can be expanded as
\begin{multline*}
\!\!\!2^{-\nu+1-g}\!\left(\!1-(2g-2)q+\frac{(2g-2)(2g-3)}2q^2-\frac{(2g-2)(2g-3)(2g-4)}{3!}q^3\!\right) \\ \qquad\times(1+12\nu q^2)(1-(2g-2)q^2)(1-(2g-2)q^3)(1-(g-1)q^2).
\end{multline*}
Multiplying out, this is
\begin{multline*}
2^{-\nu+1-g}\bigg[1-2(g-1)q+\big(12\nu+2(g-1)(g-3)\big)q^2 \\
-\frac43(g-1)\big(18\nu+g^2-8g+9\big)q^3\bigg]+O(q^4),
\end{multline*}
which agrees perfectly\footnote{Up to the sign $(-1)^{-g-\nu+1}=(-1)^{\vd}$ \eqref{vd} of footnote \ref{fnsign}.} with \eqref{answ} up to $q^3$. We find this completely extraordinary: while Vafa and Witten do briefly consider a nonzero Higgs field in \cite[Equation 2.70]{VW}, it is on a bundle rather than a sheaf. The components we have calculated with in this section consist entirely of non-locally-free sheaves not considered at all in \cite{VW}. The physics reasoning (``cosmic strings") used to derive \cite[Equation 2.70]{VW} is still much more powerful than our tools more than 20 years on.

\begin{rmk}\normalfont The first term of \cite[Equation 5.38]{VW} seems to be the  $\det E=\cO_S$ contribution to the generating series, to which we plan to return in the future.
The other terms in their equation are what G\"ottsche and Kool \cite{GK} conjecture to be the contribution of the instanton branch $\cM_L$ \eqref{cpt1} --- i.e. the virtual signed Euler characteristic \eqref{vircot} of the moduli space of stable sheaves on $S$ with determinant $L$. Using Mochizuki's work \cite{Mo} they prove a universality result similar to those in \cite{EGL, GNY}: this signed virtual Euler characteristic of $\cM_L$ is a universal expression in 7 topological constants on any surface $S$ with $h^1(\cO_S)=0$ and $p_g(S)>0$. Since it is universal, the expression can be calculated on toric surfaces, which they do by torus localisation and computer calculation.\footnote{In fact they do even better, computing the (virtual) $\chi_y$-genus refinement of the virtual Euler characteristic.} The result indeed reproduces the other terms of \cite[Equation 5.38]{VW} for powers $q^{c_2(E)}$ of $q$ up to $c_2(E)=30$.
\end{rmk}

\begin{rmk}\normalfont In \cite{GT} it is shown how to rewrite integrals over the virtual cycle of a nested Hilbert schemes as integrals over a product of ordinary Hilbert schemes, with the insertion of a Carlsson-Okounkov operator. The latter can be re-expressed in terms of Grojnowski-Nakajima operators, so this may give an easier way to compute the integrals of this Section in higher degrees.
\end{rmk}

\bibliographystyle{halphanum}
\bibliography{References}

\end{document}